\documentclass{article}
\usepackage[left=1cm,right=1cm,top=2cm,bottom=1.5cm]{geometry}

\usepackage{bbm}
\usepackage{amssymb}
\usepackage{amsmath}
\usepackage{amsthm}
\usepackage{enumitem}
\usepackage{tocloft} 
\usepackage{hyperref} 

\numberwithin{equation}{section}

\theoremstyle{plain}
\newtheorem{thrm}{Theorem}[section]
\newtheorem{lmm}[thrm]{Lemma}
\newtheorem{crllr}[thrm]{Corollary}
\newtheorem{prpstn}[thrm]{Proposition}

\theoremstyle{definition}
\newtheorem{dfntn}[thrm]{Definition}

\newtheorem{rmrk}[thrm]{Remark}

\theoremstyle{plain}

  \def\xR{\mathbb{R}}
  \def\xN{\mathbb{N}}
 
\def\xCzero{{\rm C}^{0}}
\def\xCone{{\rm C}^{1}}

\def\xHone{{\rm H}^{1}}
\def\xHtwo{{\rm H}^{2}}

\def\xWn#1{{\rm W}^{#1}}

\def\xLone{{\rm L}^{1}}
\def\xLtwo{{\rm L}^{2}} 
\def\xLinfty{{\rm L}^{\infty}} 
\def\xLn#1{{\rm L}^{#1}}

\def\xdif{\,{\rm d}}

\def\xSO{\mathop{\rm SO\,}\nolimits}

\def\O{\Omega}

\def\cF{\mathcal{F}}
\def\cB{\mathcal{B}}
\def\cM{\mathcal{M}}
\def\cS{\mathcal{S}}
\def\cT{\mathcal{T}}
\def\cK{\mathcal{K}}
\def\cD{\mathcal{D}}
\def\cU{\mathcal{U}}

\def\xT{\mathbb{T}}
\def\xD{\mathbb{D}}
\def\xO{\mathbb{O}}
\def\xI{\mathbb{I}}
\def\xJ{\mathbb{J}}
\def\x1{\mathbbm{1}}	

\def\xdiv{\mathop{\rm div\,}\nolimits}
\def\xdist{\mathop{\rm dist\,}\nolimits}
\def\xspan{\mathop{\rm span\,}\nolimits}

\def\xLbeta{{\rm L}^{\beta}}
\def\xL{{\rm L}}
\def\xloc{{\rm loc}}
\def\xweak{{\rm weak}}

\def\eg{\emph{e.g.\/}}

\title {Existence of solutions for an interaction problem between a bubble and a compressible viscous fluid}
\author {
{\sc Fabien Lespagnol$^1$ and Matthieu Hillairet$^1$} \\
\\
$^1$ {\small IMAG, ANGUS, Univ Montpellier, Inria, CNRS, Montpellier, France}}

\begin{document}

\maketitle {}

\begin{abstract} 
In this paper, we study the dynamics of a finite number of spherical bubbles in a compressible fluid within a bounded open domain of $\xR^3$. The fluid-bubble interaction is described by a system of nonlinear partial differential equations (PDEs) and ordinary differential equations (ODEs) coupling the fluid's density, velocity and pressure to the bubble's translational, rotational and radial velocities. We prove the existence of weak solutions for this model until the collision or collapse of the bubbles. The formulation of the fluid–bubble system, along with the techniques used for the existence proof, is inspired by penalization methods developed for fluid–solid interaction. The main contribution of this work is the addition of a radial expansion-contraction mode in the bubble motion, which introduces new nonlinear terms in the momentum equations that need to be treated carefully in the compactness arguments.
\end{abstract}

\tableofcontents

\section{Introduction}
Let $m$ be the number of bubbles immersed in the fluid flow, for each $i \in \{1, \cdots, m\}$ and time $t \in \xR_+$, we denote by  $\cB_i(t)$  and $\cF(t)$ the time-dependent domains occupied by bubble $i$ and the fluid, respectively. We assume that the bubbles move inside a fixed, bounded, smooth domain $\O \subset \xR^3$ and  
$$\cF(t) = \O \setminus \bigcup \limits_{i = 1}^m \overline {\cB_i(t)}.$$
The fluid flow is assumed to satisfy the isentropic compressible Navier-Stokes equations. We denote by $\rho_f$, $u_f$ and $p_f$ its density, velocity and pressure, respectively. The system reads
\begin{align}
&\partial_t \rho_f + \xdiv (\rho_f u_f) = 0, \quad t \in \xR_+^*, \quad x \in \cF(t), \label{eq:intro:rhof_continuity_equation} \\ 
&\partial_t (\rho_f u_f) + \xdiv (\rho_f u_f \otimes  u_f) = \xdiv (\xT_f(u_f, p_f)) - \rho_f g, \quad t \in \xR_+^*, \quad x \in \cF(t), \label{eq:intro:rhofuf_momentum_equation}
\end{align} 
with $- \rho_f g$ the gravitational force. We assume the fluid is Newtonian with stress tensor $\xT_f (u_f, p_f) $ given by the constitutive law 
\begin{equation*} 
\xT_f (u_f, p_f) = 2 \mu_f \big(\xD(u_f) - \frac13 \xdiv(u_f) \xI_3 \big) + \big(\nu_f \xdiv (u_f) - p_f \big) \xI_3, 
\end{equation*}  
where $$\xD(u_f) = \frac12 \big(\nabla u_f + (\nabla u_f)^{\top} \big)$$ denotes the symmetric part of the velocity gradient, $\xI_3$ the identity matrix in $\xR^{3 \times 3}$, and $\mu_f > 0$, $\nu_f \geq 0$ the shear and bulk viscosity coefficients. We further assume that the flow is in the barotropic regime and the relation between $p_f$ and $\rho_f$ is given by the constitutive law
\begin{equation}
p_f = p_f(\rho_f) = a_f \rho_f^{\gamma_f} \; \mbox{with} \; a_f > 0, \; \gamma_f > \frac32. 
\label{eq:intro:pf_definition}
\end{equation} 
We now formulate the equations describing the kinematics and dynamics of the bubbles. This model, introduced in \cite{Hillairet2023MMN, HillairetII2023MMN}, is based on the assumptions that the bubbles are spherical and behave as a compressible viscous fluid with infinite shear viscosity. In particular, denoting by $\cB(x, R)$ the ball of center $x \in \xR^3$ and radius $R \in \xR_+$, for each $i \in \{1, \cdots, m\}$ and time $t \in \xR_+$, there exists $R_i(t) \in \xR_+$ and $ x_i(t) \in \xR^3$ such that  
\begin{equation*}
\cB_i(t) = \cB(x_i(t), R_i(t)).
\end{equation*} 
Setting $\cB_{i, 0} = \cB_i(0)$, $R_{i, 0} = R_i(0)$, $ x_{i, 0} = x_i(0)$, the motion of the bubble $\cB_i$ is then given by 
\begin{equation*}
\cB_i(t) =  \eta_i[t](\cB_{i, 0}) \quad \forall \, t \in \xR_+, 
\end{equation*} 
where
\begin{equation}
\eta_i [t] :  \left\{
\begin{aligned}
& \xR^3 \rightarrow \xR^3 \\ 
& \hat {x} \mapsto  x_i(t) + \frac{R_i(t)}{R_{i, 0}} \xO_i(t) (\hat {x} - x_{i, 0})
\end{aligned} 
\right.  
\label{eq:intro:propagator_definition}
\end{equation} 
and $\xO_i : \xR_+ \rightarrow \xSO(3)$ is a rotation matrix satisfying $\xO_i(0) = \xI_3$. The Lagrangian velocity of bubble $i$ is then given for $(t, \hat {x}) \in \xR_+ \times \cB_{i, 0}$ by 
\begin{equation}
\hat {u}_i (t, \hat {x}) = \partial_t \eta_i[t](\hat {x}) = \dot {x}_i(t) + \frac{\dot {R}_i(t)}{R_{i, 0}} \xO_i(t) (\hat {x} - x_{i, 0}) +  \frac{R_i(t)}{R_{i, 0}}\dot {\xO}_i(t) (\hat {x} - x_{i, 0}).
\label{eq:intro:wub_definition}
\end{equation} 
From \eqref{eq:intro:wub_definition}, we can deduce the following expression for the Eulerian velocity field $ u_i$ of bubble $i$: 
\begin{equation}
 u_i(t,  x) = \left(\partial_t \eta_i[t] \right) \left((\eta_i[t])^{-1}(x)\right) =  V_i(t) +  \omega_i(t) \times (x - x_i(t)) + \frac13 \Lambda_i(t) (x - x_i(t)), \quad t \in \xR_+, \quad x \in \cB_i(t),  
\label{eq:intro:ub_definition}
\end{equation} 
where
\begin{equation}
V_i = \dot {x}_i, \quad \frac13 \Lambda_i = \frac{\dot {R}_i}{R_i}
\label{eq:intro:VbLb_definitions}
\end{equation} 
and $\omega_i$ is the unique vector-valued function satisfying 
\begin{equation}
\quad \dot {\xO}_i \xO_i^{\top} x = \omega_i \times x \quad \forall \, x \in \xR^3.
\label{eq:intro:ob_definition}
\end{equation} 
The quantities $V_i$, $\omega_i$ and $\Lambda_i$ then correspond to the translational, rotational and radial components of the velocity of bubble $i$, respectively. Regarding the bubble dynamics, we assume that the density of each bubble $\rho_i$ is uniform in $\cB_i$ so that $\rho_i$ is expressed in terms of the bubble radius $R_i$ and mass $m_i$ by
\begin{equation}
\rho_i = \left(\frac{R_{i, 0}}{R_i} \right)^3 = \frac{3 m_i}{4 \pi R_i^3}.
\label{eq:intro:mb_definition}
\end{equation}
Moreover, the center of mass of each bubble coincides with its geometric center, that is
\begin{equation}
\frac{1}{m_i} \int_{\cB_i} \rho_i y \xdif  y = \frac{3}{4 \pi R_i^3} \int_{\cB_i} y \xdif y = x_i.
\label{eq:intro:xb_definition}
\end{equation} 
We also introduce the inertia tensor of the bubble $i$ defined by 
\begin{equation}
\xJ_i = \int_{\cB_i} \rho_i \left(|y - x_i|^2 \xI_3 - (y - x_i) \otimes (y - x_i) \right) \xdif y = \frac{2}{5} m_i R_i^2 \xI_3, 
\label{eq:intro:Jb_definition}    
\end{equation}
and a mechanical parameter $K_i : \xR_+ \rightarrow \xR_+$  analogous to $\xJ_i$, quantifying the contribution of radius variations to kinetic energy, defined by
\begin{equation}
K_i = \frac19 \int_{\cB_i} \rho_i |y - x_i|^2 \xdif y = \frac{m_i}{15} R_i^2.
\label{eq:intro:Kb_definition}
\end{equation}
The dynamics equations of the bubble $i \in \{1, \cdots ,m\}$ read
\begin{align}
&  \frac {\xdif}{\xdif t}(m_i  V_i) = -\int_{\partial \cB_i} \xT_f ( u_f, p_f) n \xdif s - \int_{\cB_i} \rho_i  g \xdif y, \label{eq:intro:bubble_translation_dynamic} \\
& \frac{\xdif}{\xdif t} (\xJ_i  \omega_i) = - \int_{\partial \cB_i} (s -  x_i) \times \xT_f(u_f, p_f)  n \xdif s, \label{eq:intro:bubble_rotation_dynamic} \\
& \frac{\xdif}{\xdif t} (K_i \Lambda_i) - \frac13 \left((\xJ_i \omega_i) \cdot \omega_i + K_i |\Lambda_i|^2 \right) = - \frac13 \int_{\partial \cB_i} (s - x_i) \cdot \xT_f (u_f, p_f)  n  \xdif s + \Big(\nu_i \Lambda_i - p_i + \frac{\kappa_i}{R_i} \Big) |\cB_i|,  \label{eq:intro:bubble_expansion_dynamic} 
\end{align} 
where $\nu_i$ is the bulk viscosity coefficient, $\kappa_i$ a surface-tension coefficient and $p_i$ is the bubble pressure satisfying
\begin{equation} 
p_i = p_i(\rho_i) = a_i \rho_i^{\gamma_i} \; \mbox{with} \; a_i > 0, \; \gamma_i > \frac32. 
\label{eq:intro:pb_definition}
\end{equation}
To close the system, we impose the no-slip condition at each fluid–bubble interface and homogeneous Dirichlet boundary conditions on $\partial \Omega$:\begin{align}
 u_f - u_i = 0 \; &\mbox {on} \; \partial \cB_i,  
\label{eq:intro:noslip_itf_cdt} \\ 
u_f = 0 \; &\mbox {on} \; \partial \O. 
\end{align}
We also prescribe the initial data:  
\begin{equation}
\begin{aligned}
& \rho_f(0,  x) = \rho_{f, 0}(x), \quad (\rho_f  u_f)(0,  x) =  q_{f, 0} (x), \quad x \in \cF(0), \\ 
& \rho_i(0, x) = \rho_{i, 0}, \quad u_i(0, x) = V_{i, 0} + \omega_{i, 0} \times (x - x_{i, 0}) + \frac{\Lambda_{i, 0}}{3} (x - x_{i, 0}), \quad x \in \cB_i(0). 
\end{aligned} 
\label{eq:intro:initial_data}
\end{equation} 
\begin{rmrk}
Note that, by substituting equations \eqref{eq:intro:VbLb_definitions}-\eqref{eq:intro:ob_definition} and \eqref{eq:intro:mb_definition}-\eqref{eq:intro:Kb_definition} into \eqref{eq:intro:bubble_translation_dynamic}-\eqref{eq:intro:bubble_expansion_dynamic}, we obtain the dynamical equations governing the evolution of the bubble $\cB_i$ in terms of its center of mass $x_i$, its radius $R_i$ and the rotational velocity $\omega_i$: 
\begin{align*}
&\frac{\xdif}{\xdif t} \left(m_i \dot {x}_i \right) = - \int_{\partial \cB_i} \xT_f(u_f, p_f)  n \xdif  s - m_i  g, \\ 
&\frac{\xdif}{\xdif t} \left(\frac{2 m_i}{5} R_i^2 \omega_i \right) = -\int_{\partial \cB_i} (s - x_i) \times \xT_f (u_f, p_f) n \xdif s, \\ 
&\frac {\xdif}{\xdif t} \left(\frac{m_i}{5} \dot {R}_i R_i \right) - \frac{m_i R_i^2}{3} \left(\frac{2}{5} |\omega_i|^2 + \frac{3}{5} \left|\frac{\dot{R}_i}{R_i}\right|^2 \right) = \\ 
&\hspace{3cm} - \frac13 \int_{\partial \cB_i} (s - x_i) \cdot \xT_f(u_f, p_f)  n \xdif s - \left(3 \nu_i \frac{\dot {R}_i}{R_i} - p_i\left(\frac{3 m_i}{4 \pi R_i^3}\right) + \frac{\kappa_i}{R_i} \right) \frac{4 \pi R_i^3}{3}.  
\end{align*} 
\end{rmrk}

\begin{rmrk}
The second term in the left-hand side of \eqref{eq:intro:bubble_expansion_dynamic} appears due to the time-dependence of $\xJ_i$ and $K_i$, resulting from the evolution of $R_i$. These terms ensure that the left-hand side conserves the kinetic energy. In particular, these additional terms allow to recover formally the following energy identity for the bubble $i$:
\begin{equation*}
\frac{\xdif}{\xdif t} \left(\int_{\cB_i} \rho_i \frac{|u_i|^2}{2} + a_i \frac{\rho_i^{\gamma_i}}{\gamma_i - 1} \xdif y \right) +  \int_{\cB_i} \nu_i|\Lambda_i|^2  \xdif y + \frac{2\pi}{3} \kappa_i \frac{\xdif}{\xdif t} |R_i|^2  = - \int_{\partial \cB_i} \xT_f(u_f, p_f) n \cdot u_i \xdif s.  
\end{equation*}
\end{rmrk}

\subsection{Weak formulation} \label{subsec:weak_formulation}
In this paragraph, we formally derive a weak formulation for the system \eqref{eq:intro:rhof_continuity_equation}-\eqref{eq:intro:initial_data}. We set  
\begin{equation}
\rho = \rho_f \x1_{\cF} + \sum_{i = 1}^m \rho_i \x1_{\cB_i}, \quad  u =  u_f \x1_{\cF} + \sum_{i=1}^m  u_i \x1_{\cB_i}.
\label{eq:intro:rho_u_definitions}
\end{equation}
Let $T > 0$, from \eqref{eq:intro:rhof_continuity_equation}, \eqref{eq:intro:mb_definition} and \eqref{eq:intro:rho_u_definitions}, we obtain for the continuity equation 
\begin{equation}
\int_0^T \int_\O \rho \partial_t \varphi + (\rho  u) \cdot \nabla  \varphi \xdif  y \mathrm  {d} \tau = 0
\label{eq:intro:rho_continuity_equation_weak}
\end{equation}
for all $\varphi \in \cD ((0, T) \times \O)$. To derive a weak formulation of the momentum equation for the fluid and the bubble dynamics, we first define the bubble and fluid regions as follows:
\begin{equation*}
\overline {Q_i} = \{ (t,  x) \; | \; t \in [0, T], \; x \in \overline {\cB_i (t)} \},
\quad
\overline {Q_b} = \bigcup \limits_{i=1}^m \overline {Q_i}, \quad Q_f = \left((0, T) \times \O \right) \backslash \overline {Q_b}.
\end{equation*} 
We also introduce the space of motion and velocity fields corresponding to the kinematics of the bubbles:
\begin{align*}
\cM  &= \left\{ \phi_b \; | \; \exists \, a \in \xR^3, \; R \in \xR, \; \xO \in SO(3), \quad \phi_b(x) = a + R \xO x \quad \forall \, x \in \xR^3 \right\}. \\
\cS &= \left\{\varphi_b  \; | \; \exists \; V, \omega \in \xR^3, \; \Lambda \in \xR, \quad \varphi_b(x) = V + \omega \times x + \frac{\Lambda}{3} x \quad \forall \, x \in \xR^3 \right\}.   
\end{align*}
Note that the spaces $\cM$ and $\cS$ inherit the canonical topology of finite-dimensional real vector spaces. Finally, we set the following test functions space $\cT (Q_b)$ over $(0, T) \times \O$:
\begin{equation*}
\cT (Q_b) = \left\{ 
\begin{aligned} 
&\varphi \in \cD((0, T) \times \O) \; \mbox{such that there exists} \;  \{\varphi_i\}_{i= 1}^m \subset \cD(0, T; \cS) \\
& \mbox{satisfying for} \; i \in \{1, \cdots, m\}, \; \varphi = \varphi_i \; \mbox{in an open neighborhood of} \; \overline {Q_i} 
\end{aligned}
\right\}.
\end{equation*}
Observe that the space of test functions $\cT (Q_b)$ depends on the solution through the domains $\cF(t)$ and $\cB_i(t)$. Combining  \eqref{eq:intro:rhofuf_momentum_equation} with \eqref{eq:intro:bubble_translation_dynamic}–\eqref{eq:intro:bubble_expansion_dynamic} and \eqref{eq:intro:rho_u_definitions}, we obtain the following weak formulation for the momentum equation:
\begin{multline}
\int_0^T \int_\O (\rho u) \cdot \partial_t \varphi + \left(\rho u \otimes u \right) : \xD(\varphi) + p (\rho, \{\x1_{\cB_i} \}_{i = 1}^m) \xdiv (\varphi) \xdif y \xdif \tau \\ 
= \int_0^T \int_\O 2 \mu \left(\xD(u) - \frac13 \xdiv (u) \xI_3 \right) : \left( \xD(\varphi) - \frac13 \xdiv (\varphi) \xI_3 \right) \xdif y \xdif \tau \\ 
+ \int_0^T \int_\O \nu \xdiv (u) \xdiv (\varphi) + \tilde {\kappa}_b \xdiv (\varphi) - \rho g \cdot \varphi \xdif y \xdif \tau
\label{eq:intro:rhou_momentum_equation_weak}
\end{multline} 
for all $ \varphi \in  {\cT}(Q_b)$, where 
\begin{align*}
&\mu = \mu_f \x1_{\cF}, \quad \nu = \nu_f \x1_{\cF} + \sum_{i=1}^m \nu_i \x1_{\cB_i}, \quad \tilde {\kappa}_b = \sum_{i=1}^m \frac{\kappa_i}{R_i} \x1_{\cB_i}, \\
& p (\rho, \{ \x1_{\cB_i} \}_{i = 1}^m)  = a_f \rho^{\gamma_f} \left(1 - \sum_{i=1}^m \x1_{\cB_i} \right) + \sum_{i = 1}^m a_i \rho^{\gamma_i} \x1_{\cB_i}.
\end{align*} 
\begin{rmrk}
For all $V, \omega, a, b \in \xR^3$, $\Lambda \in \xR$ and $x \in \xR^3$,  
\begin{equation*}
V +  \omega \times (x -  a) + \frac{\Lambda}{3} (x - a)  = \tilde {V} + \tilde {\omega} \times (x - b) + \frac{\tilde {\Lambda}}{3} (x - b) 
\end{equation*} 
where 
\begin{equation*} 
\tilde {V} = V - \omega \times (a - b) - \frac{\Lambda}{3} (a - b), \quad \tilde {\omega} = \omega, \quad \tilde {\Lambda} = \Lambda.
\end{equation*}
Consequently, for all $a \in \xR^3$, the space $\cS$ can equivalently be defined by
\begin{equation*}
\cS  = \left\{ \varphi_b \; | \; \exists \;  V, \omega \in \xR^3, \; \Lambda \in \xR, \quad \varphi_b( x) =  V + \omega \times (x - a) + \frac{\Lambda}{3} (x - a) \quad \forall \, x \in \xR^3 \right\}.   
\end{equation*} 
In particular, any test function $\varphi \in \cT (Q_b)$ satisfies
\begin{equation*}
\varphi(t, x) = \tilde {V}_i(t) + \tilde {\omega}_i(t) \times (x - x_i(t)) + \frac{\tilde {\Lambda}_i(t)}{3}(x - x_i(t)) \; \mbox{on a neighbourhood of} \; \overline{Q_i},
\end{equation*} 
with $\tilde {V}_i, \, \tilde {\omega}_i, \, \tilde {\Lambda}_i \in \cD(0, T)$.  
\end{rmrk} 
\begin{dfntn}
Let $\O \subset \xR^3$ be a bounded smooth domain. For each $i \in \{1, \cdots, m\}$, let $\cB_{i, 0} \subset \O$ be a ball of radius $R_{i, 0} \in (0, \infty)$ and center $ x_{i, 0} \in \xR^3$. Let $\rho_{i, 0} \in (0, \infty)$,  $V_{i, 0}, \omega_{i, 0} \in \xR^3$ and $\Lambda_{i, 0} \in \xR$. We prescribe the following initial conditions: 
\begin{align}
& \rho_0 \in \xLn{\gamma_f}(\O), \quad \rho_0 \geq 0 \; \mbox{a.e in} \; \O,  \quad \rho_0 = \rho_{i, 0} \; \mbox{a.e. in} \; \cB_{i, 0}, \label{eq:intro:rho0_regularity}\\   
&q_0 \in  \xLn{\frac{2\gamma_f}{\gamma_f + 1}}(\O), \quad q_0 \x1_{\rho_0 = 0} =  0 \; \mbox{a.e. in} \; \O, \quad \frac{|q_0|^2}{\rho_0} \x1_{\rho_0 > 0} \in \xLn1(\O), \label{eq:intro:q0_regularity} \\  
& u_0 = V_{i, 0} + \omega_{i, 0} \times (x - x_{i, 0}) + \frac{\Lambda_{i, 0}}{3}(x - x_{i, 0}) \; \mbox{a.e. in} \; \cB_{i, 0}, \quad q_0 = \rho_0 u_0 \; \text{a.e. in} \; \cB_{i,0}. \label{eq:intro:u0_compatibility} 
\end{align}
A triplet $(\{\cB_i\}_{i = 1}^m, \rho, u)$ is a bounded energy weak solution of the system \eqref{eq:intro:rhof_continuity_equation}-\eqref{eq:intro:rhofuf_momentum_equation}, \eqref{eq:intro:bubble_translation_dynamic}-\eqref{eq:intro:bubble_expansion_dynamic}, \eqref{eq:intro:noslip_itf_cdt}-\eqref{eq:intro:initial_data}, associated with initial conditions \eqref{eq:intro:rho0_regularity}-\eqref{eq:intro:u0_compatibility} if the following holds
\begin{itemize}[leftmargin=0pt]
\item For all $i \in \{1, \hdots, m\}$ and $t \in [0, T]$, $\cB_i(t)$ is a ball of radius $R_i(t) \in (0, \infty)$ and center $x_i(t) \in \xR^3$.   
\item The velocity $u$ belongs to $\xLtwo(0, T; \xHone(\O))$ and there exists $\eta_i \in \xHone([0, T]; \cM)$ such that
\begin{equation}
\left\{
\begin{aligned}
& \cB_i(t) = \eta_i[t](\cB_{i, 0}) \; \mbox{for all} \; t \in [0, T], \\  
& u(t, x) = \left(\partial_t \eta_i[t] \right) \left((\eta_i[t])^{-1}(x)\right)  \; \mbox{for a.e.} \; (t, x) \in Q_i.
\end{aligned}
\right.  
\label{eq:intro:u_compatibility} 
\end{equation} 
\item The density $\rho$ belongs to $\xLinfty(0, T; \xLn{\gamma_f}(\O))$ and satisfies 
\begin{equation}
\rho(t, x) = (R_{i, 0}/R_i(t))^3 \rho_{i, 0} \; \mbox{for a.e.} \; (t, x) \in Q_i. 
\label{eq:intro:rho_compatibility}
\end{equation} 
\item Provided that they are extended by $0$ in $\xR^3 \setminus \O$, $\rho$ and $u$ satisfy the continuity equation \eqref{eq:intro:rho_continuity_equation_weak} for all $\varphi \in \cD((0,T)\times \xR^3)$ together with its renormalized form   
\begin{equation}
\partial_t b(\rho) + \xdiv (b(\rho)  u) + (b^\prime(\rho) \rho - b(\rho)) \xdiv ( u) = 0 \; \mbox{in} \; \cD^\prime((0, T) \times \xR^3),
\label{eq:intro:rho_renormalized_continuity_equation}
\end{equation} 
for any $b \in \xCzero([0, \infty)) \cap \xCone((0, \infty))$ satisfying 
\begin{equation}
|b^\prime(z)| \leq c z^{-\lambda_0}, \; z \in (0, 1],  \quad |b^\prime(z)| \leq c z^{\lambda_1}, \; z \geq 1, 
\label{eq:intro:b_cdt_renormalized_solution} 
\end{equation} 
with $\lambda_0 < 1, \; -1 < \lambda_1 \leq \frac{s(\gamma_f)}{2} - 1, \; s(\gamma_f) = \frac{5\gamma_f - 3}{3}.$
\item The weak form of the momentum equation \eqref{eq:intro:rhou_momentum_equation_weak} holds for all $\varphi \in  {\cT}(Q_b)$. 
\item The initial conditions are satisfied in the following sense: 
\begin{align}
& \lim \limits_{t \rightarrow 0^+}  \int_\O \rho(t) \varphi \xdif \tau = \int_\O \rho_0 \varphi \xdif \tau \quad \forall \, \varphi \in \cD(\O), \label{eq:intro:init_cond_rho0}\\
& \lim \limits_{t \rightarrow 0^+} \int_\O (\rho u)(t) \varphi \xdif \tau = \int_\O q_0 \varphi \xdif \tau \quad \forall \, \varphi \in \cT_0(Q_b), \label{eq:intro:init_cond_q0}
\end{align} 
with 
\[ \cT_0(Q_b) = \left\{ 
\begin{aligned}
& \varphi \in \cD(\O) \; \mbox{such that there exists} \;  \{\varphi_i\}_{i=1}^m \subset \cS \\
& \mbox{satisfying for} \; i \in \{1, \cdots, m\}, \; \varphi = \varphi_i \; \mbox{in an open neighborhood of} \; \overline {\cB_{i, 0}} 
\end{aligned} 
\right\}. \]

\item The following energy inequality holds for almost every $t \in (0, T)$: 
\begin{equation}
\begin{aligned} 
& \int_\O \left(\frac{1}{2} \rho(t, \cdot) |u(t, .)|^2 + P(\rho(t, \cdot), \{\x1_{\cB_i(t)} \}_{i=1}^m) \right) \xdif y \\ 
& + \int_0^t \int_\O \left(2 \mu |\xD( u) - \frac13 \xdiv (u) \xI_3|^2 + \nu |\xdiv (u)|^2\right) \xdif y \xdif \tau + \sum_{i = 1}^m \frac{2\pi}{3} \kappa_i |R_i(t)|^2  \\ 
& \leq - \int_0^t \int_\O \rho  g \cdot  u \xdif y \xdif \tau +  \int_\O \left(\frac12 \frac{|q_0|^2}{\rho_0} \x1_{\rho_0 > 0} + P(\rho_0, \{\x1_{\cB_{i, 0}} \}_{i=1}^m) \right) \xdif y + \sum_{i = 1}^m \frac{2\pi}{3} \kappa_i |R_{i, 0}|^2, 
\end{aligned} 
\label{eq:intro:energy_inequality}
\end{equation}
where $P(\rho, \{\x1_{\cB_i} \}_{i=1}^m)$ is the potential energy given by  
\begin{equation*}
P(\rho, \{\x1_{\cB_i} \}_{i=1}^m) = \frac{a_f \rho^{\gamma_f}}{\gamma_f - 1} \left(1 - \sum_{i=1}^m \x1_{\cB_i} \right) + \sum_{i=1}^m \frac{a_i \rho^{\gamma_i}}{\gamma_i - 1} \x1_{\cB_i} . 
\end{equation*}  
\end{itemize} 
\label{def:intro:energy_bounded_solution}
\end{dfntn}
We now present the main result of this document. 
\begin{thrm}
Let $\O \subset \xR^3$ be a bounded smooth domain and for $i \in \{1, \cdots , m\}$,  $\cB_{i, 0} \subset \Omega$ a ball of  radius $R_{i, 0} \in (0, \infty)$ and center $ x_{i, 0} \in \xR^3$. Let $g \in \xLinfty((0, T) \times \O)$, assume there exists $\sigma > 0$ such that 
\begin{equation*}
\xdist (\cB_{i, 0}, \partial \O) > 2 \sigma, \quad \xdist(\cB_{i, 0}, \cB_{j, 0}) > 2 \sigma, \quad \forall \,  i, j \in \{1, \cdots, m\}, \quad \, i \neq j. 
\end{equation*}  
Then there exists $T > 0$ such that a bounded energy weak solution to \eqref{eq:intro:rhof_continuity_equation}-\eqref{eq:intro:rhofuf_momentum_equation},  \eqref{eq:intro:bubble_translation_dynamic}-\eqref{eq:intro:bubble_expansion_dynamic},  \eqref{eq:intro:noslip_itf_cdt}-\eqref{eq:intro:initial_data} in the sense of Definition~\ref{def:intro:energy_bounded_solution} exists on $[0, T]$. Moreover, such weak solution satisfies for all $t\in [0, T]$, 
\begin{equation}
R_i(t) > 0, \quad \xdist (\cB_i(t), \partial \O) \geq \sigma, \quad \xdist(\cB_i(t), \cB_j(t)) \geq \sigma, \quad \forall \, i, j \in \{1, \cdots, m\}, \quad i \neq j.
\label{eq:intro:distance_constrained_assumptions} 
\end{equation} 
\label{theo:intro:main_existence_result}
\end{thrm} 

From now one, for the sake of simplicity, we prove Theorem~\ref{theo:intro:main_existence_result} in the case of a single bubble immerged in a linearly viscous compressible fluid $(m = 1)$. Due to the property \eqref{eq:intro:distance_constrained_assumptions} satisfied by the solution, the proof of Theorem~\ref{theo:intro:main_existence_result} can then be easily extended to the case of several bubbles $(m > 1)$.

\begin{rmrk}
Following the terminology of \cite{Feireisl2003ARM}, we shall say that the velocity $u$, respectively the density $\rho$, is compatible with the system $\{\overline {\cB}_i, \eta_i\}_{i =1}^m$, if it satisfies the condition \eqref{eq:intro:u_compatibility}, respectively \eqref{eq:intro:rho_compatibility}, of Definition~\ref{def:intro:energy_bounded_solution}.
\label{remark:intro:compatibility_definition}
\end{rmrk}

\subsection{State of the art and discussion}
The study of fluid-bubble interaction systems is a classical area of research in fluid mechanics. As discussed in the first part of the paper, the model we study is based on two prior works \cite{Hillairet2023MMN,HillairetII2023MMN} dealing with the derivation of averaged models for multiphase bubbly flows. More specifically, we aim to establish the existence of weak solutions for the microscopic model derived in \cite{Hillairet2023MMN}. In this model, the bubble dynamics \eqref{eq:intro:bubble_translation_dynamic}-\eqref{eq:intro:bubble_expansion_dynamic} are derived considering a two-phase compressible Navier--Stokes flow and assuming an infinite shear viscosity in the bubble phase. A common alternative approach to this framework is to the couple incompressible Navier-Stokes equations for the fluid and compressible Euler equations for the bubble, imposing continuity of the stress at the interface. In this formulation, the pressure inside the bubble is usually assumed homogeneous and, in contrast with the model considered here, both viscosity and inertia effects of the internal fluid are neglected. The resulting free-surface relations read
\begin{equation}
\left\{
\begin{aligned} 
&p_i(t) = c_i(t) |\cB_i(t)|^{-\gamma_b}, \quad t > 0, \\ 
&\xT(u_f, p_f) n = - p_i(t) n , \quad t  > 0, \quad x \in \partial \cB_i(t), \\ 
& \partial \cB_i(t) = \eta_f[t](\partial \cB_{i, 0}), \quad t > 0, 
\end{aligned}
\right.
\label{eq:state_art:free_surfarce_condition}
\end{equation} 
where $\eta_f$ denotes the fluid flow map. Notably, contrary to our model, these relations allow tangential discontinuity of the velocity at the interface. Under radial symmetry of the initial data, an important subclass of solutions is described by the Rayleigh–Plesset equations \cite{Rayleigh1917LED, Plesset1949JAM}. This model has been extensively used and analyzed in bubble dynamics, both theoretically and numerically, see \cite{Plesset1977ARF} for further details. Variants of this model also exist, incorporating effects such as thermal evolution or modifications of the gas law \cite{Biro2000JMA}. The existence of weak solutions for a relaxed version of the interface conditions \eqref{eq:state_art:free_surfarce_condition} and spherical bubbles, has been proven in \cite{Burtea2025ARX}. It should be noted that, although the relations \eqref{eq:state_art:free_surfarce_condition} do not impose any geometric constraint on the bubble shape, most of the available results deal with spherical bubbles.

Because of the no-slip conditions at the fluid-bubble interface and the description of the bubble dynamics by a finite set of ODEs, the fluid–bubble interaction problem we study in this paper shares many features with fluid–solid systems. More specifically, we can consider the bubble as an immersed structure with a velocity constrained to a finite number of modes. However, unlike fluid–solid systems, the bubble kinematics include an additional expansion-compression mode associated with variations in the bubble volume. From this perspective, our system is also related to the study of elastic structures, particularly when deformation is constrained to a finite number of modes.

A vast body of literature addresses the existence of weak and strong solutions for the interaction between rigid bodies and a viscous fluid governed by Navier-Stokes equations, for the most part in the incompressible setting \cite{Judakov1974DSS,Serre1987JJA,Hoffmann1996FFM,Desjardins1999ARM,
Desjardins1999CPD,Grandmont2000MMN,Hoffmann2000DC,SanMartin2002ARM,
Feireisl2003NEE,Takahashi2003ADE,Takahashi2004JMF,Gerard‐Varet2014CPA,
chemetov2019JMP} but also in the compressible one \cite{Desjardins1999CPD,Feireisl2003ARM,Guerrero2009AIH,Hieber2015EEC,
Haak2019MN,Roy2020JEE,Kreml2020JDE,Necasova2022JDE}.  

Regarding weak solutions, for a finite number of rigid bodies in a viscous incompressible fluid, their existence in dimension $2$ and $3$, until collision, was first established in \cite{Desjardins1999CPD}, including an additional assumption of sufficiently small initial data in 3D. These results were then generalized to both incompressible and compressible settings in 2D and 3D, assuming this time only the absence of collision, in \cite{Desjardins1999ARM}. In the incompressible case, the previous results were extended for global in time solutions (including potential collisions) in \cite{SanMartin2002ARM} for 2D and \cite{Feireisl2003NEE} for 3D. In the compressible framework, results allowing collisions in 2D and 3D were given in \cite{Feireisl2003ARM}. 

Analogous results for the interaction between a viscous fluid and an elastic structure exist both in the incompressible \cite{Boulakia2006JMF} and compressible framework \cite{Boulakia2005JMP}. The case of an elastic structure described by a finite number of modes have also been treated in the incompressible framework, in \cite {Desjardins2001RMC} for the existence of weak solutions and \cite {Boulakia2012IFB} for strong solutions. 

Our approach to proving the existence of weak solutions for the fluid-bubble system is mostly inspired by techniques from the theory of rigid bodies \cite{SanMartin2002ARM, Feireisl2003ARM}, where a penalization method is employed by replacing the solid with a fluid of very high viscosity. However, unlike the $\xHone$-type penalization used in the aforementioned works, we adopt an $\xLtwo$-type penalization, by penalizing the fluid velocity  within the bubble domain. Similar penalization strategies for rigid body motion have been considered as a numerical approach to deal with fluid-structure interaction problems \cite{Bost2010JNA} or to take account of slip at the fluid-solid interface \cite{Gerard‐Varet2014CPA, Necasova2022JDE}.

The idea of the proof is to introduce an approximate scheme which combines the theory of compressible fluids introduced by P.L Lions \cite{Lions1996MTF} and Feireisl \cite{feireisl2004OUP} to get the strong convergence of the density (renormalized continuity equations, effective viscous flux, artificial pressure) together with an additional penalization term in the bubble domain to approximatively enforce the constraint on the fluid velocity. 

The main novelty of our approach with respect to the rigid-body framework is that the velocity field associated to the motion of the bubbles is not divergence-free. As a result, the volume of the bubble is not constant in time and our analysis must address not only the possibility of collisions, either between bubbles or with the external boundary, but also the potential collapse of the bubbles. In particular, the equations for the kinematics and the dynamics of the bubbles are only defined as long as their radius is strictly positive. Another consequence of the variation of the volume of the bubbles is that contrary to the rigid body framework, the analysis of the non-linear terms,  the advection term and the pressure term, needs to be performed in the whole domain and not only in the fluid part.

The rest of the paper is organized as follows. We introduce three levels of approximation schemes in Section~\ref{sec:approximate_solutions}. In Section~\ref{sec:nonlinear_transport_equation}, we describe some results on the transport equation, which are needed in all the levels of approximation. Section~\ref{subsec:existence_proof_faedo-glrk_approximation} and Section~\ref{subsec:cvg_faedo-glrk_approximation} are dedicated to the construction and convergence analysis of a Faedo-Galerkin scheme associated to a finite dimensional approximation level. We discuss the limiting system associated to the high penalization term in Section~\ref{subsec:cvg_high_penalization_term_momentum_equation} and vanishing viscosity in Section~\ref{subsec:vnsh_dspt_continuity_equation}. Section~\ref{sec:proof_main_result} is devoted to the main part: the existence of weak solutions for the initial fluid-bubble system. 


\section{Approximate solutions} \label{sec:approximate_solutions}
In this section, we state the existence results for the different levels of the approximation scheme. We begin with the $\delta$-level approximation. This step consists of adding an artificial pressure so that   
\begin{equation}
p_\delta = p_\delta(\rho_\delta, \x1_{\cB_\delta}) = a_f \rho_\delta^{\gamma_f} \left(1 - \x1_{\cB_\delta} \right)  + a_b \rho_\delta^{\gamma_b} \x1_{\cB_\delta} + \delta \rho_\delta^\beta
\label{eq:appx_sol:pd_definition}
\end{equation} 
with
\begin{equation}
\delta > 0, \quad \beta \geq \max \{8, 2\gamma_f, 2\gamma_b\}.
\label{eq:appx_sol:delta_beta_conditions}
\end{equation} 
We also introduce the corresponding potential energy:
\begin{equation*}
P_\delta = P_\delta(\rho_\delta, \x1_{\cB_\delta}) = \frac{a_f \rho_\delta^{\gamma_f}}{\gamma_f - 1}  \left(1 - \x1_{\cB_\delta}\right) + \frac{a_b \rho_\delta^{\gamma_b}}{\gamma_b - 1} \x1_{\cB_\delta} + \frac{\delta \rho_\delta^\beta}{\beta - 1}.
\end{equation*} 
The $\delta$-level approximation problem reads as follows: find a triplet $(\cB_\delta, \rho_\delta,  u_\delta)$ such that 
\begin{itemize}[leftmargin=0pt]
\item For all $t \in [0, T]$,  $\cB_\delta(t)$ is a ball of radius $R_\delta(t) \in (0, \infty)$ and center $x_\delta(t) \in \xR^3$. 
\sloppy
\item Their exists $\eta_\delta \in \xHone([0, T]; \cM)$ such that the velocity $u_\delta \in \xLtwo(0, T; \xHone(\O))$ and the density ${\rho_\delta \in \xLinfty(0, T; \xLn{\gamma_f}(\O))}$ are compatible with the system $\{\overline {\cB_\delta}, \eta_\delta\}$ in the sense of Remark~\ref{remark:intro:compatibility_definition}.
\item The pair $(\rho_\delta, u_\delta)$ satisfies the continuity equation \eqref{eq:intro:rho_continuity_equation_weak} for all $\varphi \in \cD((0, T) \times \xR^3)$ provided that it is extended by $(0, 0)$ in $(\xR^3 \setminus \O) \times (\xR^3 \setminus \O)$.
\item Let
\begin{equation}
\overline {Q_\delta} = \{(t, x) \mid t \in [0, T], \; x \in \overline{\cB_\delta(t)}\}.
\label{eq:appx_sol:Qd_definition}
\end{equation}
For every test function $\varphi \in \cT(Q_\delta)$, 
the following weak form of the momentum equation holds:
\begin{equation}
\begin{aligned} 
&- \int_0^T \int_\O \rho_\delta \left(u_\delta \cdot \partial_t \varphi + u_\delta \otimes u_\delta : \xD (\varphi) \right) \xdif y \xdif \tau \\ 
&+ \int_0^T \int_\O 2 \mu_\delta \left(\xD(u_\delta) - \frac13 \xdiv (u_\delta) \xI_3 \right):\left(\xD (\varphi) - \frac13 \xdiv (\varphi) \xI_3\right) \mathrm  {d} y \mathrm  {d}  \tau \\
& + \int_0^T \int_\O (\nu_\delta \xdiv (u_\delta)  - p_\delta(\rho_\delta, \x1_{\cB_\delta})) \xdiv (\varphi) \mathrm  {d} y \mathrm  {d}  \tau = - \int_0^T \int_\O \rho_\delta g \cdot \varphi \xdif y \xdif \tau + \int_0^T \int_{\O} \x1_{\cB_\delta} \frac{\kappa_b}{R_\delta} \xdiv (\varphi) \xdif y \xdif \tau, 
\end{aligned}
\label{eq:appx_sol:rhodud_momentum_equation}
\end{equation} 
where the viscosity coefficients are given by
\begin{equation}
\mu_\delta = (1-\x1_{B_\delta}) \mu_f, \quad \nu_\delta = (1-\x1_{B_\delta}) \nu_f + \x1_{B_\delta} \nu_b.
\label{eq:appx_sol:mud_nud_definition}
\end{equation}
 
\item The initial conditions are satisfied in the sense of \eqref{eq:intro:init_cond_rho0}-\eqref{eq:intro:init_cond_q0}. 
\end{itemize}
A weak solution of problem \eqref{eq:intro:rhof_continuity_equation}-\eqref{eq:intro:noslip_itf_cdt}, in the sense of Definition~\ref{def:intro:energy_bounded_solution}, is obtained as a weak limit of solutions $(\cB_\delta, \rho_\delta, u_\delta)$ to the $\delta$-level approximation as $\delta \rightarrow 0$. The existence result of the approximate system reads: 
\begin{prpstn}
Let $\O \subset \xR^3$ be a bounded smooth domain and $\sigma, M \in (0,\infty).$ Let $\cB_0 \subset \O$ be  a ball of radius $R_0 \in (0, \infty)$ and center $x_0 \in \xR^3, \, \rho_{b, 0} \in (0, \infty), V_0, \, \omega_0 \in \xR^3$ and $\Lambda_0 \in \xR$. Assume that 
\begin{equation}
\begin{aligned} 
&\rho_0 \in \xLbeta(\O), \quad \rho_0 \geq 0 \; \mbox{a.e. in} \; \O, \quad \rho_0 = \rho_{b, 0} \; \mbox{a.e. in} \;  \cB_0, \\
&q_0 \in \xLn{\frac{2\beta}{\beta + 1}}(\O), \quad q_0 \x1_{\{\rho_0 = 0\}} = 0 \; \mbox{a.e. in} \; \O, \quad \frac{|q_0 |^2}{\rho_0} \x1_{\rho_0 > 0} \in \xLone(\O), \\
& u_0 = V_0 + \omega_0 \times (x - x_0) + \frac{\Lambda_0}{3}(x - x_0) \; \mbox{a.e. in} \; \cB_0. 
\end{aligned} 
\label{eq:appx_sol:init_cdt_delta_appx}
\end{equation}   
Let $g \in \xLinfty((0, T) \times \O)$ and $\delta \in (0, \infty)$. Assume that 
\begin{equation*}
\xdist (\cB_0, \partial \O) > 2 \sigma, 
\end{equation*} 
and the initial energy satisfies
\begin{equation}
E_\delta (\rho_0, q_0, \x1_{\cB_0}) = \int_\O \left( \frac12 \frac{|q_0|^2}{\rho_0} \x1_{\rho_0 > 0} + P_\delta(\rho_0, \x1_{\cB_0}) \right) \xdif y + \frac23 \pi \kappa_b |R_0|^2 \leq M.
\label{eq:appx_sol:Ed_definition}
\end{equation} 
Then there exists $T > 0$, depending only on $(\sigma,M)$, such that the $\delta$-level approximation problem admits a bounded energy weak solution $(\cB_\delta, \rho_\delta, u_\delta)$ on $[0,T]$, which satisfies, for almost every $t \in (0, T)$, 
\begin{multline}
E_\delta(\rho_\delta(t), (\rho_\delta u_\delta)(t), \x1_{\cB_\delta}(t)) + \int_0^t \int_\O \left(2\mu_\delta \left|\xD(u_\delta) - \frac13 \xdiv (u_\delta) \xI_3 \right|^2 + \nu_\delta |\xdiv (u_\delta)|^2\right) \xdif y \xdif \tau \\
\leq -\int_0^t \int_\O \rho_\delta g \cdot u_\delta \xdif y \xdif \tau + E_\delta (\rho_0, q_0, \x1_{\cB_0}).
\label{eq:appx_sol:energyd_estimate}
\end{multline}
Moreover, 
\begin{equation*}
R_\delta(t) \geq \frac{R_0}{2}, \quad \xdist(\cB_\delta(t), \partial \O) \geq \frac{ 3\sigma}{2}
\end{equation*} 
and provided that $\rho_\delta$ and $u_\delta$ are extended by $0$ in $\xR^3 \setminus \O$, they satisfy the renormalized continuity equation \eqref{eq:intro:rho_renormalized_continuity_equation} for all $b$ satisfying \eqref{eq:intro:b_cdt_renormalized_solution} with $\lambda_0 < 1, \; -1 < \lambda_1 \leq \frac{\beta}{2} - 1.$
\label{prop:appx_sol:existence_delta_level}
\end{prpstn}
To prove Proposition~\ref{prop:appx_sol:existence_delta_level}, we introduce the $\varepsilon$-level approximation problem, in which a dissipative term is added to the continuity equation:
\begin{equation}
\partial_t \rho_\varepsilon + \xdiv (\rho_\varepsilon u_\varepsilon) = \varepsilon \Delta \rho_\varepsilon \; \mbox{in} \; \cD^\prime((0, T) \times \Omega),  \quad \nabla \rho_\varepsilon \cdot n = 0 \; \mbox {on} \; \partial \O, 
\label{eq:appx_sol:rhoe_continuity_equation}
\end{equation} 
with $\varepsilon > 0$. To preserve the energy estimates, adding the dissipative term in \eqref{eq:appx_sol:rhoe_continuity_equation} also requires including
\begin{equation*}
\varepsilon \int_0^T \int_\O (\nabla u_\varepsilon \nabla \rho_\varepsilon) \varphi \xdif y \xdif \tau
\end{equation*} 
in the momentum equation. The $\varepsilon$-level approximation reads as follows: find a triplet $(\cB_\varepsilon, \rho_\varepsilon, u_\varepsilon)$ such that
\begin{itemize}[leftmargin=0pt]
\item For all $t \in [0, T]$, $\cB_\varepsilon(t)$ is a ball of radius $R_\varepsilon(t) \in (0, \infty)$ and center $x_\varepsilon (t) \in \xR^3$. 
\item There exists $\eta_\varepsilon \in \xHone([0, T]; \cM)$ such that the velocity field $u_\varepsilon \in  \xLtwo(0, T; \xHone(\O))$ and the density $\rho_\varepsilon \in \xLinfty(0, T; \xLbeta(\O))\cap \xLtwo(0, T; \xHone(\O))$ are compatible with the system $\{\overline {\cB_\varepsilon}, \eta_\varepsilon\}$ in the sense of Remark~\ref{remark:intro:compatibility_definition}.
\item The pair $(\rho_\varepsilon, u_\varepsilon)$ satisfies the regularized continuity equation \eqref{eq:appx_sol:rhoe_continuity_equation}.
\item Let 
\begin{equation}
\overline{Q_\varepsilon} = \{(t, x) \; | \; t \in [0, T], \; x \in \overline{\cB_\varepsilon(t)}\}.
\label{eq:appx_sol:Qe_def}
\end{equation}
For every test function $\varphi \in \cT(Q_\varepsilon)$, the following weak form of the momentum equation holds:
\begin{equation}
\begin{aligned}
&-\int_0^T \int_\O \rho_\varepsilon \left(u_\varepsilon \cdot \partial_t \varphi + u_\varepsilon \otimes u_\varepsilon : \xD (\varphi)\right) \xdif y \xdif \tau \\ 
&+ \int_0^T \int_\O 2 \mu_\varepsilon \left(\xD(u_\varepsilon) - \frac13 \xdiv (u_\varepsilon) \xI_3 \right):\left( \xD (\varphi) - \frac13 \xdiv (\varphi) \xI_3\right) \xdif y \xdif \tau \\
&+ \int_0^T \int_\O \left(\nu_\varepsilon \xdiv (u_\varepsilon) - p_\delta(\rho_\varepsilon, \x1_{\cB_\varepsilon})\right) \xdiv (\varphi) \xdif y \xdif \tau + \varepsilon \int_0^T \int_\O \left( \nabla u_\varepsilon \nabla \rho_\varepsilon \right) \cdot \varphi \xdif y \xdif \tau \\
= & - \int_0^T \int_\O \rho_\varepsilon g \cdot \varphi \xdif y \xdif \tau + \int_0^T \int_\O \x1_{\cB_\varepsilon} \frac{\kappa_b}{R_\varepsilon} \xdiv (\varphi) \xdif y \xdif \tau, 
\end{aligned}
\label{eq:appx_sol:rhoeue_momentum_equation}
\end{equation} 
where the viscosity coefficients are given by 
\begin{equation}
\mu_\varepsilon = (1 - \x1_{\cB_\varepsilon}) \mu_f , \quad \nu_\varepsilon = (1 -  \x1_{\cB_\varepsilon}) \nu_f +  \x1_{\cB_\varepsilon} \nu_b.  
\label{eq:appx_sol:mue_nue_definition}
\end{equation}
\item Initial data are satisfied in the sense of \eqref{eq:intro:init_cond_rho0}-\eqref{eq:intro:init_cond_q0}.
\end{itemize}
A weak solution of the $\delta$-level approximation, in the sense of Proposition~\ref{prop:appx_sol:existence_delta_level}, is obtained as a weak limit of solutions $(\cB_\varepsilon, \rho_\varepsilon, u_\varepsilon)$ to the $\varepsilon$-level approximation as $\varepsilon \rightarrow 0$. The existence result of the approximate system reads: 
\begin{prpstn} 
Let $\O \subset \xR^3$ be a bounded smooth domain and $\sigma, \, M, \, \overline {\rho}, \, \underline {\rho} \in (0, \infty).$ Let $\cB_0 \subset \O$ be a ball of radius $R_0 \in (0, \infty)$ and center $x_0 \in \xR^3$, $V_0, \omega_0 \in \xR^3$ and $\Lambda_0 \in \xR$. Assume that 
\begin{align*}
&\rho_0 \in \xWn{1, \infty} (\O), \quad 0 < \underline  {\rho} \leq \rho_0 \leq \overline {\rho}, \quad  q_0 \in \xLtwo(\O), \\
& u_0 = V_0 + \omega_0 \times (x - x_0) + \frac{\Lambda_0}{3}(x - x_0) \; \mbox{a.e. in} \; \cB_0, \quad q_0 = \rho_0 u_0 \; \mbox{a.e. in} \; \cB_0.
\end{align*} 
Let $g \in \xLinfty((0, T) \times \O)$ and $\delta, \varepsilon \in (0, \infty)$. Assume that 
\begin{equation*}
\xdist(\cB_0, \partial \O) > 2 \sigma, 
\end{equation*} 
and that the initial energy satisfies
\begin{equation}
\tilde {E}_\delta (\rho_0, q_0, \x1_{\mathcal {B}_0}) = \int_\O \left(\frac12 \frac{|q_0|^2}{\rho_0}  + \delta \frac{\rho_0^\beta}{\beta - 1} \right) \xdif y + \frac23 \kappa_b \pi |R_0|^2 \leq M.
\label{eq:appx_sol:Ee_definition}
\end{equation} 
Then there exists $T > 0$, depending only on $(\sigma, M, \underline{\rho}, \overline {\rho})$, such that the $\varepsilon$-level approximation problem admits a bounded energy weak solution $(\cB_\varepsilon, \rho_\varepsilon, u_\varepsilon)$, which satisfies, for almost every $t \in (0, T)$,   
\begin{multline}
\tilde {E}_\delta(\rho_\varepsilon(t), (\rho_\varepsilon u_\varepsilon)(t), \x1_{\mathcal {B}_\varepsilon(t)}) + \int_0^t \int_\O \left(2\mu_\varepsilon \left|\xD(u_\varepsilon) - \frac13 \xdiv (u_\varepsilon) \xI_3\right|^2 + \nu_\varepsilon |\xdiv (u_\varepsilon)|^2\right) \xdif y  \xdif \tau \\ 
+ \delta \varepsilon \beta \int_0^t \int_\O (\rho_\varepsilon)^{\beta-2} |\nabla \rho_\varepsilon|^2 \xdif y \xdif \tau \leq a_f \int_0^t \int_\O (1 - \x1_{\cB_\varepsilon}) \rho_\varepsilon^{\gamma_f} \xdiv (u_\varepsilon) \xdif y \xdif \tau \\ 
+ a_b \int_0^t \int_\O \x1_{\cB_\varepsilon} \rho_\varepsilon^{\gamma_b} \xdiv (u_\varepsilon) \xdif y \xdif \tau  - \int_0^t \int_\O \rho_\varepsilon g \cdot u_\varepsilon \xdif y \xdif \tau + \tilde {E}_\delta (\rho_0, q_0, \x1_{\mathcal {B}_0}).
\label{eq:appx_sol:energye_estimate} 
\end{multline}
Moreover, 
\begin{equation*}
R_\varepsilon (t) \geq \frac{R_0}{2}, \quad \xdist(\cB_\varepsilon(t), \partial \O) \geq 2 \sigma \quad \forall \, t \in [0, T], 
\end{equation*} 
and the density $\rho_\varepsilon$ satisfies
\begin{equation}
\partial_t \rho_\varepsilon, \; \Delta \rho_\varepsilon \in \xLn{\frac{5\beta - 3}{4 \beta}} ((0, T) \times \O), \quad \sqrt{\varepsilon}\|\nabla \rho_\varepsilon\|_{\xLtwo((0, T) \times \O)} + \|\rho_\varepsilon\|_{\xLn{\beta + 1}((0,T) \times \O)} \leq c, \label{eq:appx_sol:rhoe_impvd_regularity}
\end{equation} 
where $c$ is a positive constant uniformly independent of $\varepsilon$. 
\label{prop:appx_sol:existence_epsilon_level}
\end{prpstn}
To prove Proposition~\ref{prop:appx_sol:existence_epsilon_level}, we consider an additional level of approximation. In this $n$-level scheme, a penalization term is added to the momentum equation, which is designed to approximate the motion of the bubble. We seek a triplet $(\cB_n, \rho_n, u_n)$ satisfying 
\begin{itemize}[leftmargin=0pt]
\item For all $t \in [0, T]$, $\cB_n (t)$ is a ball of radius $R_n(t) \in (0, \infty)$ and center $x_n(t) \in \xR^3$ such that 
\begin{equation}
\chi_n = \x1_{\cB_n}
\label{eq:appx_sol:chibn_regularity}
\end{equation}
belongs to $\xCzero([0, T]; \xLn{p}(\xR^3))$ for all $1 \leq p < \infty$. 
\item The velocity field $u_n \in \xLtwo(0, T; \xHone(\O))$ and the density function $\rho_n \in \xLinfty(0, T; \xLbeta(\O)) \cap \xLtwo(0, T; \xHone(\O))$ satisfy the regularized continuity equation \eqref{eq:appx_sol:rhoe_continuity_equation}.
\item Let, for all $t \in [0, T]$, $\Pi_n(t) : \xLtwo(\O) \rightarrow \cS$ be defined for any $\varphi \in \xLtwo(\O)$ by
\begin{multline}
(\Pi_n(t) \varphi)(x) = \frac{3}{4 \pi R_n(t)^3} \int_\O \chi_n \varphi \xdif y + \frac{15}{8 \pi R_n(t)^5} \left(\int_\O \chi_n (y - x_n(t)) \times \varphi \xdif y \right) \times (x - x_n(t)), \\ 
+ \frac{5}{4 \pi R_n(t)^5} \left(\int_\O \chi_n (y - x_n(t)) \cdot \varphi \xdif y \right) (x - x_n(t)) \quad  \forall \, x \in \xR^3,
\label{eq:appx_sol:Pibn_definition}
\end{multline} 
with
\begin{equation} 
R_n = \left(\frac{3}{4\pi} \int_{\xR^3} \chi_n \xdif y \right)^\frac13, \quad x_n = \frac{3}{4\pi R_n^3} \left( \int_{\xR^3} \chi_n y \xdif y \right).
\label{eq:appx_sol:Rn_xn_definition} 
\end{equation} 
For every test function $ \varphi \in \cD((0, T) \times \O)$,  the following weak form of the momentum equation holds: 
\begin{equation}
\begin{aligned} 
& - \int_0^T \int_\O \rho_n \left(u_n \cdot \partial_t \varphi + u_n \otimes u_n : \xD (\varphi) \right) \xdif y \xdif \tau \\ 
& + \int_0^T \int_\O 2 \mu_n \left(\xD(u_n) - \frac13 \xdiv (u_n) \xI_3 \right):\left(\xD (\varphi) - \frac13 \xdiv (\varphi) \xI_3\right) \xdif y \xdif \tau \\
& + \int_0^T \int_\O (\nu_n \xdiv (u_n)  - p_\delta(\rho_n, \chi_n)) \xdiv (\varphi) \xdif y \xdif \tau + \varepsilon \int_0^T \int_\O \left(\nabla u_n \nabla \rho_n \right) \cdot \varphi \xdif y \xdif \tau \\ 
& + n \int_0^T \int_\O \chi_n (u_n - \Pi_n  u_n) ( \varphi -  \Pi_n  \varphi) \xdif y \xdif \tau  = - \int_0^T \int_\O \rho_n  g \cdot  \varphi \xdif y \xdif \tau + \int_0^T \int_\O \frac{\kappa_b}{R_n} \chi_n \xdiv (\varphi) \xdif y \xdif \tau,
\end{aligned}
\label{eq:appx_sol:rhonun_momentum_equation}
\end{equation} 
where the viscosity coefficients are given by 
\begin{equation}
\mu_n = (1-\chi_n) \mu_f + n \chi_n, \quad \nu_n = (1 - \chi_n) \nu_f + \chi_n \nu_b.  
\label{eq:appx_sol:mun_nun_definition}
\end{equation}
\item The color function $\chi_n$ satisfies  
\begin{equation}
\left\{
\begin{aligned}
&\partial_t \chi_n + \tilde {\Pi}_n u_n \cdot \nabla \chi_n = 0 \; \mbox{in} \; \cD^\prime ((0, T) \times \xR^3), \\
&\chi_n(0, x) = \x1_{\cB_0}(x)  \; \mbox {for a.e.} \; x \in \xR^3, 
\end{aligned}
\right. 
\label{eq:appx_sol:chibn_transport_equation}
\end{equation} 
where for $t\in [0, T]$, $\tilde {\Pi}_n(t) : \xLtwo(\O) \rightarrow \cS$ is defined for $\varphi \in \xLtwo(\O)$ by  
\begin{equation}
(\tilde {\Pi}_n(t) \varphi) (x) = \frac{3}{4 \pi 	R_n(t)^3} \int_\O \chi_n \varphi \xdif y + \frac{5}{4 \pi R_n(t)^5} \left(\int_\O \chi_n (y - x_n(t)) \cdot \varphi \xdif y \right) (x - x_n(t)) \quad \forall \, x \in \xR^3,   
\label{eq:appx_sol:tPibn_definition}
\end{equation}
and $x_n, \, R_n$ are given by \eqref{eq:appx_sol:Rn_xn_definition}.  
\item The initial conditions are satisfied in the following sense: 
\begin{align}
& \lim \limits_{t \rightarrow 0^+}  \int_\O \rho_n(t) \varphi \xdif \tau = \int_\O \rho_0 \varphi \xdif \tau \quad \forall \, \varphi \in \cD(\O), \label{eq:appx_sol:init_cond_rho0n}\\
& \lim \limits_{t \rightarrow 0^+} \int_\O (\rho_n u_n)(t) \varphi \xdif \tau = \int_\O q_0 \varphi \xdif \tau \quad \forall \, \varphi \in \cD(\O). \label{eq:appx_sol:init_cond_q0n}
\end{align}
\end{itemize}
A weak solution of the $\varepsilon$-level approximation problem in the sense of Proposition~\ref{prop:appx_sol:existence_epsilon_level} will be obtained as a weak limit of the solution $(\cB_n, \rho_n,  u_n)$ of the $n$-level approximation problem as $n \rightarrow \infty$. The existence result of the approximate system reads: 
\begin{prpstn}
Let $\O \subset \xR^3$ be a  bounded smooth domain and $\sigma, M, \overline {\rho}, \underline {\rho} \in (0, \infty)$. Let $\cB_0$ be a ball of radius $R_0 \in (0, \infty)$ and center $x_0 \in \xR^3$. Assume that 
\begin{equation}
\rho_0 \in \xWn{1, \infty} (\O) \; \mbox{satisfies} \; 0 < \underline {\rho} \leq \rho_0 \leq \overline {\rho}, \quad  q_0 \in \xLtwo(\O).
\label{eq:appx_sol:init_cdt_n_appx}
\end{equation} 
Let $g \in \xLinfty((0, T) \times \O)$, $\delta, \varepsilon \in (0, \infty)$ and $n \geq 0$. Assume that 
\begin{equation*}
\xdist (\cB_0, \partial \O) > 2 \sigma, 
\end{equation*} 
and that the initial energy satisfies
\begin{equation*}
\tilde {E}_\delta(\rho_0, q_0) = \int_\O \left(\frac12 \frac{|q_0|^2}{\rho_0} + \delta \frac{\rho_0^\beta}{\beta - 1} \right) \leq M.  
\end{equation*} 
Then there exists $T > 0$, depending only on $(\sigma, M)$, such that the $n$-level approximation problem admits a bounded energy weak solution $(\cB_n, \rho_n, u_n)$, which satisfies, for almost every $t \in (0, T)$,  
\begin{multline}
\tilde {E}_\delta(\rho_n(t), (\rho_n u_n)(t))  + \int_0^t \int_\O \left(2\mu_n \left|\xD(u_n) - \frac13 \xdiv (u_n) \xI_3\right|^2 + \nu_n |\xdiv (u_n)|^2\right) \xdif y \xdif \tau \\ 
+ \delta \varepsilon \int_0^t \int_\O (\rho_n)^{\beta-2} |\nabla \rho_n|^2 \xdif y \xdif \tau  + n \int_0^t \int_\O \chi_n |u_n - \Pi_n u_n|^2  \xdif y \xdif \tau \leq a_f \int_0^t \int_\O (1 - \chi_n) \rho_n^{\gamma_f} \xdiv (u_n) \xdif y \xdif \tau \\ 
+ a_b \int_0^t \int_\O \chi_n \rho_n^{\gamma_b} \xdiv (u_n) \xdif y \xdif \tau - \int_0^t \int_\O \rho_n g \cdot u_n \xdif y \xdif \tau  + \int_0^t \int_\O \chi_n \frac{\kappa_b}{R_n} \xdiv (u_n) \xdif y \xdif \tau  + \tilde {E}_\delta (\rho_0, q_0).
\label{eq:appx_sol:energyn_estimate}
\end{multline}
Moreover, 
\begin{equation*}
R_n(t) \geq \frac{R_0}{2}, \quad \xdist(\cB_n(t), \partial \O) \geq 2\sigma \quad \forall \, t \in [0, T].
\end{equation*} 
\label{prop:appx_sol:existence_n_solution}
\end{prpstn}
\begin{rmrk}
Note that, in addition to the penalization term in the momentum equation \eqref{eq:appx_sol:rhonun_momentum_equation}, we introduce in \eqref{eq:appx_sol:mun_nun_definition} a shear viscosity inside the bubble domain equal to $n$. Formally, in the limit $n \rightarrow \infty$, this corresponds to imposing 
\begin{equation} 
D(u_\varepsilon) = \frac13 \xdiv(u_\varepsilon) \xI_3.
\label{eq:appx_sol:H1_constraint_bubble_motion}
\end{equation}
Since this constraint is already satisfied by the bubble velocity, it is redundant from a purely theoretical standpoint. However, from a convergence perspective, it may be interesting to examine whether this $\xHone$-type
penalization improves the convergence of the penalized model toward the exact one. Nevertheless, as discussed in \cite {Hillairet2023MMN}, the constraint \eqref{eq:appx_sol:H1_constraint_bubble_motion} alone is not sufficient to guarantee the preservation of the spherical shape of the bubbles. In particular, to recover the bubble velocity, it would be necessary to additionnaly penalize a fourth compatible mode by introducing the following term in the momentum equation \eqref{eq:appx_sol:rhonun_momentum_equation}: 
\begin{equation*}
n \int_0^T \int_\O \Pi_{s, n}(t) u_n  \cdot \Pi_{s, n}(t) \varphi \xdif y \xdif \tau, 
\end{equation*}  
where, for $t \in [0,T]$, the operator $\Pi_{s, n}(t): \xLtwo(\O) \rightarrow \xR^3$ is defined for $\varphi \in \xLtwo(\O)$ by 
\begin{equation*}
\Pi_{s, n}(t) (\varphi) = \int_\O \left( (y - x_n(t))  \otimes  (y - x_n(t)) - \frac12 |y - x_n(t)|^2\right) \varphi \xdif y.
\end{equation*} 
\end{rmrk}
To prove Proposition~\ref{prop:appx_sol:existence_n_solution}, we need a final level of approximation, the $N$-level approximation, obtained via a Faedo-Galerkin approximation scheme. Let $\{\psi_i\}_{i \geq 1} \subset {\cD}(\overline {\O})$ be an orthonormal basis of $\xLtwo(\O)$
dense in $\{v \in \xWn{1, p}(\O)\; | \; v = 0 \; \mbox{on} \; \partial \O \}$ for any $1 \leq p \leq  \infty$. We set 
\begin{equation} 
X_N = \xspan \{\{\psi_i\}_{1 \leq i \leq N}\}.
\label{eq:appx_sol:psi_family}
\end{equation} 
Since $X_N$ is a finite-dimensional space, norms on $X_N$ induced by $\xWn{k, p}$ norms, $k \in \xN$, $1 \leq p  \leq \infty$ are equivalent. The task is to find a triplet $(\cB_N, \rho_N, u_N)$ satisfying:
\begin{itemize}[leftmargin=0pt]
\item For all $t \in [0, T]$, $\cB_N (t)$ is a ball of radius $R_N(t) \in (0, \infty)$ and center $x_N(t) \in \xR^3$ such that 
\begin{equation}
\chi_N = \x1_{\cB_N}
\label{eq:appx_sol:chibN_regularity} 
\end{equation} 
belongs to $\xCzero([0, T]; \xLn{p}(\O))$ for all $1 \leq p < \infty$.
\item There exists $\{\alpha_i\}_{i = 1}^N \subset \xCzero([0, T])$ such that $$u_N = \sum \limits_{i = 1}^N \alpha_i \psi_i.$$
\item The velocity field $u_N \in \xCzero([0, T]; X_N)$ and the density function $\rho_N \in \xLtwo(0, T; \xHtwo(\O)) \cap \xHone(0, T; \xLtwo(\O))$ satisfy the regularized continuity equation \eqref{eq:appx_sol:rhoe_continuity_equation}.
\item For all $\varphi \in \cD(0, T; X_N)$, the following weak formulation of the momentum equation holds: 
\begin{equation}
\begin{aligned}
& - \int_0^T \int_\O \rho_N \left(u_N \cdot \partial_t \varphi + u_N \otimes u_N : \xD (\varphi)\right) \xdif y  \xdif \tau \\ 
& + \int_0^T \int_\O 2 \mu_N \left(\xD(u_N) - \frac13 \xdiv (u_N) \xI_3 \right):\left(\xD (\varphi) - \frac13 \xdiv (\varphi) \xI_3\right) \xdif y \xdif \tau \\ 
& + \int_0^T \int_\O \left(\nu_N \xdiv (u_N) - p_\delta(\rho_N, \chi_N) \right) \xdiv (\varphi) \xdif y \xdif \tau + \varepsilon \int_0^T \int_\O \left(\nabla u_N \nabla \rho_N \right) \cdot \varphi \xdif y \xdif \tau \\ 
& + n \int_0^T \int_\O \chi_N (u_N - \Pi_N u_N) \cdot (\varphi - \Pi_N \varphi) \xdif y \xdif \tau  \\ 
= & - \int_0^T \int_\O \rho_N g \cdot \varphi \xdif y \xdif \tau + \int_0^T \int_\O \chi_N \frac{\kappa_b}{R_N} \xdiv (\varphi) \xdif y \xdif \tau. 
\end{aligned}
\label{eq:appx_sol:rhoNuN_momentum_equation}    
\end{equation}
where $\Pi_N$ is defined as in \eqref{eq:appx_sol:Pibn_definition} with $\chi_n$ replaced by $\chi_N$ and the viscosity coefficients are given by 
\begin{equation}
\mu_N = (1-\chi_N) \mu_f + n \chi_N, \quad \nu_N = (1 - \chi_N) \nu_f + \chi_N \nu_b.  
\label{eq:appx_sol:muN_nuN_definition}
\end{equation}
\item The color function $\chi_N $ satisfies 
\begin{equation*}
\left\{ 
\begin{aligned}
&\partial_t \chi_N + \tilde {\Pi}_N u_N \cdot \nabla \chi_N = 0 \; \mbox{in} \; \cD^\prime ((0, T) \times \xR^3), \\
&\chi_N(0, x) = \x1_{\cB_0}(x)  \; \mbox{for a.e.} \; x \in \xR^3, 
\end{aligned}
\right. 
\end{equation*} 
where $\tilde {\Pi}_N$ is defined as in \eqref{eq:appx_sol:tPibn_definition} with $\chi_n$ replaced by $\chi_N$. 
\item The initial conditions are satisfied in the sense \eqref{eq:appx_sol:init_cond_rho0n} --  \eqref{eq:appx_sol:init_cond_q0n}.
\end{itemize}
A weak solution of the $n$-level approximation problem, in the sense of Proposition~\ref{prop:appx_sol:existence_n_solution}, is obtained as a weak limit of solutions $(\cB_N, \rho_N, u_N)$ of the Faedo-Galerkin scheme as $N \rightarrow \infty$. The existence result of the approximate system reads: 
\begin{prpstn}
Let $\O \subset \xR^3 $ be a bounded smooth domain and $\sigma, \,  M, \, \overline {\rho}, \, \underline {\rho} \in (0, \infty)$. Let $\cB_0$ be a ball of radius $R_0 \in (0, \infty)$ and center $x_0 \in \xR^3$. Assume that 
\begin{equation}
\rho_0 \in \xWn{1, \infty}(\O), \quad 0 \leq \underline {\rho} \leq \rho_0 \leq \overline {\rho}, \quad u_0 \in X_N, \quad q_0 = \rho_0 u_0 \; \mbox{a.e. in} \; \O.
\label{appx-sol:rho0N_assumptions}
\end{equation}
Let $g \in \xLinfty((0, T) \times \O)$, $\delta, \varepsilon \in (0, \infty)$, $n \geq 0$ and $N \geq 1$. Assume that
\begin{equation*}
\xdist (\cB_0, \partial \O)  > 2 \sigma, 
\end{equation*} 
and that the initial energy satisfies
\begin{equation*}
\tilde {E}_\delta(\rho_0, q_0) = \int_\O \left(\frac12 \frac{|q_0|^2}{\rho_0}  + \delta \frac{\rho_0^\beta}{\beta - 1} \right) \xdif y \leq M. 
\end{equation*} 
Then, there exists $T > 0$, depending only on $(\sigma, M, \overline {\rho}, \underline {\rho})$, such that the Faedo-Galerkin scheme admits a bounded energy weak solution $(\cB_N, \rho_N, u_N)$, which satisfies, for all $t \in (0, T)$, 
\begin{multline}
\tilde {E}_\delta(\rho_N(t), (\rho_N u_N)(t)) + \int_0^t \int_\O \left(2\mu_N \left|\xD(u_N) - \frac13 \xdiv (u_N) \xI_3 \right|^2 + \nu_N \left|\xdiv (u_N)\right|^2\right) \xdif y \xdif \tau \\ 
+ \delta \varepsilon \int_0^t \int_\O \rho_N^{\beta-2} |\nabla \rho_N|^2 \mathrm{d} y \xdif \tau + n \int_0^t \int_\O \chi_N |u_N - \Pi_N u_N|^2 \xdif y \xdif \tau  \leq a_f \int_0^t \int_\O (1 - \chi_N) \rho_N^{\gamma_f} \xdiv (u_N) \xdif y \xdif \tau \\ 
+ a_b \int_0^t \int_\O \chi_N \rho_N^{\gamma_b} \xdiv (u_N) \xdif y \xdif \tau - \int_0^t\int_\O \rho_N g \cdot u_N \xdif y \xdif \tau + \int_0^t \int_\O \chi_N \frac{\kappa_b}{R_N} \xdiv(u_N)  \xdif y \xdif \tau + \tilde {E}_\delta(\rho_0, q_0)
\label{eq:appx_sol:energyN_estimate} 
\end{multline}
Moreover, 
\begin{equation*}
R_N(t) \geq \frac{R_0}{2}, \quad \xdist(\cB_N(t), \partial \O) \geq 2\sigma \quad \forall \, t \in [0, T].
\end{equation*}
\label{prop:app-sol:existence_Faedo_Galerkin_solution}
\end{prpstn} 

\section{A nonlinear transport equation}\label{sec:nonlinear_transport_equation}
This section is devoted to the nonlinear transport equation
\begin{equation}
\left\{
\begin{aligned}
& \partial_t \chi + \tilde {\Pi} u \cdot \nabla \chi = 0 \; \mbox{in} \; \cD^\prime((0, T) \times \xR^3), \\
& \chi\vert_{t=0} = \x1_{\cB_0} \; \mbox{in} \; \xR^3, 
\end{aligned}
\right. 
\label{eq:smlrt-prop:chib_transport_equation}
\end{equation} 
where $ \cB_0$ is a ball of radius $R_0 \in (0, \infty)$ and center $x_0 \in \xR^3$, and  $\tilde {\Pi}$ is defined analogously to $\tilde {\Pi}_n$ in \eqref{eq:appx_sol:tPibn_definition}, with $\chi_n$ replaced by $\chi$. We first address the regular case. 
\begin{prpstn}[Well-posedness]
Let $T> 0$, $u \in {\xCzero}([0, T]; \cD(\overline {\O}))$. 
\begin{itemize}[leftmargin=0pt]
\item If $T$ and $u$ satisfy 
\begin{equation}
T  \|u\|_{\xLinfty(0, T; \xLtwo(\O))} \leq R_0 \sqrt{\frac{\pi}{5} R_0} , 
\label{eq:smlr-prop:tildeT_normu_cdt}
\end{equation} 
then there exists a unique solution $\chi$ of \eqref{eq:smlrt-prop:chib_transport_equation} and 
\[\chi \in \xLinfty((0, T) \times  \xR^3) \cap \xCzero([0, T]; \xLn{p}(\xR^3)) \] for all $1 \leq p < \infty$.
\item For all $t \in [0, T]$,
\[
\chi(t, \cdot) = \x1_{\tilde {\eta}[t](\cB_0)}, 
\]
where  $\tilde {\eta} \in \xCone([0, T]; \cM)$ is the propagator associated to $\tilde {\Pi} u$. Moreover, for  $x_b, R_b \in \xCone([0, T])$ satisfying
\[ 
\chi = \x1_{\cB(x_b, R_b)}, 
\]
it holds $R_b(t) \geq R_0/2$ for all $t \in [0, T]$. 
\end{itemize} 
\label{prop:smlrt-prop:existence_regular_velocity}
\end{prpstn} 

\begin{proof}[Proof of proposition~\ref{prop:smlrt-prop:existence_regular_velocity}]
We can assume $u \in {\xCzero}([0, T]; {\cD}(\xR^3))$ without loss of generality. Assume that $\chi$ is a solution to \eqref{eq:smlrt-prop:chib_transport_equation}. Then, the equation \eqref{eq:smlrt-prop:chib_transport_equation} can be interpreted as a linear transport equation with prescribed velocity field $\tilde {\Pi} u \in \xCzero([0, T]; \cS)$ satisfying
\begin{equation*}
(\tilde {\Pi} u)(t, x) = V(t) + \frac{\Lambda(t)}{3}(x - x_b(t)) \quad \forall \, (t, x) \in [0, T] \times \xR^3. 
\end{equation*}
By the method of characteristics, we obtain
\begin{equation}
\chi(t, \tilde{\eta}[t](x)) = \x1_{\cB_0}(x) \quad \forall \, (t, x) \in [0, T] \times \xR^3, 
\label{eq:smlrt-prop:chib_characteristics}
\end{equation}
where $\tilde {\eta} \in \xCone([0, T]; \cM)$ is the propagator associated to $\tilde {\Pi} u$. In particular, we have 
\begin{equation}
\chi(t, x) = \x1_{\tilde{\eta}[t](\cB_0)}(x), \quad \tilde{\eta}[t](x) = x_b(t) + \frac{R_b(t)}{R_0} (x - x_0) \quad \forall \, (t, x) \in [0, T] \times \xR^3,
\label{eq:smlrt-prop:teta_definition}
\end{equation}
and $(x_b, R_b)$ is the unique solution to the ODE
\begin{equation*}
\left\{
\begin{aligned}
& \dot {x}_b = {V}, \quad \dot {R}_b = \frac{\Lambda}{3} R_b, \\ 
& x_b(0) = x_0, \quad R_b(0) = R_0. 
\end{aligned}
\right.
\end{equation*}
We then substitute \eqref{eq:smlrt-prop:teta_definition} and \eqref{eq:smlrt-prop:chib_characteristics} in the expression for $\tilde {\Pi} u$ to get for $(t, x) \in [0, T] \times \xR^3$,   
\begin{multline*}
(\tilde {\Pi} u) (t, x) = \frac{3}{4 \pi R_0^3} \int_{\cB_0} \x1_\O(\tilde {\eta}[t](y)) u (t , \tilde {\eta}[t](y)) \xdif y \\ 
+ \frac{5}{4 \pi R_0^4 R_b(t)} \left(\int_{\cB_0} \x1_\O(\tilde {\eta}[t](y)) (y - x_0) \cdot u (t, \tilde {\eta}[t](y)) \xdif y \right) (x - x_b(t)).
\end{multline*} 
It follows that $X_b = (x_b, \, R_b)$ satisfies the nonlinear ODE 
\begin{equation}
\dot {X}_b = f(t, X_b) \; \mbox{with} \; X_b(0) = (x_0, \, R_0), 
\label{eq:smlrt-prop:chib_ODE_definition}
\end{equation} 
where
\begin{equation*}
f : \left\{ 
\begin{aligned} 
&  [0, T] \times (\xR^3 \times (0, \infty)) \to \xR^3 \times (0, \infty) \\ 
& (t, X) \mapsto (f_{x}(t, X), f_{R}(t, X)) 
\end{aligned}
\right.
\end{equation*}
is defined, for $X = (x, R) \in \xR^3 \times (0, \infty)$, by 
\begin{equation}
f_x (t, X) = \frac{3}{4 \pi R_0^3} \displaystyle \int_{\cB_0} \x1_{\O}(\tilde {\eta}_X(y)) u (t, \tilde{\eta}_X(y)) \xdif y, \quad 
f_R (t, X) = \frac{5}{4 \pi R_0^4} \displaystyle \int_{\cB_0} \x1_{\O}(\tilde {\eta}_X(y)) (y - x_0) \cdot u (t, \tilde{\eta}_X(y)) \xdif y, 
\label{eq:smlrt-prop:fx_fR_ODE_definition}
\end{equation} 
and 
\begin{equation}
\tilde {\eta}_X(y) = x + \frac{R}{R_0}(y - x_0) \quad \forall \, y \in \xR^3. 
\label{eq:smlrt-prop:tetaX_definition}
\end{equation} 
By the Cauchy-Lipschitz theorem, there exists a unique maximal $\xCone$ solution to equation~\eqref{eq:smlrt-prop:chib_ODE_definition},  provided that $f$ is continuous in $(t, X)$ and locally Lipschitz in $X$. The continuity is obvious. Let $t \in [0, T]$ and $X_i = (x_i, R_i), \, i \in \{1, 2\}$, we write
\begin{equation*}
f_x(t, X_1) - f_x(t, X_2) = M_1(t) + M_2(t),
\end{equation*}
with 
\begin{align*}
& M_1(t) = \frac{3}{4 \pi R_0^3} \int_{\cB_0} \x1_\O(\tilde {\eta}_{X_1}(y)) \left( u (t,\tilde {\eta}_{X_1}(y)) - u (t,\tilde {\eta}_{X_2}(y)) \right) \xdif y, \\ 
& M_2(t) = \frac{3}{4 \pi R_0^3} \int_{\cB_0} u (t,\tilde {\eta}_{X_2}(y))  \left(\x1_\O(\tilde {\eta}_{X_1}(y)) -  \x1_\O(\tilde {\eta}_{X_2}(y)) \right) \xdif y.
\end{align*} 
On one hand, 
\begin{align}
& |M_1(t)| \notag \\ 
& \leq \frac{3}{4 \pi R_0^3}   \int_{\cB_0} \left| x_1 - x_2 + \frac{R_1 - R_2}{R_0}(y - x_0)\right| \int_0^1 \left| \nabla u \left(t, s \big(x_1 + \frac{R_1}{R_0}(y - x_0)\big) + (1-s) \big(x_2 + \frac{R_2}{R_0}(y - x_0)\big) \right) \right| \xdif s \xdif y \notag \\ 
& \leq \|\nabla u\|_{\xLinfty((0, T) \times \O)}  \left(|x_1 -  x_2| + |R_1 - R_2| \right). \label{eq:smlrt-prop:M1_Lipschitz}
\end{align}
And on the other hand, 
\[
|M_2(t)| \leq  \frac{3}{4 \pi R_0^3} \|u\|_{\xLinfty((0, T) \times \O)} \int_{\cB_0} \left| \x1_\O(\tilde {\eta}_{X_1}(y)) -  \x1_\O(\tilde {\eta}_{X_2}(y)) \right| \xdif y.
\]
Since for all $y \in \cB_0$, $$|\tilde {\eta}_{X_1}(y) -\tilde {\eta}_{X_2}(y)| \leq  |x_1 -  x_2| + |R_1 - R_2| = d_{1, 2},$$ we deduce that the integrand is nonzero only if $\tilde {\eta}_{X_1}(y)$ and $\tilde {\eta}_{X_2}(y)$  lie in a $d_{1, 2}$-neighborhood $\cU$ of $\partial \O$. Hence, 
\begin{align}
|M_2(t)| & \leq \frac{3}{4 \pi R_0^3} \|u\|_{\xLinfty((0, T) \times \O)} \left( \int_{\tilde {\eta}_{X_1}^{-1}(\cU)} \xdif y + \int_{\tilde {\eta}_{X_2}^{-1}(\cU)} \xdif y\right), \notag \\ 
& \leq \frac{3}{4 \pi} \left( \frac{1}{R_1^3} + \frac{1}{R_2^3} \right) \|u\|_{\xLinfty((0, T) \times \O)}  |\cU|,\notag \\ 
& \lesssim \left( \frac{1}{R_1^3} + \frac{1}{R_2^3} \right) \|u\|_{\xLinfty((0, T) \times \O)}  \left( |x_1 -  x_2| + |R_1 - R_2| \right),
\label{eq:smlrt-prop:M2_Lipschitz}
\end{align}
where the hidden constant is independent of $X$. Following similar arguments used for $f_x$, we deduce that  
\begin{equation}
|f_R(t, X_1) - f_R(t, X_2)| \lesssim \left(\|\nabla u \|_{\xLinfty((0, T) \times \O)}  +  \left(\frac{1}{R_1^3} + \frac{1}{R_2^3} \right)  \|u\|_{\xLinfty((0, T) \times \O)} \right) \left(|x_1 -  x_2| + |R_1 - R_2| \right).
\label{eq:smlrt-prop:fR_Lipschitz}
\end{equation}
We therefore obtain that $f$ is locally Lipschitz in $X$, and conclude that equation \eqref{eq:smlrt-prop:chib_ODE_definition} admits a unique $\xCone$ maximal solution $(x_b, R_b)$. Moreover, this solution can be extended to $[0, T]$ provided that $$R_b(t) > 0 \quad \forall \, t \in [0, T].$$ From \eqref{eq:smlrt-prop:chib_ODE_definition} and \eqref{eq:smlrt-prop:fx_fR_ODE_definition}, we get the estimate
\begin{equation*}
\|\dot {R_b}\|_{\xLn1(0, T)} \leq  T \|\dot {R_b}\|_{\xLinfty(0, T)} \leq \sqrt{\frac{5}{\pi}} \frac{T}{2 R_0 \sqrt{R_0}}  \|u \|_{\xLinfty(0, T;\xLtwo(\O))}.  
\end{equation*} 
Since 
\begin{equation*}
R_b(t) = \int_0^t \dot {R_b}(s) \xdif s + R_0 \quad \forall \, t \in [0, T], 
\end{equation*}
it follows that if
\begin{equation}
T \|u\|_{\xLinfty(0, T; \xLtwo(\O))} \leq R_0 \sqrt{\frac{\pi}{5} R_0},
\label{eq:smlrt-prop:tdT_L2L2_condition_colision}
\end{equation}
then
\begin{equation*}
R_b(t) \geq \frac{R_0}{2} > 0 \quad \forall \, t \in  [0, T]. 
\end{equation*}
This completes the proof.
\end{proof}
We now analyse the stability properties of the nonlinear transport equation \eqref{eq:smlrt-prop:chib_transport_equation}.  
\begin{prpstn}[Strong sequential continuity]
Let $\{u_k, \chi_k\}$ be a bounded sequence in $ \xCone([0, T]; {\cD}(\overline {\O})) \times \xLinfty((0, T) \times \O)$ such that 
\begin{align}
& \star \;  u_k \rightarrow  u \; \mbox{in} \;  \xCone([0, T];  {\cD} (\overline {\O})),\label{eq:smlrt-prop:hypothesis_strong_seq_conv_uk} \\ 
& \star \; \partial_t \chi_k + \tilde {\Pi}_k  u_k \cdot \nabla \chi_k
 = 0 \; \mbox {in} \; \cD((0, T) \times \xR^3), \\
& \phantom {\star} \; \; \chi_k \vert_{t = 0} = \x1_{\cB_0} \; \mbox{a.e. in} \; \xR^3, \; \cB_0 = \cB(x_0, R_0), \; x_0 \in \xR^3, \; R_0 \in (0, \infty), \label{eq:smlrt-prop:hypothesis_strong_seq_continuity_2} \\ 
& \star \; \chi_k = \x1_{\cB(x_k, R_k)},  \;    R_k \geq \frac{R_0}{2}.  \label{eq:smlrt-prop:hypothesis_strong_seq_def_chik}
\end{align}
Then, 
\begin{equation*}
\chi_k \rightarrow \chi \; \mbox{weakly--* in} \; \xLinfty((0, T) \times \xR^3), \; \mbox{strongly in} \; \xCzero([0, T]; \xLn{p}(\xR^3)),
\end{equation*}
for all $1 \leq p < \infty$, and $(u, \chi)$ is solution of 
\begin{equation*}
\partial_t \chi + \tilde {\Pi} u \cdot \nabla \chi = 0 \; \mbox{in} \; (0, T) \times \xR^3, \quad \chi \vert_{t = 0} = \x1_{\cB_0} \; \mbox{in} \; \xR^3.
\end{equation*}
Moreover, if $\tilde {\eta}_k$ and  $\tilde {\eta}$ denote the propagators associated to $\tilde {\Pi}_k u_k$ and $\tilde {\Pi} u$, respectively, then the following additional convergences hold:
\begin{align*}
& \tilde {\Pi}_k u_k \rightarrow \tilde {\Pi} u \; \mbox{in} \; {\xCzero} ([0, T]; \cS), \\ 
& \tilde {\eta}_k \rightarrow \tilde {\eta} \; \mbox{in} \;\xCone([0, T]; \cM). 
\end{align*}
In particular, for $x_b, R_b \in \xCone([0, T])$ satisfying 
\begin{equation*}
\chi = \x1_{\cB(x_b, R_b)}, 
\end{equation*}
it holds $R_b(t) \geq R_0/2$ for all $t \in [0, T]$. 
\label{prop:smlrt-prop:strong_sequential_continuity}
\end{prpstn}

\begin{proof}[Proof of Propostion~\ref{prop:smlrt-prop:strong_sequential_continuity}]
This proposition is equivalent to the stability of the family of ODEs 
\begin{equation*}
\dot {X}_k = f_k(t, X_k) = (f_{x, k}(t, X_k),  f_{R, k}(t, X_k)) \; \mbox{with} \; X_k(0) = (x_0, R_0), 
\end{equation*} 
where for $t \in [0, T]$ and $X  \in \xR^3 \times [R_0/2, \infty)$,  
\begin{equation}
\left\{
\begin{aligned} 
& f_{x, k} (t, X) = \frac{3}{4 \pi R_0^3} \displaystyle \int_{\cB_0} \x1_\O(\tilde {\eta}_X[t](y)) u_k (t, \tilde {\eta}_X[t](y)) \xdif y, \\ 
& f_{R, k} (t, X) = \frac{5}{4 \pi R_0^4} \displaystyle \int_{\cB_0} \x1_\O(\tilde {\eta}_X[t](y)) (y - x_0) \cdot u_k (t, \tilde {\eta}_X[t](y)) \xdif y. 
\end{aligned}
\right.  
\label{eq:smlrt-prop:fxk_fRk_ODE_definition}
\end{equation} 
and $\tilde {\eta}_X$ is defined in \eqref{eq:smlrt-prop:tetaX_definition}. From \eqref{eq:smlrt-prop:M1_Lipschitz}, \eqref{eq:smlrt-prop:M2_Lipschitz} and \eqref{eq:smlrt-prop:fR_Lipschitz} we get, for all $k \in \xN$, $t \in [0, T]$,  $X_1, X_2 \in \xR^3 \times [R_0/2, \infty)$,   
\begin{equation*}
|f_k(t, X_1) - f_k(t, X_2)| \lesssim \| u_k\|_{\xLinfty(0, T; \xWn{1, \infty}(\O))} \left|X_1-X_2\right|, 
\end{equation*}
where the hidden constant is independent of $k$, $t$ and $X_1$, $X_2$. By \eqref{eq:smlrt-prop:hypothesis_strong_seq_conv_uk}, the sequence $\{u_k\}$ is uniformly bounded in $\xLinfty(0, T; \xWn{1, \infty}(\O))$. Consequently, the $f_k$ are uniformly Lipschitz with respect to $X$ in $\xR^3 \times [R_0/2, \infty)$. Moreover for $t \in [0, T]$ and $X \in \xR^3 \times [R_0/2, \infty)$, 
\begin{equation*}
|f_k(t, X) - f(t, X) | \lesssim \|u_k(t, \cdot) - u(t, \cdot) \|_{\xLinfty(\O)}.
\end{equation*} 
We deduce that $f_k \rightarrow f$ uniformly on $[0, T] \times \xR^3 \times [R_0/2, \infty)$. The proof then follows by direct application of standard stability results on ODEs.
\end{proof}

\begin{prpstn}[Weak sequential continuity]
Let $\{u_k, \chi_k\} \in \xLtwo(0, T; \xHone(\O)) \times \xLinfty((0, T) \times \xR^3)$ such that 
\begin{align}
\star & \; u_k \rightarrow u \; \mbox{weakly in} \; \xLtwo(0, T; \xHone(\O)), \label{eq:smlrt-prop:hypothesis_weak_seq_continuity_1} \\ 
\star & \; \partial_t \chi_k + \tilde {\Pi}_k u_k \cdot \nabla \chi_k
 = 0 \; \mbox {in} \; \cD((0, T) \times \xR^3), \notag \\
& \; \chi_k \vert_{t = 0} = \x1_{\cB_0} \; \mbox{a.e. in} \; \xR^3, \; \cB_0 = \cB(x_0, R_0), \; x_0 \in \xR^3, \;  R_0 \in (0, \infty), \label{eq:smlrt-prop:hypothesis_weak_seq_continuity_2} \\ 
\star & \; \chi_k = \x1_{\cB(x_k, R_k)}, \;  R_k \geq \frac{R_0}{2}. \label{eq:smlrt-prop:hypothesis_weak_seq_continuity_3}
\end{align}
Then, up to a subsequence, 
\begin{align*}
& \chi_k \rightarrow \chi \; \mbox{weakly--* in} \; \xLinfty((0, T) \times \xR^3), \; \mbox{strongly in} \; \xCzero([0, T]; \xLn{p}_{\xloc}(\xR^3))
\end{align*}
for all $1 \leq p < \infty$ and $(u, \chi)$ is solution of 
\begin{equation*}
\partial_t \chi + \tilde {\Pi} u \cdot \nabla \chi = 0 \; \mbox{in} \; (0, T) \times \xR^3, \quad \chi \vert_{t = 0} = \x1_{\cB_0}.
\end{equation*}
Moreover, if $\tilde {\eta}_k$ and  $\tilde {\eta}$ denote the propagators associated to $\tilde {\Pi}_k u_k$ and $\tilde {\Pi} u$, respectively, then the following additional convergences hold:
\begin{align*}
& \tilde {\Pi}_k u_k \rightarrow \tilde {\Pi} u \; \mbox{weakly in} \; \xLtwo(0, T; \cS), \\
& \tilde {\eta}_k \rightarrow \tilde {\eta} \; \mbox{weakly in} \; \xHone([0, T];  \cM), \; \mbox{strongly in} \; \xCzero([0, T]; \cM). 
\end{align*}
In particular, for $x_b, \, R_b \in \xHone([0, T])$ satisfying
\begin{equation*}
\chi  = \x1_{\cB (x_b, R_b)},
\end{equation*}
it holds $R_b(t) \geq R_0/2$ for all $t \in [0, T]$. 
\label{prop:smlrt-prop:weak_sequential_continuity}
\end{prpstn} 

\begin{proof}[Proof of propostion~\ref{prop:smlrt-prop:weak_sequential_continuity}]
Since the sequence $\{u_k\}$ is bounded in $\xLtwo(0, T; \xHone(\O))$, the sequence $\{\tilde {\Pi}_k u_k\}$ is bounded in $\xLtwo(0, T; \cS)$  and, up to a subsequence, we have 
\begin{equation*}
\tilde {\Pi}_k u_k \rightarrow \overline {\tilde {\Pi} u} \; \mbox{weakly in} \;  \xLtwo(0, T; \cS).
\end{equation*}
Since $\chi_k(0)$ converges strongly in $\xLone(\xR^3)$ (= $\x1_{\cB_0}$ for all $k \in \xN$) and $\{\chi_k\}$ is bounded in $\xLinfty((0, T) \times \O)$, applying Di Perna-Lions theory, we get that $\chi_k$ converges weakly--* in $\xLinfty((0, T) \times \xR^3)$ and strongly in $\xCzero([0, T]; \xLn{p}_{\xloc}(\xR^3))$ for all $ 1 \leq  p  < \infty $. Its limit $\overline {\chi}$ satisfies 
\[ \partial_t \overline {\chi} + \overline {\tilde {\Pi} u} \cdot \nabla \overline {\chi} = 0. \]
It follows from the convergence of $\chi_k$  and \eqref{eq:smlrt-prop:hypothesis_weak_seq_continuity_1} that, up to a subsequence, $\tilde {\Pi}_k u_k \rightarrow \overline {\tilde {\Pi}} u$ weakly in $\xLtwo(0, T; \cS)$, where $\overline {\tilde {\Pi}}$ is defined as in \eqref{eq:appx_sol:tPibn_definition} replacing $\chi_n$ by $\overline {\chi}$. In particular, $(\overline {\chi}, \overline {\tilde {\Pi}} u)$ is a solution to our transport equation. The convergence of the propagators $\tilde {\eta}_k$ then  follows from the convergence of $\chi_k$  easily. 
\end{proof}
A straightforward consequence of Proposition~\ref{prop:smlrt-prop:weak_sequential_continuity} is the following corollary valid for sequences $\{u_k\}$ satisfying $\tilde {\Pi}_k u_k = u_k.$ 
\begin{crllr}
Let $\{u_k\}$ a bounded sequence of $\xLtwo(0, T; \xHone(\O))$ and $\{\cB_k\}$  given for all $t \in [0, T]$ by $\cB_k(t) = \eta_k[t] (\cB_0)$ with $\eta_k \in \xHone([0, T]; \cM)$ and $\cB_0$ a ball of radius $R_0 \in (0, \infty)$ and center $x_0 \in \xR^3$. Assume that for each $k \in \xN$, $u_k$ is compatible with the system $\{\overline {\cB_k}, \eta_k\}$ in the sense of Remark~\ref{remark:intro:compatibility_definition} and satisfy 
\begin{equation*}
u_k \rightarrow u \;\mbox{weakly in} \; \xLtwo(0, T; \xHone(\O)).  
\end{equation*}
Then, passing to a subsequence as the case may be, we have 
\begin{equation*}
\eta_k \rightarrow \eta \; \mbox{weakly in} \; \xHone([0, T]; \cM), \; \mbox{strongly in} \; \xCzero([0, T]; \cM). 
\end{equation*} 
Moreover, defining $\cB(t)$ for all $t \in [0, T]$ by $$\cB(t) = \eta[t](\cB_0), $$ the velocity field $u$ is compatible with the system $\{\overline {\cB}, \eta\}$.
\label{coro:smlrt_prop:stability_compatible_velocities}
\end{crllr}
\begin{proof}
This proposition is a direct consequence of Proposition~\ref{prop:smlrt-prop:weak_sequential_continuity} taking $\chi_k = \x1_{\cB_k}$.
\end{proof}

The last proposition deals with how densities propagate along the characteristics of the bubble motion. Its formulation and its proof is inspired of \cite[Lemme 3.2]{Feireisl2003ARM} in which the case of rigid solids is treated.                                                                                                                                                                                                                                                                                                                                                                                                                                                                                                                                                                                                                                                                                                                                                                                                                                                                                                                                                                                                                                                                                                                                                                                                                                                                                                                                                                                                                                                                                                                                                                                                                                                                                                                                                                                                                                                                                           
\begin{prpstn}
Let $\gamma > 1$, $\rho \in \xLinfty (0, T; \xLn\gamma (\O))$ and $u \in \xLtwo(0, T; \xHone(\O))$ such that $(\rho, u)$ satisfy the continuity equation \eqref{eq:intro:rho_continuity_equation_weak} for all $\varphi \in \cD((0, T) \times \xR^3).$ Let us suppose moreover that $u$ is compatible with the system $\{ \overline {\cB}, \eta \}$ and there exists $\rho_{b, 0} \in (0, \infty)$ such that
\begin{equation*}
\rho_0 = \rho_{b, 0} \; \mbox{in} \; \cB_0. 
\end{equation*} 
Then 
\begin{equation*}
|\cB(t)| \rho (t, \eta[t](x)) = |\cB_0| \rho_{b, 0} \; \mbox{for all} \;  t \in [0, T] \; \mbox{and a.e.} \; x \in \cB_0.
\end{equation*}
\label{lem:smlrt-prop:density_bubble}
\end{prpstn} 

\begin{proof} 
Following the same arguments as in \cite[Lemme 3.2]{Feireisl2003ARM}, we obtain
\begin{equation*}
\int_{\cU} \rho(t, \eta[t](y)) \frac{|\cB(t)|}{|\cB_0|} \xdif y = \int_{\cU} \rho_{b, 0} \xdif y
\end{equation*} 
for any open ball $\cU$. It follows that 
\begin{equation*}
|\cB(t)| \rho(t, \eta[t](x)) =  |\cB_0| \rho_{b, 0} \; \mbox{for all} \; t \in [0, T] \;  \mbox{and a.e.} \; x \in \cB_0.
\end{equation*}
\end{proof}

\section{Existence proofs of the approximate solutions} \label{sec:existence_proof_approximate_solutions}
In this section, we successively prove Proposition~\ref{prop:app-sol:existence_Faedo_Galerkin_solution}, \ref{prop:appx_sol:existence_n_solution}, \ref{prop:appx_sol:existence_epsilon_level} and \ref{prop:appx_sol:existence_delta_level}. The main contribution with respect to fluid-solid interaction is the convergence analysis of the nonlinear terms arising in the bubble domain, specifically the convective term $\rho u \otimes u : \xdiv (\varphi) $ and the pressure term $p(\rho, \x1_\cB) \xdiv (\varphi)$ appearing in the momentum equation \eqref{eq:intro:rhou_momentum_equation_weak}. These terms vanish naturally for rigid-body motion as $\xdiv (\varphi) = 0$ in the solid domain. 

\subsection{Existence proof of the Faedo-Galerkin approximation}\label{subsec:existence_proof_faedo-glrk_approximation}
In this subsection, we construct a solution $(\cB_N, \rho_N, u_N)$ to the problem \eqref{eq:appx_sol:chibN_regularity}-\eqref{eq:appx_sol:muN_nuN_definition}. The strategy consists in reformulating the Galerkin problem as a fixed point problem. More precisely, we will look for $u_N$ as a fixed point of the operator 
\begin{equation*}
\cF_N : 
\left\{
\begin{aligned} 
&B_{Q, T} \mapsto B_{Q, T} \\
& u \mapsto \tilde {u} 
\end{aligned}
\right. ,
\end{equation*} 
where
\begin{equation*}
B_{Q, T} = \{u \in \xCzero([0, T]; X_N), \quad \| u\|_{\xLinfty(0, T; \xLtwo(\O))} \leq Q \},  
\end{equation*} 
for some constants $Q, \, T \in (0, \infty)$. This approach first yields a solution on a time interval $[0, T]$, with $T$ possibly depending on $N$. We will then prove that we can extend this solution to an arbitrary time $\tilde {T}$ independent of $N$. To prove the existence of a fixed point for $\cF_N$, we will apply Schauder's fixed point theorem (for the proof see, \eg, \cite{Evans2022PDE}).
\begin{thrm}[Schauder's fixed point theorem] Let $X$ be a real Banach space and $K \subset X$ a closed, bounded and convex subset. Assume that
\begin{equation*}
\cF : K \rightarrow K
\end{equation*}
is continuous and compact. Then $\cF$ has a fixed point in $K$.
\label{eq:faedo-glrk:schauder-fixed-point}
\end{thrm}

We now prove that for a certain $(Q, T)$, $\cF_N$ satisfies the hypotheses of Theorem~\ref{eq:faedo-glrk:schauder-fixed-point}. According to Proposition~\ref{prop:smlrt-prop:existence_regular_velocity}, if $Q$ and $T$ satisfy the condition
\begin{equation}
TQ \leq R_0 \sqrt{\frac{\pi}{5} R_0}, 
\label{eq:faedo-glrk:QT_cdt}
\end{equation}
then for all $u \in B_{Q, T}$, there exists a unique solution to 
\begin{equation}
\partial_t \chi + \tilde {\Pi} u  \cdot \nabla \chi = 0 \; \mbox {in} \;  \cD^\prime((0, T) \times \xR^3), \quad \chi \vert_{t=0} = \x1_{\cB_0} \; \mbox{in} \; \xR^3,
\label{eq:faedo-glrk:chib_definition} 
\end{equation} 
and
$
\chi \in \xCzero ([0, T]; \xLn{p}_{\xloc}(\xR^3))  
$
for all $1 \leq p < \infty$. Moreover, there exist $x_b, R_b \in \xCone([0, T])$ such that 
\begin{equation}
\chi = \x1_{\cB(x_b, R_b)}, \quad  R_b \geq \frac{R_0}{2}.
\label{eq:faedo-glrk:chib_expression}
\end{equation}
Accordingly, we define
\begin{equation*}
\mu = (1 - \chi) \mu_f + n \chi, \quad \nu = (1 - \chi) \nu_f + \chi \nu_b, \quad p_\delta(\rho, \chi) = (1 - \chi) a_f \rho^{\gamma_f} + \chi a_b \rho^{\gamma_b} + \delta \rho^\beta.
\end{equation*}
From now on, we assume that $Q$ and $T$ satisfies \eqref{eq:faedo-glrk:QT_cdt} and we fix $u \in B_{Q, T}$. Let $\rho$ be the solution to 
\begin{equation}
\partial_t \rho + \xdiv (\rho u) = \varepsilon \Delta \rho \; \mbox{in} \; (0, T) \times \O, \quad \nabla \rho \cdot n = 0 \; \mbox{on} \; \partial \O, \quad \rho\vert_{t = 0} = \rho_0 \; \mbox{in} \; \O, 
\label{eq:faedo-glrk:continuity_equation}
\end{equation} 
with $ 0 < \underline {\rho} \leq \rho_0 \leq \overline {\rho}.$ We recall the following classical maximal regularity result for the parabolic problem, (see, \eg,  \cite{Evans2022PDE}):
\begin{prpstn}
Let $\O$ be a bounded smooth domain, $\underline {\rho}, \overline {\rho} \in (0, \infty)$ and $u \in \xLinfty(0, T; \xWn{1, \infty}(\O))$. Assume that $\rho_0 \in \xWn{1, \infty}(\O)$ satisfies $\underline {\rho} \leq \rho_0 \leq \overline {\rho}$. Then the parabolic problem \eqref{eq:faedo-glrk:continuity_equation} admits a unique solution, which satisfies 
\begin{equation}
\rho \in \xLtwo(0, T; \xHtwo(\O)) \cap \xCzero ([0, T]; \xHone(\O)) \cap \xHone(0, T; \xLtwo(\O))
\label{eq:faedo-glrk:pblc_regularity}
\end{equation} 
and
\begin{equation}
\underline {\rho} \exp \left( - \int_0^t \|\xdiv (u (\tau))\|_{\xLinfty(\O)} \mathrm  {d} \tau \right) \leq \rho(t, x) \leq \overline {\rho} \exp \left( \int_0^t \|\xdiv (u (\tau))\|_{\xLinfty(\O)} \mathrm  {d} \tau \right) 
\label{eq:faedo-glrk:pblc_estimate} 
\end{equation} 
for all $t \in [0, T]$ and a.e. $x \in \O$.
\label{prop:faedo-glrk:parabolic_regularity}
\end{prpstn} 
Since $\rho_0 \in \xWn{1, \infty}(\O)$ and $u \in B_{Q, T}$ in \eqref{eq:faedo-glrk:continuity_equation}, we may apply Proposition~\ref{prop:faedo-glrk:parabolic_regularity}. We conclude that $\rho$ satisfies \eqref{eq:faedo-glrk:pblc_regularity} as well as the estimate \eqref{eq:faedo-glrk:pblc_estimate}. Since all norms are equivalent on the finite-dimensional space $X_N$, there exists a constant $c(N)$, depending only on $N$, such that for all $\varphi \in X_N$, 
\begin{equation*}
c(N)^{-1} \|\varphi\|_{\xWn{1, \infty}(\O)} \leq \|\varphi\|_{\xLtwo(\O)} \leq c(N) \|\varphi\|_{ \xWn{1, \infty}(\O)}. 
\end{equation*}
In particular, we deeduce that 
\begin{equation}
\underline {\rho} \exp(-c(N)QT) \leq \rho(t, x) \leq  \overline {\rho} \exp(c(N)QT) \;  \; \mbox{for all} \; t \in [0, T] \; \mbox{and a.e.} \; x \in \O.
\label{eq:faedo-glrk:rho_bounds}
\end{equation} 
For $\tilde {u} : [0, T] \rightarrow X_N$, we write
\begin{equation*}
\tilde {u}(t) = \sum_{i = 1}^N g_i(t) \psi_i
\end{equation*}
where the functions $\psi_i$ are defined in \eqref{eq:appx_sol:psi_family}. We then consider the following system of ODEs:
\begin{equation}
\mathbb{A}(t) \frac {\xdif}{\xdif t} \tilde {u}(t) + \mathbb{B}(t) \tilde {u}(t) = F(t), \quad \tilde {u}(0) = u_0 = \sum_{i=1}^N \left(\int_\O u_0 \cdot \psi_i \xdif y \right) \psi_i, 
\label{eq:faedo-glrk:ODE_tdu}
\end{equation}
where the matrices $\mathbb{A}(t) = (a_{i, j}(t))_{1 \leq i, j \leq N}$, $\mathbb{B}(t) = (b_{i, j}(t))_{1 \leq i, j \leq N}$ and $F(t) = (f_i(t))_{1 \leq i \leq N}$ are defined by 
 \begin{align}
a_{i, j}(t) =& \int_\O \rho \psi_i \cdot \psi_j \xdif y, \label{eq:faedo-glrk:aij_definition} \\
b_{i, j}(t) =& \int_\O \rho (u \cdot \nabla \psi_j) \cdot \psi_i \xdif y + \int_\O 2 \mu \left(\xD(\psi_i) - \frac13 \xdiv (\psi_i) \right): \left(\xD(\psi_j) - \frac13 \xdiv (\psi_j) \right)  + \nu \xdiv (\psi_i) \xdiv (\psi_j) \xdif y \notag \\ 
& + \int_\O \varepsilon \left(\nabla \psi_j \nabla \rho \right)\cdot \psi_i \xdif y + n \int_\O \chi (\psi_i - \Pi(t) \psi_i) \cdot (\psi_j - \Pi(t) \psi_j) \xdif y, \label{eq:faedo-glrk:bij_definition} \\
f_j(t) =& - \int_\O \rho g \cdot \psi_j \xdif y +  \int_\O  \left(\chi \frac{\kappa_b}{R_b} + p_\delta(\rho, \chi)\right) \xdiv (\psi_j) \xdif y, \label{eq:faedo-glrk:fj_definition} 
\end{align} 
and the projector $\Pi$ is defined as in \eqref{eq:appx_sol:Pibn_definition}, with $\chi_n$  replaced with $\chi$. The positive lower bound of $\rho$ in \eqref{eq:faedo-glrk:rho_bounds} implies $\mathbb{A}(t) \geq \underline {\rho} \xI_N$ in the sense of symmetric positive matrices. In particular, the matrix $\mathbb{A}$ is invertible. Using the regularity of $\chi$ (Proposition~\ref{prop:smlrt-prop:existence_regular_velocity}) and $\rho$ (Proposition~\ref{prop:faedo-glrk:parabolic_regularity}), we deduce that $\mathbb{A}$, $\mathbb{B}$ and $F$ are continuious on $[0, T]$. By the Cauchy-Lipschitz theorem, the system \eqref{eq:faedo-glrk:ODE_tdu} admits a unique solution 
$$
\tilde {u} \in \xCone([0, T]; X_N). 
$$

We now prove that there exist $T, Q \in (0, \infty)$ such that $\cF_N$ defines a continuous and compact mapping from $B_{Q, T}$ into itself. We first show that, for suitable choice of $Q$ and $T$, $B_{Q ,T}$ is stable under $\cF_N$. Using the continuity equation \eqref{eq:faedo-glrk:continuity_equation}, we obtain the following identities by integration by parts: 
\begin{align}
& \int_0^t \int_\O \rho \frac{\xdif}{\xdif t} {\tilde {u}} \cdot \tilde {u}  \xdif y \xdif \tau = - \frac{1}{2} \int_0^t \int_\O \partial_t \rho |\tilde {u}|^2 \xdif y \xdif \tau + \frac12 \int_\O (\rho |\tilde {u}|^2) (t) \xdif y - \frac12 \int_\O  \rho_0 |u_0|^2 \xdif y, \label{eq:faedo-glrk:IPP_rhouu}
 \\ 
& \int_\O \rho (u \cdot \nabla) \tilde {u} \cdot \tilde {u} \xdif y  = -\frac{1}{2} \int_\O \xdiv (\rho u) |\tilde {u}|^2 \xdif y 
\label{eq:faedo-glrk:IPP_rhouguu}
\end{align}
Multiplying \eqref{eq:faedo-glrk:ODE_tdu} by $\tilde {u}$, integrating in time, and combining with \eqref{eq:faedo-glrk:IPP_rhouu}-\eqref{eq:faedo-glrk:IPP_rhouguu}, we obtain the following energy identity: 
\begin{multline}
\int_\O \frac12 \rho |\tilde {u}|^2 \xdif y + \int_0^t \int_\O 2 \mu |\xD(\tilde {u}) - \frac13 \xdiv (\tilde {u})\xI_3|^2 + \nu |\xdiv (\tilde {u})|^2 \xdif y \xdif \tau \\
= - \int_0^t \int_\O \rho g \cdot \tilde {u} \xdif y \xdif \tau + \int_0^t \int_\O \left(\chi \frac{\kappa_b}{R_b} + p_\delta(\rho, \chi)\right) \xdiv (\tilde {u}) \xdif y \xdif \tau + \frac12 \int_\O \rho_0 |u_0|^2 \xdif y
\label{eq:faedo-glrk:nrj_estimate}
\end{multline} 
Using Hölder's inequality, we bound the right-hand side of \eqref{eq:faedo-glrk:nrj_estimate} by 
\begin{multline*}
T \|\rho\|_{\xLinfty((0, T) \times \O)} \|g\|_{\xLinfty(0, T; \xLtwo(\O))}  \| \tilde {u}\|_{\xLinfty(0, T; \xLtwo(\O))} \\
+ \left(\left\|\chi (\kappa_b/R_b) \right\|_{\xLone((0, T) \times \O)} + \left\| p(\rho, \chi) \right\|_{\xLone((0, T) \times \O)} \right) \|\xdiv (\tilde {u})\|_{\xLinfty ((0, T)\times \O)}  + \frac12 \|\sqrt{\rho_0} u_0\|_{\xLtwo(\O)}^2. 
\end{multline*} 
Using the lower bound on $R_b$ from \eqref{eq:faedo-glrk:chib_expression} and the bounds on $\rho$ from \eqref{eq:faedo-glrk:rho_bounds}, we get 
\begin{equation*} 
\left\|\chi (\kappa_b/R_b) \right\|_{\xLone((0, T) \times \O)}  \leq (2\kappa_b/R_0) |\O| T.
\end{equation*}
and 
\begin{equation*}
\begin{aligned}
\left\|p_\delta(\rho, \chi) \right\|_{\xLone((0, T)\times \O)}&\leq a_f \|\rho^{\gamma_f}\|_{\xLone((0, T)\times \O)} + a_b \|\rho^{\gamma_b}\|_{\xLone((0, T)\times \O)} + \delta \|\rho^\beta\|_{\xLone((0, T)\times \O)} , \\
& \leq  \left(a_f \overline {\rho}^{\gamma_f} \exp(\gamma_f c(N) Q T ) + a_b \overline {\rho}^{\gamma_b} \exp( \gamma_b c(N) QT) + \delta \overline {\rho}^{\beta} \exp(\beta c(N) Q T ) \right) |\O| T.
\end{aligned}
\end{equation*}  
Combining these bounds and using \eqref{eq:faedo-glrk:rho_bounds}, we deduce
\begin{multline*}
\|\tilde {u}\|_{\xLinfty(0, T; \xLtwo(\O))}^2 \leq  2 \frac{\overline{\rho}}{\underline{\rho}} \exp(2 c(N) QT) \|g\|_{\xLinfty(0, T; \xLtwo(\O))} Q T \\
+ 2 c(N) \frac{|\O|}{\underline {\rho}}  \exp(c(N)QT) \left(\frac{2 \kappa_b}{R_0} + a_f \overline {\rho}^{\gamma_f} \exp(\gamma_f c(N) Q T)  + a_b \overline {\rho}^{\gamma_b} \exp(\gamma_b c(N) Q T) + \delta \overline {\rho}^{\beta} \exp(\beta c(N) Q T ) \right) Q T 
\\ +  \frac{1}{\underline{\rho}}  \exp(c(N) Q T)   \|\sqrt{\rho_0} u_0\|_{\xLtwo(\O)}^2. 
\end{multline*} 
We now choose 
\begin{equation}
Q \geq \max \left\{1, \frac{4}{\underline {\rho}} \|\sqrt{\rho_0} u_0\|_{\xLtwo(\O)}^2\right\}, 
\label{eq:faedo-glrk:Q_constraint}
\end{equation}
and 
\begin{equation}
T = T(N) = \min \left\{\frac{\log(2)}{2  Q},  \frac{ \underline{\rho}Q}{16 \overline {\rho} \|g\|_{\xLinfty(0, T; \xLtwo(\O))}},  \frac{\log(2)}{\overline{\gamma} c(N) Q}, \frac{ \underline {\rho} Q}{32 c(N) |\O|((\kappa_b/R_0) + a_f \overline {\rho}^{\gamma_f} + a_b \overline {\rho}^{\gamma_b} + \delta \overline {\rho}^\beta)}\right\}, 
\label{eq:faedo-glrk:T_constraint}
\end{equation}
where $\overline {\gamma} = \max \{\gamma_f, \gamma_b, \beta \}$, and therefore $$\cF_N(B_{Q, T}) \subset B_{Q, T}.$$

We now prove that $\cF_N$ is continuous on $B_{Q, T}$. Let $\{u_k\} \subset B_{Q, T}$ such that $u_k \rightarrow u$ in $\xCzero ([0, T]; \cD(\overline {\O}))$ and $\chi_k$ satisfying \eqref{eq:faedo-glrk:chib_definition}. According to Proposition~\ref{prop:smlrt-prop:strong_sequential_continuity}, we have, for all $1 \leq p < \infty$,
\[
\begin{aligned}
& \chi_k \rightarrow \chi \; \mbox{weakly--* in} \; \xLinfty((0, T) \times \xR^3),  \; \mbox{strongly in} \; \xCzero([0, T];  \xLn{p}(\xR^3)), \\
& \tilde {\eta}_k \rightarrow \tilde {\eta} \; \mbox {in} \; \xCone([0, T]; \cM).
\end{aligned}
\]
We deduce that
\[
a_{i,j}^k \rightarrow a_{i, j}, \quad b_{i, j}^k \rightarrow b_{i, j}, \quad f_j^k \rightarrow f_j \; \mbox{in} \; \xCzero([0, T]), 
\]
and so we obtain, from standard stability results on ODE's, 
\begin{equation*}
\cF_N (u_k) = \tilde {u}_k \rightarrow \tilde {u} = \cF_N (u) \; \mbox {in} \; \xCzero([0, T]; X_N). 
\end{equation*}

We now prove that $\cF_N : B_{Q, T} \rightarrow B_{Q, T}$ is a compact operator. Let $\tilde {u}$ be solution to \eqref{eq:faedo-glrk:ODE_tdu}. Then, for all $t \in [0, T]$,  
\begin{equation*}
\left|\frac{\xdif}{\xdif t} \tilde {u}(t)\right| \leq |\mathbb{A}^{-1}(t)|(|\mathbb{B}(t)| |\tilde {u}(t)| +  |F(t)|) \leq  |\mathbb{A}^{-1}(t)|(Q|\mathbb{B}(t)| +  |F(t)|), 
\end{equation*} 
so that 
\begin{equation*}
\sup \limits_{t \in [0, T]} \left\{ |\tilde{u}(t)| + \left|\frac{\xdif}{\xdif t} \tilde {u}(t)\right| \right\} \leq C(N), 
\end{equation*} 
for a constant $c(N)$ depending only on $N$. Consequently, 
\begin{equation*}
\sup \limits_{u \in B_{Q, T}} \|\cF_N (u) \|_{\xCone([0, T]; X_N)} \leq C(N).
\end{equation*} 
By the Arzelà-Ascoli theorem, it follows that $\cF_N : B_{Q, T} \rightarrow B_{Q, T}$ is compact. We can then apply Schauder's fixed point Theorem~\ref{eq:faedo-glrk:schauder-fixed-point} to conclude the existence of a fixed point $u_N \in B_{Q, T}$. At this stage, we have thus established the existence of a solution $u_N$ of the finite dimensional problem on an interval $[0, T]$ with $T$ possibly depending on $N$. We denote by $\rho_N$ and $\chi_N$ the corresponding solutions of the regularized continuity and transport equations on $[0, T] \times \xR^3$. We also introduce $\tilde {\eta}_N \in \xCone([0, T]; \cM)$ the propagator associated to the velocity field $\tilde {\Pi}_N u$ and $x_N, R_N \in \xCone([0, T])$ such that 
\begin{equation}
\tilde {\eta}_N[t] (x) = x_N(t) + \frac{R_N(t)}{R_0}(x - x_0) \quad \forall \, (t, x) \in [0, T] \times \xR^3. 
\label{eq:faedo-glrk:tetaN_definition}
\end{equation}

To extend this solution to an arbitrary time $\tilde {T} > 0$ independent of $N$, we use a standard continuation argument based on energy estimates  uniform with respect to $T(N)$. To obtain such uniform bounds on $u_N$, we multiply the momentum equation \eqref{eq:appx_sol:rhonun_momentum_equation} by $u_N$ and derive the following energy identity, valid on $[0, T(N)]$: 
\begin{multline}
\frac{\xdif}{\xdif t} \tilde {E}_\delta (\rho_N, \rho_N u_N) + \int_\O \left(2 \mu_N |\xD (u_N) - \frac13 \xdiv (u_N)|^2 + \nu_N |\xdiv (u_N)|^2 \right) \xdif y \xdif \tau \\ 
+ \delta \varepsilon \beta \int_\O (\rho_N)^{\beta-2} |\nabla \rho_N|^2 \xdif y \xdif \tau = a_f \int_\O (1 - \chi_N) \rho_N^{\gamma_f} \xdiv (u_N) \xdif y + a_b \int_\O \chi_N \rho_N^{\gamma_b} \xdiv (u_N) \xdif y \\ 
- \int_\O \rho_N g \cdot u_N \xdif y \xdif \tau + \int_\O \chi_N \frac{\kappa_b}{R_N} \xdiv (u_N) \xdif y \xdif \tau, 
\label{eq:cvg_fglrk:energyN_estimate}
\end{multline}
where 
\begin{equation*}
\tilde {E}_\delta(\rho_N, \rho_N u_N)= \int_\O \left(\frac12 \rho_N |u_N|^2 + \delta \frac{\rho_N^\beta}{\beta - 1} \right) \xdif y. 
\end{equation*}
Since $\beta \geq \max\{2\gamma_f, 2 \gamma_b\}$, the terms on the right-hand side of \eqref{eq:cvg_fglrk:energyN_estimate} can be estimated using Hölder's inequality and the embeddings $$\xLbeta(\O)  \hookrightarrow \xLn{2\gamma_f}(\O), \quad \xLbeta(\O)  \hookrightarrow  \xLn{2\gamma_b}(\O).$$ More precisely, they are bounded by
\begin{multline*}
a_f |\O|^{(\beta - 2 \gamma_f)/2\beta} \|\rho_N\|_{\xLbeta(\O)}^{\gamma_f}\|\xdiv (u_N)\|_{\xLtwo(\O)} + a_b |\O|^{(\beta - 2 \gamma_b)/2\beta} \|\rho_N\|_{\xLbeta(\O)}^{\gamma_b}\|\xdiv (u_N)\|_{\xLtwo(\O)} \\ 
+ \|\rho_N |u_N|^2\|_{\xLone(\O)}^\frac12 \|\rho_N\|_{\xLbeta(\O)}^\frac12\|g\|_{\xLn{\frac{2\beta}{\beta -1}}(\O)} +  \kappa_b \|\xdiv (u_N)\|_{\xLtwo(\O)} \|\chi_N/R_N\|_{\xLtwo(\O)}.
\end{multline*} 
Applying Young's inequality, we get 
\begin{multline*}
\frac{a_f^2}{4 \theta} |\O|^{(\beta - 2 \gamma_f)/\beta} \|\rho_N\|_{\xLbeta(\O)}^{2 \gamma_f} + \theta \|\xdiv (u_N)\|_{\xLtwo(\O)}^2 +  \frac{a_b^2}{4\theta} |\O|^{(\beta - 2 \gamma_b)/\beta} \|\rho_N\|_{\xLbeta(\O)}^{2 \gamma_b} + \theta \|\xdiv (u_N)\|_{\xLtwo(\O)}^2 \\
\leq a_f^{\frac{2\beta}{\beta-2\gamma_f}} \frac{\beta - 2 \gamma_f}{4 \theta \beta} \left(\frac{6\gamma_f(\beta-1)}{4 \theta \delta \beta}\right)^\frac{2\gamma_f}{\beta - 2\gamma_f} |\O| + \frac{\delta}{3 (\beta - 1)} \|\rho_N\|_{\xLbeta(\O)}^\beta  \\ 
+ a_b^{\frac{2\beta}{\beta-2\gamma_b}}  \frac{\beta - 2 \gamma_b}{4 \theta \beta} \left(\frac{6\gamma_b(\beta-1)}{4 \theta \delta \beta}\right)^\frac{2\gamma_b}{\beta - 2\gamma_b} |\O| + \frac{\delta}{3(\beta - 1)} \|\rho_N\|_{\xLbeta(\O)}^\beta + 2 \theta \|\xdiv (u_N)\|_{\xLtwo(\O)}^2
\end{multline*} 
and 
 \begin{multline*}
\frac12 \|\rho |u_N|^2\|_{\xLone(\O)} + \frac12 \|\rho_N\|_{\xLbeta(\O)}\|g\|_{\xLn{\frac{2\beta}{\beta -1}}(\O)}^2 + \theta  \|\xdiv (u_N)\|_{\xLtwo(\O)}^2 + \frac{\kappa_b}{4\theta}\|(\chi_N/R_N) \|_{\xLtwo(\O)}^2  \\
\leq \frac12 \|\rho |u_N|^2\|_{\xLone(\O)} + \frac{\delta}{3(\beta - 1)} \|\rho_N\|_{\xLbeta(\O)}^\beta + \frac{\beta - 1}{2\beta}\left(\frac{3(\beta - 1)}{2\delta\beta}\right)^\frac{1}{\beta - 1}\|g\|_{\xLn{\frac{2\beta}{\beta - 1}}(\O)}^{\frac{2\beta}{\beta - 1}} \\
+ \frac{\kappa_b^2}{4\theta}\|\chi_N/R_N\|_{\xLtwo(\O)}^2 + \theta  \|\xdiv (u_N)\|_{\xLtwo(\O)}^2, 
\end{multline*} 
where $\theta > 0$ is a positive constant to be fixed later. Finally, since $R_N \geq R_0/2$ on  $[0,T(N)]$, we have the explicit bound
\begin{equation*}
\frac{\kappa_b^2}{4\theta}\|\chi_N/R_{b, N}\|_{\xLtwo(\O)}^2 \leq \frac{\kappa_b^2|\O|}{\theta R_0^2}. 
\end{equation*} 
Setting $\mu_m = \min\{n, \mu_f\}$ and $\nu_m = \min\{\nu_f, \nu_b\}$, we have
\begin{multline*}
\int_{\O} 2\mu |\xD(u_N) - \frac13 \xdiv (u_N) \xI_3|^2 + \nu |\xdiv (u_N)|^2 \xdif y \geq  \\ 
\int_\O 2 \mu_m |\xD(u_N) - \frac13 \xdiv (u_N) \xI_3|^2 + \nu_m |\xdiv (u_N)|^2 \xdif y = \int_\O \mu_m |\nabla u_N|^2 + \left(\frac{\mu_m}{3} +  \nu_m\right) |\xdiv (u_N)|^2 \xdif y 
\end{multline*}
and choosing suitable $n$ and $\theta$ such that 
\begin{equation*}
\mu_m = \mu_f, \quad \theta \leq \frac16 \big(\nu_m + \frac{\mu_f}{3} \big), 
\end{equation*}
we obtain
\begin{multline*} 
\frac{\xdif}{\xdif t} \tilde {E}_\delta (\rho_N, \rho_N u_N) + \mu_f \|\nabla u_N\|_{\xLtwo(\O)}^2  \leq \tilde {E}_\delta(\rho_N, \rho_N u_N)  \\ 
+  c(\beta, \delta)\|g\|_{\xLn{\frac{2\beta}{\beta - 1}}(\O)}^{\frac{2\beta}{\beta - 1}} + c(\gamma_f, \gamma_b, \beta, \delta, \mu_f, \nu_m) |\O| + c(\kappa_b, \mu_f, \nu_m) \frac{|\O|}{R_0^2}.
\end{multline*} 
Consequently, by Gronwall's lemma, for $t \in [0, T(N)]$, we have
\begin{multline}
E_N(\rho_N(t), \rho_N u_N(t))  + \mu_f \int_0^t \|\nabla u_N\|_{\xLtwo(\O)}^2 \xdif \tau \\ 
\leq e^t \tilde {E}_\delta(\rho_0, q_0) + \int_0^t \left( c(\beta, \delta)\|g\|_{\xLn{\frac{2\beta}{\beta - 1}}(\O)}^{\frac{2\beta}{\beta - 1}}  + c(\gamma_f, \gamma_b, \beta, \delta, \mu_f, \nu_m) |\O| + c(\kappa_b, \mu_f, \nu_m)\frac{|\O|}{R_0^2}  \right) e^{t-\tau} \xdif \tau
\label{eq:cvg_fglrk:energyN_estimate_Gronwall}
\end{multline} 
A direct consequence of \eqref{eq:cvg_fglrk:energyN_estimate_Gronwall} is the uniform bound
\begin{multline}
\|\nabla u_N\|_{\xLtwo((0, T(N)) \times \O)}^2 \\ 
\leq e^{\tilde {T}} \left( \tilde {E}_\delta(\rho_0, q_0)   + c(\beta, \delta)\|g\|_{\xLn{\frac{2\beta}{\beta - 1}}((0, \tilde {T})\times \O)}^{\frac{2\beta}{\beta - 1}} + c(\gamma_f, \gamma_b, \beta, \delta, \mu_f, \nu_m) \tilde {T} |\O| + c(\kappa_b, \mu_f, \mu_m) \frac{\tilde {T} |\O|}{R_0^2} \right) = K,
\label{eq:cvg_fglrk:gduN_bounds}
\end{multline}
for any fixed $\tilde {T} > 0$ and $0 < T(N) \leq \tilde {T}$. By equivalence of the $\xLinfty$ and $\xLtwo$ norms on $X_N$, it follows that   
\begin{equation}
0 < \underline{\rho} \exp\left(- d(N) K \right) \leq \rho(t, x) \leq  \overline {\rho} \exp\left(d(N) K \right) \; \mbox{for all} \; t \in [0, T(N)] \; \mbox{and a.e.} \; x \in \O. 
\label{eq:cvg_fglrk:rhoN_bounds} 
\end{equation} 
Furthermore, \eqref{eq:cvg_fglrk:energyN_estimate_Gronwall} implies
\begin{equation*}
\sup_{t \in [0, T(N)]}  \int_\O \rho_N(t) |u_N(t)|^2 \xdif y \leq K,
\end{equation*} 
which, together with \eqref{eq:cvg_fglrk:rhoN_bounds}, yields
\begin{equation}
\|u_N\|_{\xLinfty(0, T(N); \xLtwo(\O))} \leq \overline{\rho} \exp \left(d(N) K\right) K.
\label{eq:cvg_fglrk:unN_estimate}
\end{equation}
\sloppy
Since \eqref{eq:cvg_fglrk:unN_estimate} is independent of $T(N)$, we can iterate the fixed-point argument with updated initial data $(B_N(T_N),\rho_N(T_N), u_N(T_N))$, provided that
\begin{equation}
\inf_{t \in [0, T(N)]}  R_N(T(N)) \geq \frac{R_0}{2} \quad \mbox{and} \quad \inf_{t \in [0, T(N)]} \xdist (\cB_N(t), \partial \O) \geq 2 \sigma > 0.
\label{eq:cvg_fglrk:RbN_bounds} 
\end{equation}
The first condition in \eqref{eq:cvg_fglrk:RbN_bounds} holds if 
\begin{equation*}
\|\dot {R}_N\|_{\xLone(0, T(N))} \leq \frac{R_0}{2}.
\end{equation*} 
Let $0 < \overline {T} \leq T(N)$, recalling that, for all $t \in [0, \overline {T}]$, 
\begin{equation*}
\dot {R}_N(t) = \frac{5}{4 \pi R_0^4}  \int_{\cB_0} \x1_\O(\tilde{\eta}_N[t](y)) (y - x_0) \cdot u_N(t, \tilde{\eta}_N[t](y)) \xdif y, \quad \tilde {\eta}_N[t](x) = x_N(t) + \frac{R_N(t)}{R_0}(x - x_0),
\end{equation*} 
we deduce 
\begin{equation*}
\|\dot {R}_N\|_{\xLone(0, \overline {T})} \leq \frac{\sqrt{\overline {T}}}{R_0} \sqrt{\frac{5}{4 \pi R_0}} \|u_N\|_{\xLtwo((0, \overline {T}) \times \O)}.
\end{equation*} 
By Poincaré inequality and estimate \eqref{eq:cvg_fglrk:energyN_estimate_Gronwall},
\begin{equation*}
\|u_N\|_{\xLtwo((0, \overline {T}) \times \O)}^2 \leq c_p^2 K.
\end{equation*}
Consequently, if
\begin{equation}
T(N) \leq T_1 = \frac{\pi R_0^5}{5 c_p^2 K}, 
\label{eq:cvg_fglrk:tT_definition}
\end{equation}  
then the first condition in \eqref{eq:cvg_fglrk:RbN_bounds} is satisfied. The second condition in \eqref{eq:cvg_fglrk:RbN_bounds} holds if, for almost every $x \in \cB_0$, 
\begin{equation}
\left\| \frac{\xdif}{\xdif t} \tilde {\eta}_N[\cdot](x)\right \|_{\xLone(0, T(N))} \leq \xdist(\cB_0, \partial \O) - 2 \sigma.
\label{eq:cvg_fglrk:dist_eta_bound}
\end{equation} 
Let $0 < \overline {T} \leq T(N)$, for all $t \in [0, \overline {T}]$ and almost every $x \in \cB_0$, 
\begin{equation*}
\left|\frac{\xdif}{\xdif t} \tilde {\eta}_N[t](x)\right| = \left|\dot {x}_N(t) + \frac{\dot {R}_N(t)}{R_0}(x -  x_0) \right| \leq |\dot {x}_N(t) | + \left|\frac{\dot {R}_N(t)}{R_0}\right| |x -  x_0| \leq |\dot {x}_N(t)| + |\dot {R}_N(t)|.
\end{equation*} 
Recalling that 
\begin{equation*}
\dot {x}_N = \frac{3}{4 \pi R_0^3} \int_{\cB_0} \x1_\O(\tilde {\eta}_N[t](y))  u_N(t , \tilde {\eta}_N[t](y)) \xdif y, \quad \dot {R}_N = \frac{5}{4 \pi R_0^4} \int_{\cB_0} \x1_\O(\tilde {\eta}_N[t](y))  (y -  x_0) \cdot  u_N(t, \tilde {\eta}_N[t](y)) \xdif y,  
\end{equation*}
we obtain  
\begin{equation*}
\|\dot {x}_N \|_{\xLtwo(0, \overline {T})} \leq \frac{\sqrt{\overline {T}}}{R_0}\sqrt{\frac{3}{4\pi R_0}} \|u_N\|_{ \xLtwo((0, \overline {T}) \times \O)}, \qquad \|\dot {R}_N\|_{\xLtwo(0, \overline {T})}  \leq \frac{\sqrt{\overline {T}}}{R_0} \sqrt{\frac{5}{4\pi R_0}} \| u_N \|_{ \xLtwo((0, \overline {T}) \times \O)}. 
\end{equation*}
Thus, for almost every $x \in \cB_0$, 
\begin{equation*}
 \left\|\frac{\xdif}{\xdif t} \tilde {\eta}_N[\cdot]  (x) \right\|_{\xLone(0, \overline {T})} \leq c_0  \sqrt{\overline {T}}  \| u_N \|_{ \xLtwo((0, \overline {T}) \times \O)}, 
\end{equation*} 
with $c_0 = \frac{\sqrt{3} + \sqrt{5}}{R_0\sqrt{4\pi R_0}}$. It follows that if
\begin{equation}
T(N) \leq T_2 = \frac{(\xdist(\cB_0, \partial \O) - 2\sigma)^2}{c_0^2 c_p^2 K}, 
\label{eq:faedo-glrk:distance_T_constraint}
\end{equation} 
then the second condition in \eqref{eq:cvg_fglrk:RbN_bounds} is satisfied. Finally, we choose $\tilde {T}$ in \eqref{eq:cvg_fglrk:gduN_bounds} such that  $\tilde {T} \leq \min \{T_1, T_2\}$ and conclude the existence of a solution $u_N$ at least until time $\tilde {T}$. 

\subsection{Convergence of the Faedo-Galerkin scheme and the limiting system} \label{subsec:cvg_faedo-glrk_approximation}
Proposition~\ref{prop:app-sol:existence_Faedo_Galerkin_solution} establishes the existence of solutions $(\cB_N, \rho_N,  u_N)$ to  \eqref{eq:appx_sol:chibN_regularity}-\eqref{eq:appx_sol:muN_nuN_definition} on a time interval $[0, T]$ with $T$ independent of $N$. In this section, we prove Proposition~\ref{prop:appx_sol:existence_n_solution} by passing to the limit in \eqref{eq:appx_sol:chibN_regularity}-\eqref{eq:appx_sol:muN_nuN_definition} as $N \rightarrow \infty$, thereby recovering a solution to the  $n$-level approximation problem \eqref{eq:appx_sol:chibn_regularity}-\eqref{eq:appx_sol:mun_nun_definition}. Given initial data $(\rho_{0, n}, u_{0, n})$ for the $n$-level approximation problem, we construct initial data $(\rho_{0, N}, u_{0, N})$ for the Faedo-Galerkin scheme satisfying
\begin{equation*}
\rho_{0, N} \rightarrow \rho_{0, n} \; \mbox{in} \; \xWn{1, \infty}(\O), \quad \rho_{0, N} u_{0, N} \rightarrow q_{0, n} \; \mbox{in} \;  \xLtwo(\O), 
\end{equation*} 
and 
\begin{equation}
\int_\O \left( \frac12 \rho_{0, N} | u_{0, N}|^2  + \delta \frac{\rho_N^\beta}{\beta-1} \xdif  y \right) \rightarrow \int_\O \left(\frac12 \frac{|q_{0, n}|^2}{\rho_{0, n}} +  \delta \frac{\rho_N^\beta}{\beta-1} \right) \xdif  y.   
\label{eq:cvg_fglrk:E0N_njr_estimate}
\end{equation} 
 The energy estimate \eqref{eq:appx_sol:energyN_estimate} yields, up to a subsequence,
\begin{align}
& u_N \rightarrow u_n \; \mbox{weakly in} \; \xLtwo(0, T; \xHone(\O)), \label{eq:cvg_fglrk:cv_uN}\\ 
& \rho_N \rightarrow \rho_n \; \mbox{weakly--* in} \; \xLinfty(0, T; \xLbeta(\O)), \notag \\ 
& \nabla \rho_N \rightarrow \nabla \rho_n \; \mbox{weakly in} \; \xLtwo((0, T) \times \O).\label{eq:cvg_fglrk:cv_gpN}
\end{align}
It follows from \eqref{eq:cvg_fglrk:cv_uN} and Proposition~\ref{prop:smlrt-prop:weak_sequential_continuity} that, up to a subsequence, for all $1 \leq p < \infty$, 
\begin{equation}
\chi_N \rightarrow \chi_n \; \mbox{weakly--* in} \; \xLinfty((0, T)\times \xR^3), \; \mbox{strongly in} \; \xCzero([0, T]; \xL_{\xloc}^p(\xR^3)), 
\label{eq:cvg_fglrk:chibN_cvg}
\end{equation} 
where the pair $(\chi_n, u_n)$ is a solution of \eqref{eq:appx_sol:chibn_transport_equation}. Moreover, there exist $x_n, R_n \in \xHone(0, T)$ such that,
\begin{equation*}
\chi_n = \x1_{\cB(x_n, R_n)}, \quad R_n  \geq \frac{R_0}{2}.
\end{equation*} 
Finally, using the strong convergence of $\chi_N$ given in \eqref{eq:cvg_fglrk:chibN_cvg} and the weak convergence of $u_N$ given in \eqref{eq:cvg_fglrk:cv_uN}, it follows that
\begin{equation}
\Pi_N u_N \rightarrow \Pi_n u_n \; \mbox{weakly in} \; \xLtwo (0, T; \cS).  
\label{eq:cvg_fglrk:PiNuN_cvg}
\end{equation} 
Following the same arguments as in \cite[Section 2.4]{Feireisl2001JMF}, \cite[Section 7.8.1]{Novotny2004OUP}, we further get, for all $1 \leq p < \frac{4}{3} \beta$, 
\begin{align*}
& \rho_N \rightarrow \rho_n \; \mbox{in} \; \xCzero([0, T]; \xL_{\rm {weak}}^\beta(\O)), \; \mbox{strongly in} \; \xLn{p}((0, T) \times \O), \\ 
& \rho_N u_N \rightarrow \rho_n u_n \; \mbox{weakly in} \; \xLtwo(0, T; \xLn{\frac{6\beta}{\beta+6}}(\O)), \; \mbox{weakly--* in} \; \xLinfty(0, T; \xLn{\frac{2\beta}{\beta+1}}(\O)).
\end{align*} 
These convergence results allow us to pass to the limit $N \rightarrow \infty$ in the regularized continuity equation \eqref{eq:appx_sol:rhoe_continuity_equation}. We now consider the limit in the momentum equation \eqref{eq:appx_sol:rhonun_momentum_equation}. The difficult terms are, for $t \in [0, T]$ and $1 \leq k \leq N$, 
\begin{equation*}
A_N(t, \psi_k) = \int_\O \rho_N u_N \otimes u_N : \nabla \psi_k \xdif y, \quad B_N(t, \psi_k) = \int_\O \varepsilon (\nabla u_N \nabla \rho_N) \cdot \psi_k \xdif y, \quad C_N (t, \psi_k)  = \int_\O \rho_N^\beta \xdiv (\psi_k) \xdif y, 
\end{equation*}
where the functions $\psi_k$ are defined in \eqref{eq:appx_sol:psi_family}. Proceeding similarly to \cite[Section 2.4]{Feireisl2001JMF},  \cite[Section 7.8.2]{Novotny2004OUP}, we get, for $1 \leq p < \frac{4}{3}\beta$, 
\begin{align}
\rho_N u_N \otimes u_N &\rightarrow \rho_n u_n \otimes u_n \; \mbox{weakly in} \; \xLtwo(0, T; \xLn{\frac{6\beta}{4\beta+3}}(\O)), 
\label{eq:cvg_fglrk:rhoNuNuN_cvg} \\ 
\varepsilon \nabla u_N \nabla \rho_N &\rightarrow \varepsilon \nabla u_n \nabla \rho_n \; \mbox {weakly in} \; \xLtwo(0, T; \xLn{\frac{5 \beta - 3}{4\beta}}(\O)), \\ 
\rho_N &\rightarrow \rho_n \; \mbox{strongly in} \; \xLn{p}((0, T) \times \O). 
\label{eq:cvg_fglrk:rhoN_cvg_LpLp}
\end{align} 
It follows from \eqref{eq:cvg_fglrk:rhoNuNuN_cvg}-\eqref{eq:cvg_fglrk:rhoN_cvg_LpLp} the following weak convergences in $\xLone(0, T)$: 
\begin{equation*}
\begin{aligned}
& A_N(\cdot, \psi_k) \rightarrow A_n(\cdot, \psi_k) = \int_\O \rho_n u_n \otimes u_n : \nabla \psi_k \xdif y, \\
& B_N(\cdot, \psi_k) \rightarrow B_n(t, \psi_k) = \int_\O \varepsilon \left(\nabla u_n \nabla \rho_n \right) \cdot \psi_k \xdif y, \\
& C_N(\cdot, \psi_k) \rightarrow C_n(t, \psi_k) = \int_\O \rho_n^\beta \xdiv (\psi_k) \xdif y.
\end{aligned} 
\end{equation*} 
Thus, we have recovered \eqref{eq:appx_sol:rhonun_momentum_equation} as a limit of equation \eqref{eq:appx_sol:rhoNuN_momentum_equation} as $N \rightarrow \infty$, establishing the existence of a solution $(\cB_n, \rho_n, u_n)$ to system \eqref{eq:appx_sol:chibn_regularity}-\eqref{eq:appx_sol:mun_nun_definition}.

It remains to prove energy inequality \eqref{eq:appx_sol:energyn_estimate}. To this purpose, we follow the same lines as in \cite[Section 2.4]{Feireisl2001JMF} \cite[Section 7.8.3]{Novotny2004OUP}. The strong convergence of $\chi_N$ given in \eqref{eq:cvg_fglrk:chibN_cvg} and $\rho_N$ given in \eqref{eq:cvg_fglrk:rhoN_cvg_LpLp}, and the convergence of $\rho_N u_N \otimes u_N$ given in \eqref{eq:cvg_fglrk:rhoNuNuN_cvg} ensure that, up to a subsequence, 
\begin{equation*}
\int_\O \left(\frac12 \rho_N |u_N|^2 + P_\delta(\chi_N, \rho_N) \right) \xdif y \rightarrow \int_\O \left(\frac12 \rho_n |u_n|^2 + P_\delta(\chi_n, \rho_n) \right) \xdif y \; \mbox {in} \; \cD^\prime(0, T).
\end{equation*} 
By the lower semicontinuity of convex functionals, the weak convergence of $u_N$ given in \eqref{eq:cvg_fglrk:cv_uN}, the strong convergence of $\chi_N$ given in \eqref{eq:cvg_fglrk:chibN_cvg} and the weak convergence of $\Pi_N u_N$ given in \eqref{eq:cvg_fglrk:PiNuN_cvg}, we deduce 
\begin{equation}
\begin{aligned}
& \int_0^T \int_\O \left(2 \mu_n \left|\xD(u_n) - \frac13 \xdiv (u_n) \right|^2 + \nu_n |\xdiv (u_n)|^2 \right) \xdif y \xdif \tau \\
& \qquad \leq \liminf \limits_{N \rightarrow \infty} \int_0^T \int_\O \left(2 \mu_N \left|\xD(u_N) - \frac13 \xdiv (u_N) \right|^2 + \nu_N |\xdiv u_N|^2 \right) \xdif y \xdif \tau, \\  
& \int_0^T \int_\O \chi_n |u_n - \Pi_n u_n|^2  \xdif y \xdif \tau \leq \liminf \limits_{N \rightarrow \infty} \int_0^T \int_\O \chi_N |u_N - \Pi_N u_N|^2 \xdif y \xdif \tau . 
\end{aligned} 
\label{eq:cvg_fglrk:lsc_dspt_pnlzt}
\end{equation} 
Using the strong convergence of $\nabla \rho_N \rightarrow \nabla \rho_n$  given in \eqref{eq:cvg_fglrk:cv_gpN}, strong convergence of $\rho_N$ given in \eqref{eq:cvg_fglrk:rhoN_cvg_LpLp}, and Fatou's lemma, we further obtain
\begin{equation}
\int_0^T \int_\O \rho_n^{\beta - 2} |\nabla \rho_n|^2  \xdif y \xdif \tau  \leq \liminf \limits_{N \rightarrow \infty} \int_0^T \int_\O \rho_N^{\beta - 2} |\nabla \rho_N|^2  \xdif y \xdif \tau.
\label{eq:cvg_fglrk:lsc_gdt_rhon}  
\end{equation} 
Collecting these results, we conclude that the energy inequality \eqref{eq:appx_sol:energyn_estimate} holds. Finally, we follow the same ideas as the calculations \eqref{eq:cvg_fglrk:dist_eta_bound}-\eqref{eq:faedo-glrk:distance_T_constraint} to conclude that there exists $T > 0$ such that  
\[\xdist (\cB_N(t), \partial \O) \geq 2 \sigma \quad \forall \, t \in [0, T]. \]

\subsection{High penalization limit in the momentum equation} \label{subsec:cvg_high_penalization_term_momentum_equation}
Proposition~\ref{prop:appx_sol:existence_n_solution} establishes the existence of weak solutions $(\rho_n, u_n, \cB_n)$ to \eqref{eq:appx_sol:chibn_regularity}-\eqref{eq:appx_sol:mun_nun_definition}. In this section, we prove Proposition~\ref{prop:appx_sol:existence_epsilon_level} by passing to the limit $n \rightarrow \infty$, thereby recovering a weak solution to the $\varepsilon$-level approximation problem \eqref{eq:appx_sol:rhoe_continuity_equation}-\eqref{eq:appx_sol:mue_nue_definition}. Given initial data $(\rho_{0, \varepsilon}, u_{0, \varepsilon})$ for the $\varepsilon$-level approximation problem, we consider initial data $(\rho_{0, n}, u_{0, n})$ for the $n$-level problem satisfying $$\rho_{0, n} = \rho_{0, \varepsilon}, \quad u_{0, n} = u_{0, \varepsilon}.$$
The energy estimate \eqref{eq:appx_sol:energyn_estimate} yields, up to a subsequence,
\begin{align}
& u_n \rightarrow u_\varepsilon \; \mbox{weakly in} \; \xLtwo(0, T; \xHone(\O)), \label{eq:high_pen:cv_un} \\ 
& \rho_n \rightarrow \rho_\varepsilon \; \mbox{weakly--* in} \;  \xLinfty(0, T;  \xLbeta(\O)), \notag \\
&\nabla \rho_n \rightarrow \nabla \rho_\varepsilon \; \mbox{weakly in} \; \xLtwo((0, T)\times \O). \notag
\end{align} 
It follows from \eqref{eq:high_pen:cv_un} and Proposition~\ref{prop:smlrt-prop:weak_sequential_continuity} that, up to a subsequence, for all $1\leq p<\infty$, 
\begin{equation}
 \chi_n \rightarrow \x1_{\cB_\varepsilon} \; \mbox{weakly--* in} \; \xLinfty((0, T) \times \xR^3)), \; \mbox{strongly in} \; \xCzero([0, T]; L_{\xloc}^p(\xR^3)), 
 \label{eq:high_pen:chin_convergence}
\end{equation} 
where the pair  $(\x1_{\cB_\varepsilon}, u_\varepsilon)$ satisfies
\begin{equation*}
\partial_t \x1_{\cB_\varepsilon} + \tilde {\Pi}_\varepsilon u_\varepsilon \cdot \nabla \x1_{\cB_\varepsilon} = 0 \; \mbox{in} \; \cD^\prime ((0, T) \times \xR^3), \quad \cB_\varepsilon (0)  = \cB_0. 
\end{equation*}
Let $x_\varepsilon, R_\varepsilon \in \xHone([0, T])$ such that $\cB_\varepsilon = \cB(x_\varepsilon, R_\varepsilon)$, we also have 
\begin{equation*}
R_\varepsilon(t) \geq \frac{R_0}{2} \quad \forall \, t \in [0, T]. 
\end{equation*}
Finally, using the strong convergence of $\chi_n$ given in \eqref{eq:high_pen:chin_convergence} and the weak convergence of $u_n$ given in \eqref{eq:high_pen:cv_un}, it follows that
\begin{equation}
\Pi_n u_n \rightarrow \Pi_\varepsilon u_\varepsilon \; \mbox{weakly in} \; \xLtwo (0, T; \cS).  
\label{eq:high_pen:Pinun_convergence}
\end{equation} 
The penalization term in the energy estimate \eqref{eq:appx_sol:energyn_estimate} leads to 
\begin{equation}
\sqrt{\chi_n}(u_n - \Pi_n u_n) \rightarrow 0 \; \mbox{in} \; \xLtwo((0,T)\times \O).
\label{eq:high_pen:chinPinun_convergence}
\end{equation} 
Using the strong convergence of $\chi_n$ given in \eqref{eq:high_pen:chin_convergence}, together with the weak convergences of $u_n$ and $\Pi_n u_n$ given in \eqref{eq:high_pen:cv_un}  and \eqref{eq:high_pen:Pinun_convergence}, respectively, we get in the limit 
\begin{equation*}
\x1_{\cB_\varepsilon} (u_\varepsilon - \Pi_\varepsilon u_\varepsilon) = 0 \; \mbox{in} \; (0, T) \times \O.
\end{equation*} 
It follows that 
\begin{equation*}
u_\varepsilon = \Pi_\varepsilon u_\varepsilon \; \mbox{in} \; Q_\varepsilon,
\end{equation*}
where $Q_\varepsilon$ is the space-time set defined in \eqref{eq:appx_sol:Qe_def}. By spherical symmetry, the pair $(u_\varepsilon, \x1_{\cB_\varepsilon})$ satisfies 
\begin{equation*}
\partial_t \x1_{\cB_\varepsilon} + \Pi_\varepsilon u_\varepsilon \cdot \nabla \x1_{\cB_\varepsilon} = 0 \; \mbox{in} \; \cD^\prime ((0, T) \times \xR^3), \quad \cB_\varepsilon (0)  = \cB_0.
\end{equation*}
Consequently, the velocity $u_\varepsilon$ is compatible with the system $\{\overline{\cB_\varepsilon}, \eta_\varepsilon\}$, where $\eta_\varepsilon$ denotes the propagator associated to $\Pi_\varepsilon u_\varepsilon$.

Following the same arguments as in \cite[Section 2.4]{Feireisl2001JMF},  \cite[Section 7.8.1]{Novotny2004OUP}, we further get, for all $1 \leq p < \frac{4}{3} \beta$, 
\begin{align}
 & \rho_n \rightarrow \rho_\varepsilon \; \mbox{in} \; \xCzero([0, T]; \xLn{\beta}_{\xweak}(\O)), \; \mbox{strongly in} \; \xLn{p}((0, T) \times \O), \\
&\rho_n u_n \rightarrow \rho_\varepsilon u_\varepsilon \; \mbox {weakly in} \; \xLtwo(0, T; \xLn{\frac{6 \beta}{\beta + 6}}(\O)), \; \mbox{weakly--* in} \; \xLinfty(0, T; \xLn{\frac{2\beta}{\beta + 1}}(\O)), \label{eq:high_pen:cv_rhonun} 
\end{align} 
These  convergence results allow us to pass to the limit $n \rightarrow \infty$ in the regularized continuity equation \eqref{eq:appx_sol:rhoe_continuity_equation}. 

We now consider the limit of the momentum equation \eqref{eq:appx_sol:rhoeue_momentum_equation}. Let $\varphi \in  {\cT}(Q_\varepsilon)$. Since $\overline {Q_\varepsilon}$ is a compact subset of $(0, T) \times \O$, there exists $\sigma_0 > 0$ and $\varphi_b \in \xCzero([0, T]; \cS)$ such that, for all $ t \in [0, T]$, 
\begin{equation*}
\varphi(t, \cdot) = \varphi_b(t, \cdot) \; \mbox {in} \;  \cB_\varepsilon^{\sigma_0}(t)
\end{equation*}
with for all $\sigma > 0 $, $\cB_\varepsilon^{\sigma} = \cB(x_\varepsilon, R_\varepsilon + \sigma)$. In particular, for $0 < \sigma \leq \sigma_0$, the following identity holds: 
\begin{equation}
\varphi = (1 - \x1_{\cB_\varepsilon^\sigma}) \varphi + \x1_{\cB_\varepsilon^\sigma} \varphi_b.
\label{eq:high_pen:varphi_1Bes_decomp}
\end{equation}
Moreover, using the strong convergence of $\chi_n$ given in \eqref{eq:high_pen:chin_convergence}, it follows that there exists $n_0 \in \xN$ such that, for all $n \geq n_0$ and $t \in [0, T]$, $$\cB_n(t) \subset \cB_\varepsilon^{\sigma_0}(t),$$ and consequently
\begin{equation}
\chi_n \x1_{\cB_\varepsilon^{\sigma_0}} = \chi_n, \quad (1 - \chi_n) (1 - \x1_{\cB_\varepsilon^{\sigma_0}}) = (1 - \x1_{\cB_\varepsilon^{\sigma_0}}), \quad \chi_n  (1 - \x1_{\cB_\varepsilon^{\sigma_0}}) = 0.
\label{eq:high_pen:chin_1Bes_relations}
\end{equation}
Using the decomposition \eqref{eq:high_pen:varphi_1Bes_decomp}, the identity 
\begin{equation}
\xD(\varphi_b) - \frac13 \xdiv (\varphi_b) = 0,
\label{eq:high_pen:phib_shear_viscosity}
\end{equation}
and the definition of $\mu_n = \mu_f (1 - \chi_n) + n \chi_n$, we have 
\begin{equation*}
\begin{aligned} 
\mu_n (\xD(\varphi) - \frac13 \xdiv (\varphi)) & = \mu_n (1 - \x1_{\cB_\varepsilon^{\sigma_0}}) (\xD(\varphi) - \frac13 \xdiv (\varphi)) +  \mu_n \x1_{\cB_\varepsilon^{\sigma_0}} (\xD(\varphi_b) - \frac13 \xdiv (\varphi_b)), \\ 
& = (\mu_f  (1-\chi_n) (1 - \x1_{\cB_\varepsilon^{\sigma_0}}) + n \chi_n  (1 - \x1_{\cB_\varepsilon^{\sigma_0}}) ) (\xD( \varphi) - \frac13 \xdiv (\varphi)).
\end{aligned}
\end{equation*}
Applying \eqref{eq:high_pen:chin_1Bes_relations} yields 
\begin{equation*}
\mu_n (\xD(\varphi) - \frac13 \xdiv (\varphi)) = \mu_f(1- \x1_{\cB_\varepsilon^{\sigma_0}}) (\xD (\varphi) - \frac13 \xdiv (\varphi)).
\end{equation*}
The weak convergence of $u_n$ given in \eqref{eq:high_pen:cv_un} and the strong convergence of $\chi_n$ given in \eqref{eq:high_pen:chin_convergence} therefore imply
\begin{multline}
\displaystyle  \int_0^T \int_\O 2 \mu_n \left(\xD (u_n) - \frac13 \xdiv (u_n) \right) : \left(\xD (\varphi) - \frac13 \xdiv (\varphi) \right)  + \nu_n \xdiv (u_n):  \xdiv (\varphi) \xdif y \xdif \tau \\ 
\xrightarrow[]{n \rightarrow \infty}  \displaystyle  \int_0^T \int_\O   2 \mu_f (1 - \x1_{\cB_\varepsilon^{\sigma_0}}) \left(\xD (u_\varepsilon) - \frac13 \xdiv (u_\varepsilon) \right) : \left(\xD (\varphi) - \frac13 \xdiv (\varphi) \right)  \xdif y \xdif \tau \\ 
+ \displaystyle  \int_0^T \int_\O \nu_\varepsilon \xdiv (u_\varepsilon):\xdiv (\varphi)  \xdif y \xdif \tau.
\label{eq:high_pen:limit_viscosity_term}
\end{multline}
Finally, using again the identity \eqref{eq:high_pen:phib_shear_viscosity} together with the definition of $\mu_\varepsilon = \mu_f (1 - \x1_{\cB_\varepsilon})$, we compute
\begin{equation*}
\begin{aligned}  
\mu_f  (1 - \x1_{\cB_\varepsilon^{\sigma_0}}) (\xD( \varphi) - \frac13 \xdiv (\varphi)) & = \mu_f  (1 - \x1_{\cB_\varepsilon}) (1 - \x1_{\cB_\varepsilon^{\sigma_0}}) (\xD(\varphi) - \frac13 \xdiv (\varphi)), \\ 
& = \mu_\varepsilon (1 - \x1_{\cB_\varepsilon^{\sigma_0}}) (\xD(\varphi) - \frac13 \xdiv (\varphi)) + \mu_\varepsilon \x1_{\cB_\varepsilon^{\sigma_0} }(\xD(\varphi_b) - \frac13 \xdiv (\varphi_b)), \\  
& = \mu_\varepsilon (\xD(\varphi) - \frac13 \xdiv (\varphi)).
\end{aligned} 
\end{equation*}
Consequently, the limit in \eqref{eq:high_pen:limit_viscosity_term} can be written as
\begin{equation*}
\displaystyle \int_0^T \int_\O 2 \mu_\varepsilon \left(\xD (u_\varepsilon) - \frac13 \xdiv (u_\varepsilon) \right) : \left( \xD (\varphi) - \frac13 \xdiv (\varphi) \right)  + \nu_\varepsilon \xdiv (u_\varepsilon):\xdiv (\varphi) \xdif y \xdif \tau.
\end{equation*} 
Similarly, using the identity \eqref{eq:high_pen:chin_1Bes_relations}, we obtain that the penalization term vanishes for $n \geq n_0$:
\begin{equation*}
\begin{aligned}
n \int_0^T \int_\O \chi_n (u_n - \Pi_n u_n) \cdot (\varphi - \Pi_n \varphi) & =  n  \int_0^T \int_\O \x1_{\cB_\varepsilon^{\sigma_0}} \chi_n (u_n - \Pi_n u_n) \cdot (\varphi - \Pi_n \varphi), \\ 
& = n  \int_0^T \int_\O \chi_n (u_n - \Pi_n u_n) \cdot (\varphi_b - \Pi_n \varphi_b) = 0  .
\end{aligned} 
\end{equation*}

We now focus on the convergence of the nonlinear term
\begin{equation}
\int_0^T \int_\O \rho_n u_n \otimes u_n : \xD (\varphi) \xdif y \xdif \tau. 
\label{eq:high_pen:adv_term_decomp}
\end{equation}
For $0 < \sigma \leq \sigma_0$, we write 
\begin{equation}
\int_0^T \int_\O \rho_n u_n \otimes u_n : \xD (\varphi) \xdif y \xdif \tau
- \int_0^T \int_\O \rho_\varepsilon u_\varepsilon \otimes u_\varepsilon : \xD (\varphi) \xdif y \xdif \tau = \displaystyle \sum_{i=1}^5 M_i(\sigma, n),
\end{equation}
where 
\begin{equation*}
\begin{aligned}
& M_1(\sigma, n) =  \int_0^T \int_\O (1 - \x1_{\cB_\varepsilon^\sigma}) \rho_n u_n \otimes u_n : \xD (\varphi) \xdif y \xdif \tau  - \int_0^T \int_\O (1 - \x1_{\cB_\varepsilon^\sigma}) \rho_\varepsilon u_\varepsilon \otimes u_\varepsilon : \xD (\varphi) \xdif y \xdif \tau \\ 
& M_2(\sigma, n) = \int_0^T \int_\O (\x1_{\cB_\varepsilon^\sigma} - \x1_{\cB_\varepsilon}) \rho_n u_n \otimes u_n : \xD (\varphi)  \xdif y \xdif \tau - \int_0^T \int_\O (\x1_{\cB_\varepsilon^\sigma} - \x1_{\cB_\varepsilon}) \rho_\varepsilon u_\varepsilon \otimes u_\varepsilon : \xD (\varphi)  \xdif y \xdif \tau, \\
& M_3(\sigma, n) = \int_0^T \int_\O (\x1_{\cB_\varepsilon} - \chi_n) \rho_n u_n \otimes u_n : \xD (\varphi)  \xdif y \xdif \tau, \\ 
& M_4(\sigma, n) = \int_0^T \int_\O \chi_n \rho_n u_n \otimes (u_n - \Pi_n u_n) : \xD (\varphi)  \xdif y \xdif \tau, \\ 
& M_5(\sigma, n) = \int_0^T \int_\O \chi_n \rho_n u_n \otimes \Pi_n u_n : \xD (\varphi)  \xdif y \xdif \tau - \int_0^T \int_\O \x1_{\cB_\varepsilon} \rho_\varepsilon u_\varepsilon \otimes \Pi_\varepsilon u_\varepsilon : \xD (\varphi)  \xdif y \xdif \tau.
\end{aligned} 
\end{equation*} 
This decomposition separates the contributions arising from the fluid region, the fluid–bubble interface, the approximation of the bubble domain, and the penalization errors. Using the convergence of $u_n$ given in \eqref{eq:high_pen:cv_un} together with the convergence of $\rho_n u_n$ given in \eqref{eq:high_pen:cv_rhonun}, we deduce that both $\rho u_n \otimes u_n$ and  $\rho u_\varepsilon \otimes u_\varepsilon$ belong to $\xLtwo(0, T; \xLn{\frac{6\beta}{4\beta + 3}}(\O)),$  and that the sequence $\{\rho_n u_n \otimes u_n\}$ is uniformly bounded in $\xLtwo(0, T; \xLn{\frac{6\beta}{4\beta + 3}}(\O))$. Moreover, since for all $1 \leq  p < \infty$, 
\begin{equation} 
\x1_{\cB_\varepsilon^\sigma} \xrightarrow []{\sigma \rightarrow 0} \x1_{\cB_\varepsilon} \; \mbox{in} \; \xCzero([0, T]; \xLn{p}(\O)), 
\label{eq:high_pen:1Bes_cvg}
\end{equation} 
it follows that 
\begin{equation}
M_2 (\sigma, n) \xrightarrow []{\sigma \rightarrow 0} 0 \; \mbox{uniformly in} \; n.
\label{eq:high_pen:M2_cvg}
\end{equation} 
Similarly, by applying \eqref{eq:high_pen:chin_convergence} in place of \eqref{eq:high_pen:1Bes_cvg}, we obtain
\begin{equation}
M_3(\sigma, n) \xrightarrow []{n \rightarrow \infty} 0 \; \mbox{uniformly in} \; \sigma.
\label{eq:high_pen:M3_cvg}
\end{equation}
Finally, the convergence \eqref{eq:high_pen:cv_rhonun} implies that $\rho_n u_n$ is uniformly bounded in $\xLinfty(0, T; \xLn{\frac{2\beta}{\beta + 1}}(\O))$ and combined with \eqref{eq:high_pen:chinPinun_convergence}, this yields 
\begin{equation}
M_4(\sigma, n) \xrightarrow []{n \rightarrow \infty} 0 \; \mbox{uniformly in} \; \sigma.
\label{eq:high_pen:M4_cvg}
\end{equation} 
Let $\theta > 0$, we deduce from \eqref{eq:high_pen:M2_cvg}-\eqref{eq:high_pen:M4_cvg} that there exists $(\sigma_1, n_1)$ such that for all $(\sigma, n)$ satisfying $0 < \sigma \leq \min \{\sigma_0, \sigma_1\}$ and $n \geq \max \{ n_0, n_1 \}$,  
\begin{equation*}
|M_2(\sigma, n)| \leq \theta/5, \quad |M_3(\sigma, n)| \leq \theta/5, \quad |M_4(\sigma, n)| \leq \theta/5.  
\end{equation*} 
The difficult terms are $M_1 (\sigma, n)$ and $M_5(\sigma, n)$, as so far we have only established weak convergence for the sequences $\{ \rho_n u_n \}$, $\{ u_n \}$ and $\{ \Pi_n u_n \}$. We first consider the convergence of the term $M_5(\sigma, n)$. Recall that, for all $n \in \xN$, there exist $ V_n, \omega_n, \Lambda_n \in \xLtwo(0, T)$ and $ V_\varepsilon, \omega_\varepsilon, \Lambda_\varepsilon \in \xLtwo(0, T)$ such that  
\begin{equation}
\Pi_n u_n = V_n + \omega_n \times (x - x_n) + \frac{\Lambda_n}{3} (x- x_n), \quad \Pi_\varepsilon u_\varepsilon = V_\varepsilon + \omega_\varepsilon \times (x - x_\varepsilon) + \frac{\Lambda_\varepsilon}{3} (x- x_\varepsilon), \label{eq:high_pen:Pibnun_expression}
\end{equation} 
and, by \eqref{eq:high_pen:Pinun_convergence},
\begin{equation}
V_n \rightarrow V_\varepsilon, \quad \omega_n \rightarrow \omega_\varepsilon, \quad \Lambda_n \rightarrow \Lambda_\varepsilon \; \mbox{weakly in} \; \xLtwo(0, T), \quad x_n \rightarrow x_\varepsilon \; \mbox{in} \; \xCzero([0, T]).  
\label{eq:high_pen:VoL_weak_convergence}
\end{equation} 
We introduce the functions $\tilde {V}_n, \tilde {\omega}_n, \tilde {\Lambda}_n$ and  $\tilde {V}_\varepsilon, \tilde {\omega}_\varepsilon, \tilde {\Lambda}_\varepsilon$ in $\xLtwo(0, T)$ defined, for $i \in \{1, 2, 3\}$, by 
\begin{align*}
& \tilde {V}_n \cdot e_i = (\rho_n u_n, \, \chi_n e_i)_\O,  \quad \tilde {\omega}_n \cdot e_i = (\rho_n u_n, \chi_n e_i \times (x - x_n))_\O, \quad \tilde {\Lambda}_n(t) = (\rho_n u_n, \chi_n (x - x_n))_\O, \\ 
& \tilde {V}_\varepsilon \cdot e_i = (\rho_\varepsilon u_\varepsilon, \, \x1_{\cB_\varepsilon} e_i)_\O, \quad \tilde {\omega}_\varepsilon \cdot e_i = (\rho_\varepsilon u_\varepsilon, \x1_{\cB_\varepsilon} e_i \times (x - x_\varepsilon))_\O, \quad \tilde {\Lambda}_\varepsilon(t) = (\rho_\varepsilon u_\varepsilon, \x1_{\cB_\varepsilon} (x - x_\varepsilon))_\O.
\end{align*} 
where $\{e_i\}$ is the canonical orthonormal basis of $\xR^3$.  Since, for $n \geq n_0$, $$ \chi_n \xD (\varphi) = \x1_{\cB_\varepsilon} \xD (\varphi) = \xdiv (\varphi_b),$$ it follows that
\begin{equation*}
M_5(\sigma, n) = \int_0^T \xdiv(\varphi_b) (V_n \cdot \tilde {V}_n + \omega_n \cdot \tilde {\omega}_n + \Lambda_n \tilde {\Lambda}_n -  V_\varepsilon \cdot \tilde {V}_\varepsilon - \omega_\varepsilon \cdot \tilde {\omega}_\varepsilon - \Lambda_\varepsilon \tilde {\Lambda}_\varepsilon) \xdif \tau .
\end{equation*} 
Consequently, to establish the convergence of the term $M_5(\sigma, n)$, it  is sufficient to prove the strong convergence in time of $(\tilde {V}_n, \tilde {\omega}_n, \tilde {\Lambda}_n)$ to $(\tilde {V}_\varepsilon, \tilde {\omega}_\varepsilon, \tilde {\Lambda}_\varepsilon)$. Let $i \in \{1,2,3\}$, the idea is to use $$\chi_n e_i, \quad  \chi_n e_i \times (x - x_n), \; \chi_n (x - x_n), $$ as test functions in the momentum equation \eqref{eq:appx_sol:rhonun_momentum_equation}. However, since these functions are not regular enough to be used directly as test functions, we introduce $\nu$-regularizations defined by 
\begin{equation*}
\tilde {V}_{i, n}^\nu = (\rho_n u_n, \chi_n^\nu e_i)_\O,  \quad \tilde {\omega}_{i, n}^\nu= (\rho_n u_n, \chi_n^\nu e_i \times (x - x_n))_\O, \quad \tilde {\Lambda}_n(t) = (\rho_n u_n, \chi_n^\nu (x - x_n))_\O,
\end{equation*} 
where $\chi_n^\nu$ is a smooth approximation of the color function $\chi_n$ defined by 
\begin{equation*}
\chi_n^\nu = \cK^\nu \star \x1_{\cB(x_n, (1 + \nu) R_n)}, 
\end{equation*}
with $\cK^\nu$ a standard mollifier in $\Omega$. In particular, the following properties hold:
\[ \chi_n^\nu \in \cD((0, T) \times \O), \quad  \chi_n^\nu \xrightarrow[]{\nu \rightarrow 0} \chi_n \; \mbox{in} \; \xCzero([0, T]; \xLn{p}(\O)) \; \mbox {uniformly in} \; n \] and  
\[ \left\{
\begin{aligned} 
& \chi_n (\Pi_n (\chi_n^\nu e_i) - \chi_n^\nu e_i) = 0, \\
& \chi_n (\Pi_n (\chi_n^\nu e_i \times (x - x_n)) - \chi_n^\nu e_i \times (x - x_n)) = 0, \\  
& \chi_n (\Pi_n (\chi_n^\nu (x - x_n)) - \chi_n^\nu  (x - x_n)) = 0.
\end{aligned} 
\right. 
\] 
\begin{lmm}
For all $\nu > 0$, there exists a non-decreasing map $\phi_\nu$ such that, for all $i \in \{1,2,3\}$, 
\begin{equation*}
\tilde {V}_{i, \phi_\nu(n)}^\nu \rightarrow \tilde {V}_i^\nu, \quad \tilde {\omega}_{i, \phi_\nu(n)}^\nu \rightarrow \tilde {\omega}_i^\nu, \quad \tilde {\Lambda}_{\phi_\nu(n)}^\nu \rightarrow \tilde {\Lambda}^\nu  \; \mbox{in} \; \xCzero([0, T])
\end{equation*} 
\label{lem:high_pen:tdV_tdo_tdL_convergence}
\end{lmm} 
\begin{proof}[Proof of Lemma~\ref{lem:high_pen:tdV_tdo_tdL_convergence}]
We only prove the result for $\tilde {V}_n^\nu$, the proof for $ \tilde {\omega}_n^\nu$ and $\tilde {\Lambda}_n^\nu$ follows the same arguments. Let $i \in \{1, 2, 3\}$ and $t_1, \, t_2 \in [0, T]$, setting $\varphi_{i, n}^\nu = \chi_n^\nu e_i$ as a test function in \eqref{eq:appx_sol:rhonun_momentum_equation}, we have on one hand, 
\begin{equation} 
\begin{aligned} 
&((\rho_n u_n)(t_2) - (\rho_n u_n)(t_1), \varphi_{i, n}^\nu(t_2))_\O =  \int_{t_1}^{t_2} \int_{\O} \rho_n u_n \otimes u_n : \xD(\varphi_{i, n}^\nu(t_2)) \xdif y  \xdif \tau \\
&- \int_{t_1}^{t_2} \int_{\O} \mu_f \left(\xD (u_n) - \frac13 \xdiv (u_n) \xI_3 \right) : \left( \xD(\varphi_{i, n}^\nu(t_2)) - \frac13 \xdiv (\varphi_{i, n}^\nu(t_2)) \xI_3 \right)  \xdif y \xdif \tau \\ 
&+ \int_{t_1}^{t_2} \int_\O \nu_n \xdiv (u_n) \xdiv (\varphi_{i, n}^\nu(t_2)) \xdif y \xdif \tau + \int_{t_1}^{t_2} \int_\O p(\rho_n) \xdiv (\varphi_{i, n}^\nu(t_2)) \xdif y \xdif \tau \\ 
&- \varepsilon \int_{t_1}^{t_2} \int_\O (\nabla \rho_n \cdot \nabla) u_n \cdot \varphi_{i, n}^\nu(t_2) \xdif y \xdif \tau - \int_{t_1}^{t_2} \int_\O \rho_n g \cdot \varphi_{i, n}^\nu(t_2) \xdif y \xdif \tau  + \int_{t_1}^{t_2} \int_\O  \frac{\kappa_n}{R_n} \xdiv (\varphi_{i, n}^\nu(t_2)) \xdif y \xdif \tau.
\end{aligned}
\label{eq:high_pen:VoL_identity}
\end{equation} 
Using Hölder inequalities together with uniform bounds on $\rho_n u_n$, $u_n$, $\rho_n$ and the regularity of $\varphi_{i, n}^\nu = \chi_n^\nu e_i$, we deduce that the right-hand side of \eqref{eq:high_pen:VoL_identity} is bounded by 
\begin{equation}
\begin{aligned}
& \int_{t_1}^{t_2} \|\rho_n u_n \otimes u_n\|_{\xLn{\frac{6\beta}{4\beta + 3}}(\O)} \|\nabla \varphi_{i, n}^\nu(t_2) \|_{\xLn{\frac{6\beta}{2\beta - 3}}(\O)}  + \left(\frac{4}{3} \mu_f + \nu_f\right) \|\nabla u_n\|_{\xLtwo(\O)} \|\nabla \varphi_{i, n}^\nu(t_2) \|_{\xLtwo(\O)} \\ 
& + \nu_b \|\nabla u_n\|_{\xLtwo(\O)} \|\nabla \varphi_{i, n}^\nu(t_2) \|_{\xLtwo(\O)} + a_f \|\rho_n\|_{\xLn{\beta + 1}(\O)}^{\gamma_f} \|\nabla \varphi_{i, n}^\nu(t_2)\|_{\xLn{\frac{\beta + 1}{\beta + 1 - \gamma_f}}(\O)} \\ 
& + a_b \|\rho_n\|_{\xLn{\beta + 1}(\O)}^{\gamma_b} \|\nabla \varphi_{i, n}^\nu(t_2)\|_{\xLn{\frac{\beta + 1}{\beta + 1 - \gamma_b}}(\O)} + \delta \|\rho_n\|_{\xLn{\beta + 1}(\O)}^\beta \|\nabla \varphi_{i, n}^\nu(t_2)\|_{\xLn{\beta + 1}(\O)} \\ 
& + \varepsilon \|\nabla \rho_n \cdot \nabla u_n \|_{\xLn{\frac{5 \beta-3}{4 \beta}}(\O)} \|\varphi_{i, n}^\nu(t_2)\|_{\xLn{\frac{5 \beta - 3}{\beta -3}}(\O)} + \|\rho_n g\|_{\xLn{\beta + 1}(\O)} \|\varphi_{i, n}^\nu(t_2)\|_{\xLn{\frac{\beta + 1}{\beta}}(\O)} \\ 
& + \frac{2 \kappa_b}{R_0} \|\nabla \varphi_{i, n}^\nu(t_2)\|_{\xLone(\O)} \xdif \tau, 
\end{aligned} 
\label{eq:high_pen:VoL_estimate}
\end{equation}
where for $ 1 \leq p < \infty$, by Young's inequality,
\begin{align*}
\|\varphi_{n, i}^\nu\|_{\xWn{1, \infty}(0, T; \xLn{p}(\O))} \leq \|\x1_{\cB(x_n, (1 + \nu) R_n)}\|_{\xLinfty((0, T) \times \O)} \|\cK^\nu\|_{\xWn{1, 1}(\O)} \leq \|\cK^\nu\|_{\xWn{1, 1}(\O)}.
\end{align*} 
On the other hand,
\begin{equation*}
\begin{aligned} 
 (\rho_n u_n(t_2), \varphi_n^\nu(t_2) - \varphi_n^\nu(t_1))_\O  = & \left(\rho_n u_n(t_2) \cdot e_i, \int_{t_1}^{t_2} \cK^\nu * \partial_t \x1_{\cB(x_n, (1 + \nu) R_n)} \xdif \tau \right)_\O \\
=&  - \int_{t_1}^{t_2} \bigg(\rho_n u_n(t_2) \cdot e_i, \cK^\nu*\left(\Pi_n u_n \cdot \nabla \x1_{\cB(x_n, (1 + \nu) R_n)}\right) \bigg)_\O  \xdif \tau \\
=& - \int_{t_1}^{t_2} \bigg( \rho_n u_n(t_2) \cdot e_i, \nabla \cK^\nu*\left(\Pi_n u_n \x1_{\cB(x_n, (1 + \nu) R_n)}\right) \bigg)_\O \xdif \tau  \\ 
& + \int_{t_1}^{t_2} \bigg(\rho_n u_n(t_2) \cdot e_i, \cK^\nu*\left(\xdiv (\Pi_n u_n) \x1_{\cB(x_n, (1 + \nu) R_n)}\right) \bigg)_\O \xdif \tau.
\end{aligned}
\end{equation*}
Consequently, using Holder's and Young's inequality, we obtain
\begin{multline*}
| (\rho_n u_n(t_2), \varphi_n^\nu (t_2) - \varphi_n^\nu (t_1))_\O | \\ 
 \leq \|\rho_n u_n\|_{\xLinfty(0, T; \xLn{\frac{2\beta}{\beta + 1}}(\O))} \|\x1_{\cB(x_n, (1 + \nu) R_n)}\|_{\xLinfty((0, T)\times \O)} \|\cK^\nu\|_{\xWn{1, 1}(\O)} \int_{t_1}^{t_2} \|\Pi_n u_n\|_{\xWn{1, \frac{2\beta}{2\beta - 1}}(\O)} \xdif \tau. 
\end{multline*} 
We deduce that $\tilde {V}_{i, \nu}^n$ is uniformly equicontinuous on $(0,T)$ uniformly in $n$ (although not in $\nu$). The Arzela-Ascoli theorem implies that there exists a non-decreasing function $\phi_\nu : \xN \rightarrow \xN$ and $\tilde {V}_i^\nu \in \xCzero([0, T])$ such that
\begin{equation*}
\tilde {V}_{i, \phi_\nu(n)}^\nu \xrightarrow[]{n \rightarrow \infty} \tilde {V}_i^\nu \; \mbox {in} \; \xCzero ([0, T]). 
\end{equation*} 
This concludes the proof of the lemma.
\end{proof} 
The next lemma addresses the convergence of $\{\tilde {V}_n\}$, $\{\tilde {\omega}_n\}$ and $\{\tilde {\Lambda}_n \}$. 
\begin{lmm}
Up to a subsequence, for all $i \in \{1, 2, 3\}$,  
\begin{equation*}
\tilde {V}_{i, n} \rightarrow \tilde {V}_{i, \varepsilon}, \quad
\tilde {\omega}_{i, n} \rightarrow \tilde {\omega}_{i, \varepsilon}, \quad
\tilde {\Lambda}_n \rightarrow \tilde {\Lambda}_\varepsilon \; \mbox {in} \; \xCzero([0, T]). 
\end{equation*} 
\label{lem:high_pen:VoL_convergence}
\end{lmm}
\begin{proof}[Proof of Lemma~\ref{lem:high_pen:VoL_convergence}]
Similarly to Lemma~\ref{lem:high_pen:tdV_tdo_tdL_convergence}, we focus on the proof for $\tilde {V}_n$, as the proof is analogous for $\tilde{\omega}_n$ and $\tilde{\Lambda}_n$. Let $i \in \{1, 2, 3\}$ and consider the regularized sequence $\{\tilde {V}_{i, n}^\nu\}$. By construction, we have
\begin{equation*}
\tilde {V}_{i, n}^\nu \xrightarrow[]{\sigma \rightarrow 0}  \tilde {V}_{i, n} \; \mbox{in} \; \xCzero([0, T])  \; \mbox{uniformly in} \; n.
\end{equation*}
In particular, for any $\theta > 0$, there exists $\nu > 0$ sufficiently small such that for all $n \in \xN$, 
\begin{equation*}
\|\tilde {V}_{i, n}^\nu - \tilde {V}_{i, n}\|_{\xLinfty(0, T)} \leq \theta/3.
\end{equation*} 
By Lemma~\ref{lem:high_pen:tdV_tdo_tdL_convergence}, there exists $n_0(\nu) \in \xN$ such that for $m, l \geq n_0(\nu)$, 
\begin{equation*} 
\|V_{i, \phi_\nu(m)}^\nu - V_{i, \phi_\nu(l)}^\nu \|_{\xLinfty(0, T)} \leq  \theta/3.
\end{equation*} 
Consequently, for $m, l \geq n_0(\nu)$,  
\begin{multline}
\|\tilde {V}_{i, \phi_\nu(m)} - \tilde {V}_{i, \phi_\nu(l)}\|_{\xLinfty(0, T)} \\ 
\leq \|\tilde {V}_{i, \phi_\nu(m)} - \tilde {V}_{i,\phi_\nu(m)}^\nu \|_{\xLinfty(0, T)} + \|\tilde {V}_{i, \phi_\nu(m)}^\nu - \tilde {V}_{i, \phi_\nu(l)}^\nu \|_{\xLinfty(0, T)} +  \| \tilde {V}_{i, \phi_\nu(l)}^\nu - \tilde {V}_{i, \phi_\nu(l)}\|_{\xLinfty(0, T)} \leq \theta.
\label{eq:high_pen:tVip_tri_ineq}
\end{multline}
Applying \eqref{eq:high_pen:tVip_tri_ineq} with $\theta = 1, \frac12, \frac14, \cdots $ and using a standard diagonal argument, we can extract a subsequence $\tilde {V}_{i, \psi(n)}$ that is Cauchy in $\xCzero([0, T])$ and hence converges in $\xCzero([0, T])$. To identify the limit, we consider $\varphi \in \cD(0, T)$. Using the weak convergence of $\rho_n u_n$ given in \eqref{eq:high_pen:cv_rhonun}, the strong convergence of $\chi_n$ given in \eqref{eq:high_pen:chin_convergence}, we have
\begin{equation*}
\int_0^T \varphi (\rho_n u_n, \chi_n e_i)_\O \xdif \tau \rightarrow \int_0^T \varphi (\rho_\varepsilon u_\varepsilon, \chi_\varepsilon e_i)_\O \xdif \tau. 
\end{equation*} 
This identifies the limit as $\tilde{V}_{i,\varepsilon}$ and concludes the proof of the lemma.   
\end{proof}
From the convergences of the sequences $\{V_n\}$, $\{\omega_n\}$ and $\{\Lambda_n\}$ given in \eqref{eq:high_pen:VoL_weak_convergence} and Lemma~\ref{lem:high_pen:VoL_convergence}, we deduce 
\begin{equation*}
V_n \cdot \tilde {V}_n \rightarrow V_\varepsilon \cdot \tilde {V}_\varepsilon, \quad \omega_n \cdot \tilde {\omega}_n \rightarrow \omega_\varepsilon \cdot \tilde {\omega}_\varepsilon, \quad \Lambda_n \tilde {\Lambda}_n \rightarrow \Lambda_\varepsilon \tilde {\Lambda}_\varepsilon \; \mbox{in} \; \xLtwo(0, T),   
\end{equation*} 
and there exists $n_2$ such that for all $n$ satisfying  $n \geq \{n_0, n_2\}$,   
\begin{equation*}
|M_5(\sigma, n)| \leq \frac{\theta}{5}.
\end{equation*} 
For the last term $M_1(\sigma, n)$, we use a result presented in \cite [Section 7]{Feireisl2003ARM} in the context of fluid-solid interaction. Let $\sigma_m = \min \{\sigma_0, \sigma_1\}$, we have 
\begin{equation*}
\rho_n u_n \otimes u_n \rightarrow \rho_\varepsilon u_\varepsilon \otimes u_\varepsilon \; \mbox {in} \; \xLn{\frac{6\beta}{4\beta + 3}}(Q_{f ,\varepsilon}^{\sigma_m})
\end{equation*}
where 
\begin{equation*}
Q_{f, \varepsilon}^{\sigma_m} = (0, T) \backslash \overline {Q_\varepsilon^{\sigma_m}}, \quad \overline{Q_\varepsilon^{\sigma_m}} = \{ (t, x) \in (0, T) \times \O \; | \; x \in \overline {\cB_\varepsilon^{\sigma_m}(t)} \}. 
\end{equation*}
Consequently, there exists $n_3 \in \xN$, possibly depending on $\sigma_m$, such that for all $n \geq n_3$, 
\begin{equation*}
|M_5(\sigma_m, n)| \leq \theta/5.
\end{equation*}
Finally, we conclude that for $n \geq \max \{n_0, n_1, n_2, n_3\}$,  
\begin{equation*}
\left| \displaystyle \int_0^T \int_\O \rho_n u_n \otimes u_n : \xD (\varphi) \xdif y \xdif \tau - \displaystyle \int_0^T \int_\O \rho_\varepsilon u_\varepsilon \otimes u_\varepsilon : \xD (\varphi) \xdif y \xdif \tau \right| \leq \theta, 
\end{equation*} 
which leads to 
\begin{equation*}
 \displaystyle \int_0^T \int_\O \rho_n u_n \otimes u_n : \xD (\varphi) \xdif y \xdif \tau \rightarrow \displaystyle \int_0^T \int_\O \rho_\varepsilon u_\varepsilon \otimes u_\varepsilon : \xD (\varphi) \xdif y \xdif \tau. 
\end{equation*} 
The convergence results for the pressure terms in the momentum equation  and the energy estimate are derived using arguments similar to those presented in Section~\ref{subsec:cvg_faedo-glrk_approximation}. Finally, we follow the same ideas as the calculations \eqref{eq:cvg_fglrk:dist_eta_bound}-\eqref{eq:faedo-glrk:distance_T_constraint} to conclude that 
\[\xdist (\cB_n(t), \partial \O) \geq 2 \sigma \quad \forall \, t \in [0, T]. \]

\subsection{Vanishing dissipation in the continuity equation and the limiting system} \label{subsec:vnsh_dspt_continuity_equation}
Proposition~\ref{prop:appx_sol:existence_n_solution} establishes the existence of weak solutions $(\rho_\varepsilon, u_\varepsilon, \cB_\varepsilon)$ to \eqref{eq:appx_sol:Qe_def}-\eqref{eq:appx_sol:mue_nue_definition}. In this section, we prove Proposition~\ref{prop:appx_sol:existence_delta_level} by passing to the limit $\varepsilon \rightarrow 0$ in the $\varepsilon$-level approximation problem. A key step of the analysis consists in identifying the pressure corresponding to the limiting system. Given initial data $(\rho_{0, \delta}, u_{0, \delta})$ for the $\delta$-level approximation problem, we construct initial data $(\rho_{0, \varepsilon}, u_{0, \varepsilon})$ for the $\varepsilon$-level approximation problem such that 
\begin{equation*}
\rho_{0,  \varepsilon} > 0, \quad \rho_{0, \varepsilon} \in \xWn{1, \infty}(\O), \quad \rho_{0, \varepsilon} \rightarrow \rho_{0, \delta} \; \mbox{in} \; \xLbeta(\O), \quad q_{0, \varepsilon} \rightarrow q_{0, \delta} \; \mbox{in} \; \xLn{\frac{2\beta}{\beta + 1}}(\O), 
\end{equation*} 
and 
\begin{equation*}
E_\delta(\rho_{0, \varepsilon}, q_{0, \varepsilon}, \x1_{\cB_0}) \rightarrow E_\delta(\rho_{0, \delta}, q_{0, \delta}, \x1_{\cB_0}), 
\end{equation*} 
where the definition of $E_\delta$ is given in  \eqref{eq:appx_sol:Ed_definition}.\\

The energy estimate \eqref{eq:appx_sol:energye_estimate} yields, up to a subsequence,  
\begin{equation}
u_\varepsilon \rightarrow u_\delta \; \mbox{weakly in} \; \xLtwo(0, T; \xHone(\O)). \label{eq:vnsh_dspt:compacity_results_ue}  
\end{equation}
It then follows from Proposition~\ref{prop:smlrt-prop:weak_sequential_continuity} that up to a subsequence, for all $1 \leq p < \infty$, 
\begin{equation}
\x1_{\cB_\varepsilon} \rightarrow \x1_{\cB_\delta}  \; \mbox{weakly--* in} \; \xLinfty((0, T) \times \O), \; \mbox{strongly in} \; \xCzero([0, T]; \xLn{p}_{\xloc}(\xR^3)),
\label{eq:vnsh_dspt:1Be_cvg} 
\end{equation}
where the pair $(\x1_{\cB_\delta}, u_\delta)$ satisfies
\begin{equation*}
\partial_t \x1_{\cB_\delta} + u_\delta \cdot \nabla \x1_{\cB_\delta} = 0 \; \mbox{in} \; \cD^\prime((0, T) \times \xR^3), \quad \cB_\delta(0) = \cB_0.
\end{equation*}
Moreover, the sequences $\{u_\varepsilon\}$ and $\{\eta_\varepsilon\}$ satisfy the assumptions of Corollary~\ref{coro:smlrt_prop:stability_compatible_velocities}. It then follows that    
\begin{equation}
\eta_\varepsilon \rightarrow \eta_\delta \; \mbox{weakly in} \; \xHone([0, T]; \cM), \; \mbox{strongly in} \; \xCzero([0, T]; \cM), 
\label{eq:vnsh_dspt:etae_cvg}  
\end{equation}
and, if we define $$\cB_\delta(t) = \eta_\delta[t](\cB_0) \quad \forall \, t \in [0, T],$$ then the velocity field $u_\delta$ is compatible with the system $\{ \overline {\cB_\delta}, \eta_\delta \}$. 

Let $x_\delta, R_\delta \in \xHone([0, T])$ such that $\cB_\delta = \cB(x_\delta, R_\delta)$, we also have $$ R_\delta(t) \geq \frac{R_0}{2} \quad \forall \, t \in [0, T].$$ From the convergence of $\eta_\varepsilon$ given in \eqref{eq:vnsh_dspt:etae_cvg}, it follows that, for $\varepsilon > 0$ sufficiently small, 
\begin{equation*}
\cB_\delta(t) \subset \cB_\varepsilon^\frac{\sigma}{2}(t) \quad \forall \, t \in [0, T]. 
\end{equation*} 
In particular, since $\cB_\varepsilon$ satisfies $\xdist(\cB_\varepsilon, \partial \O) \geq 2 \sigma$ in $[0, T]$, we deduce that 
\[ \xdist(\cB_\delta(t), \partial \O) \geq 3 \sigma/2  \quad \forall \, t \in [0, T]. \] 

Following the same arguments as in \cite[Section 3.3]{Feireisl2001JMF}, \cite[Section 7.9.1]{Novotny2004OUP}, we get  
\begin{align}
&\rho_\varepsilon \rightarrow \rho_\delta \; \mbox{in} \; \xCzero([0, T];  \xLn{\beta}_{\xweak} (\O)), \label{eq:vnshg_dspt:rhoe_cvg} \\ 
&\rho_\varepsilon u_\varepsilon \rightarrow \rho_\delta u_\delta \; \mbox{weakly--* in} \; \xLinfty(0, T; \xLn{\frac{2\beta}{\beta + 1}}(\O)),  \label{eq:vnshg_dspt:rhoeue_cvg} \\ 
& \varepsilon \nabla \rho_\delta \rightarrow 0 \; \mbox{strongly in} \; \xLtwo((0, T) \times \O).\label{eq:vnsh_dspt:compacity_results_nrhoe}
\end{align}
These convergences allow us to pass to the limit in the continuity equation \eqref{eq:appx_sol:rhoe_continuity_equation}. In addition, the couple $(\rho_\delta, u_\delta)$ satisfies the renormalized continuity equation \eqref{eq:intro:rho_renormalized_continuity_equation} for all $b$ satisfying \eqref{eq:intro:b_cdt_renormalized_solution}.  Finally, by Proposition~\ref{lem:smlrt-prop:density_bubble}, the density $\rho_\delta$ is compatible with the system $\{\overline{\cB_\delta}, \eta_\delta\}$. \\

We now address the convergence of the momentum equation \eqref{eq:appx_sol:rhoeue_momentum_equation}. Let $\varphi \in {\cT} (Q_\delta)$. Since $\overline {Q_\delta}$ is a compact subset of $(0, T)\times \O$, an argument analogous to that of Section~\ref{subsec:cvg_high_penalization_term_momentum_equation} yields the existence of $\sigma_0 > 0$ and $\varphi_b \in \xCzero([0, T]; \cS)$ such that, for all $ t \in [0, T]$, 
\begin{equation*}
\varphi(t, \cdot) = \varphi_b(t, \cdot) \; \mbox{in} \; \cB_\delta^{\sigma_0}(t)
\end{equation*} 
where, for all $\sigma > 0$, $\cB_\delta^\sigma = \cB(x_\delta, R_\delta + \sigma)$. In particular, for $0 < \sigma \leq \sigma_0$, the following identity holds:
\begin{equation}
\varphi = (1 - \x1_{\cB_\delta^\sigma})\varphi + \x1_{\cB_\delta^\sigma} \varphi_b.
\label{eq:vnsh_dspt:varphi_decomposition}
\end{equation}
Using once more the strong convergence of $\eta_\varepsilon$ given in \eqref{eq:vnsh_dspt:etae_cvg}, we deduce that there exists $\varepsilon_0 > 0$ such that for $0 \leq \varepsilon \leq \varepsilon_0$, 
\begin{equation*}
\cB_\varepsilon(t) \subset \cB_\delta^{\sigma_0}(t) \quad \forall \, t \in [0, T], 
\end{equation*} 
and consequently 
\begin{equation*}
\x1_{\cB_\varepsilon} \x1_{\cB_\delta^{\sigma_0}} = \x1_{\cB_\varepsilon}, \quad (1 - \x1_{\cB_\varepsilon}) (1 - \x1_{\cB_\delta^{\sigma_0}}) = (1 - \x1_{\cB_\delta^{\sigma_0}}), \quad \x1_{\cB_\varepsilon}  (1 - \x1_{\cB_\delta^{\sigma_0}}) = 0.
\end{equation*} 
We now look at the convergence of 
\begin{equation}
\int_0^T \int_\O \rho_\varepsilon u_\varepsilon \otimes u_\varepsilon : \xD(\varphi) \xdif y \xdif \tau.
\label{eq:vnsh_dspt:conv_term_f}
\end{equation}
When details are omitted, the arguments for the convergence of \eqref{eq:vnsh_dspt:conv_term_f} are analogous to those presented in Section~\ref{subsec:cvg_high_penalization_term_momentum_equation}. For $0 < \sigma \leq \sigma_0$, we write 
$$\int_0^T \int_\O \rho_\varepsilon u_\varepsilon \otimes u_\varepsilon : \xD(\varphi) \xdif y \xdif \tau -\int_0^T \int_\O \rho_\delta u_\delta \otimes u_\delta : \xD(\varphi) \xdif y \xdif \tau = \sum_{i=1}^4 M_i(\sigma, \varepsilon)$$ where 
\begin{equation*}
\begin{aligned}
& M_1(\sigma, \varepsilon) =  \int_0^T \int_\O (1 - \x1_{\cB_\delta^\sigma}) \rho_\varepsilon u_\varepsilon \otimes u_\varepsilon : \xD (\varphi) \xdif y \xdif \tau  - \int_0^T \int_\O (1 - \x1_{\cB_\delta^\sigma}) \rho_\delta u_\delta \otimes u_\delta : \xD (\varphi) \xdif y \xdif \tau, \\ 
& M_2(\sigma, \varepsilon) = \int_0^T \int_\O (\x1_{\cB_\delta^\sigma} - \x1_{\cB_\delta}) \rho_\varepsilon u_\varepsilon \otimes u_\varepsilon : \xD (\varphi)  \xdif y \xdif \tau + \int_0^T \int_\O (\x1_{\cB_\delta^\sigma} - \x1_{\cB_\delta}) \rho_\delta u_\delta \otimes u_\delta : \xD (\varphi)  \xdif y \xdif \tau, \\
& M_3(\sigma, \varepsilon) = \int_0^T \int_\O (\x1_{\cB_\delta} - \x1_{\cB_\varepsilon}) \rho_\varepsilon u_\varepsilon \otimes u_\varepsilon : \xD (\varphi) \xdif y \xdif \tau, \\ 
& M_4(\sigma, \varepsilon) = \int_0^T \int_\O \x1_{\cB_\varepsilon} \rho_\varepsilon u_\varepsilon \otimes  u_\varepsilon : \xD (\varphi)  \xdif y \xdif \tau - \int_0^T \int_\O \x1_{\cB_\delta} \rho_\delta u_\delta \otimes u_\delta : \xD (\varphi) \xdif y \xdif \tau.
\end{aligned} 
\end{equation*}
Let $\theta > 0$, there exists $\sigma_1, \varepsilon_1 >$ such that for $0 < \sigma \leq \min\{\sigma_0, \sigma_1\}$ and $0 < \varepsilon \leq \min\{\varepsilon_0, \varepsilon_1\}$, 
$$ |M_2(\sigma, \varepsilon)| \leq \theta/4, \quad |M_3(\sigma, \varepsilon)| \leq \theta/4. $$ 
Arguing as in Lemma~\ref{lem:high_pen:VoL_convergence}, we get for $i\in \{1,2,3\}$:
\begin{align}
&(\rho_\varepsilon u_\varepsilon, \x1_{\cB_\varepsilon} e_i)_\O \xrightarrow[]{\varepsilon \rightarrow 0} (\rho_\delta u_\delta , \x1_{\cB_\delta} e_i)_\O \; \mbox{in} \; \xCzero([0, T]), \label{eq:cvg_vnsg_dspt:V_convergence} 
\\  
& (\rho_\varepsilon u_\varepsilon, \x1_{\cB_\varepsilon}  e_i \times (x - x_\varepsilon))_\O \xrightarrow[]{\varepsilon \rightarrow 0} (\rho_\delta u_\delta,\x1_{\cB_\delta}  e_i \times  (x - x_\delta))_\O \; \mbox {in} \; \xCzero([0, T]),\label{eq:cvg_vnsg_dspt:o_convergence} \\
& (\rho_\varepsilon u_\varepsilon,  \x1_{\cB_\varepsilon} (x - x_\varepsilon))_\O \xrightarrow[]{\varepsilon \rightarrow 0} \left(\rho_\delta u_\delta,  \x1_{\cB_\delta} (x - x_\delta) \right)_\O \; \mbox{in} \; \xCzero([0, T]).  \label{eq:cvg_vnsg_dspt:L_convergence}
\end{align}
Using the convergence \eqref{eq:cvg_vnsg_dspt:V_convergence}-\eqref{eq:cvg_vnsg_dspt:L_convergence}, we obtain  
\begin{equation*}
\lim \limits_{\varepsilon \rightarrow 0} \int_0^T \int_\O \xdiv (\varphi_b) \x1_{\cB_\varepsilon}  \rho_\varepsilon |u_\varepsilon|^2 \xdif y \xdif \tau = \int_0^T \int_\O \xdiv (\varphi_b) \x1_{\cB_\delta} \rho_\delta |u_\delta|^2 \xdif y \xdif \tau.
\end{equation*} 
Consequently, there exists $\varepsilon_2 > 0$ such that for all $0 < \varepsilon \leq \varepsilon_2$,
\begin{equation*}
|M_4(\sigma, \varepsilon)|\leq \theta/4
\end{equation*}
Finally, setting $\sigma_m = \min \{\sigma_0, \sigma_1\}$ and following similar arguments to the ones presented in \cite[Section 8]{Feireisl2003ARM} and already presented in Section~\ref{subsec:cvg_high_penalization_term_momentum_equation}, we can prove that 
\begin{equation}
\rho_\varepsilon u_\varepsilon \otimes u_\varepsilon \rightarrow \rho_\delta u_\delta \otimes u_\delta \; \mbox{in} \; \xLtwo(Q_{f, \delta}^{\sigma_m})
\label{eq:vnsh_dspt:rhoueue_kf}
\end{equation} 
where $$ Q_{f, \delta}^{\sigma_m} = (0, T) \backslash \overline {Q_\delta^{\sigma_m}}, \quad \overline {Q_\delta^{\sigma_m}} = \{(t, x) \in (0,T) \times \O \; | \; x \in  \overline {\cB_\delta^{\sigma_m}(t)} \}. $$ Consequently, there exists $\varepsilon_3 > 0$, possibly depending on $\sigma_m$, such that for $0 < \varepsilon \leq \varepsilon_3$, 
\begin{equation*}
|M_5(\sigma_m, \varepsilon)| \leq \theta/4.
\end{equation*} 
Finally, for $0 < \varepsilon \leq \min \{\varepsilon_1, \varepsilon_2, \varepsilon_3\}$, 
\begin{equation} 
\left|\int_0^T \int_\O \rho_\varepsilon u_\varepsilon \otimes u_\varepsilon : \xD(\varphi) \xdif y \xdif \tau - \int_0^T \int_\O \rho_\delta u_\delta \otimes u_\delta : \xD (\varphi) \xdif y \xdif \tau \right| \leq \theta,  
\label{eq:vnsh_dspt:rhoueue_rhodudud_theta}
\end{equation} 
which leads to
\begin{equation} 
\int_0^T \int_\O \rho_\varepsilon u_\varepsilon \otimes u_\varepsilon : \xD(\varphi) \xdif y \xdif \tau \xrightarrow[]{} \int_0^T \int_\O \rho_\delta u_\delta \otimes u_\delta : \xD (\varphi) \xdif y \xdif \tau. 
\label{eq:vnsh_dspt:cvg_rhoeueue}
\end{equation} \\ 

We now aim to identify the pressure term. First, we first establish that $\{\rho_\varepsilon\}$ is locally uniformly bounded in $\xLn{\beta + 1}(Q_{f, \delta})$, 
where
\begin{equation*}
Q_{f, \delta} = (0, T) \times\Omega \setminus \overline {Q_\delta}
\end{equation*}
with $\overline {Q_\delta}$ defined by \eqref{eq:appx_sol:Qd_definition}. 
\begin{lmm}
For any compact $\cK_f \subset Q_{f, \delta}$, there exists a constant $c$ independent of $\varepsilon$ such that 
\begin{equation}
\|\rho_\varepsilon\|_{\xLn{\beta + 1}(\cK_f)} \leq c (\cK_f, E_\delta(\rho_{0, \varepsilon}, q_{0, \varepsilon}, \x1_{\cB_0}) , \|g\|_{\xLinfty((0, T)  \times \xR^3)}, \kappa_b).  
\label{eq:vnsh_dspt:pressure_estimate}
\end{equation}
\label{lemma:vnsh_dspt:pressure_estimate}
\end{lmm}
\begin{proof}
The proof follows the same strategy as in \cite[Lemma 8.1]{Feireisl2003ARM}, \cite[Lemmma 3.1]{Feireisl2001JMF}. It relies on testing the momentum equation \eqref{eq:appx_sol:rhoeue_momentum_equation} with functions of the form $$ \varphi = \psi(t)\, \cB \left(\rho_\varepsilon - |\cK_f|^{-1} \int_{\cK_f} \rho_\varepsilon \mathrm{d} y \right),$$
where $\psi\in \cD(0,T)$ and $\cB$ denotes the Bogovskiĭ operator on $\cK_f$.
\end{proof}
Applying Lemma~\ref{lemma:vnsh_dspt:pressure_estimate}, we deduce that 
\begin{align}
& \rho_\varepsilon^\beta \rightarrow \overline {\rho_\delta^\beta} \; \mbox{weakly in} \; \xLn{\frac{\beta + 1}{\beta}}(\cK_f), 
\label{eq:vnsh_dspt:compacity_results_rhoeb} \\
& \rho_\varepsilon^{\gamma_f} \rightarrow \overline {\rho_\delta^{\gamma_f}} \; \mbox{weakly in} \; \xLn{\frac{\gamma_f + 1}{\gamma_f}}(\cK_f), 
\label{eq:vnsh_dspt:compacity_results_rhoeg}
\end{align}
for any compact $\cK_f \subset Q_{f, \delta} $. Together with \eqref{eq:vnsh_dspt:compacity_results_ue},\eqref{eq:vnshg_dspt:rhoe_cvg},   and \eqref{eq:vnsh_dspt:rhoueue_kf}, the convergences  \eqref{eq:vnsh_dspt:compacity_results_rhoeb}-\eqref{eq:vnsh_dspt:compacity_results_rhoeg} allow us to pass to the limit $\varepsilon \rightarrow 0$ in \eqref{eq:appx_sol:rhoeue_momentum_equation}. Consequently, for any test function $\varphi \in {\cD} (Q_{f, \delta})$, we obtain
\begin{multline*}
\int_0^T \int_{\xR^3} (\rho_\delta u_\delta) \partial_t \varphi + (\rho_\delta u_\delta \otimes u_\delta):\xD(\varphi) + (a_f \overline {\rho_\delta^{\gamma_f}} + \delta \overline {\rho_\delta^\beta}) \xdiv (\varphi) \xdif y \xdif \tau \\ 
=  \int_0^T \int_{\xR^3} \xT_f (u_\delta) : \xD (\varphi)  - \rho g \cdot \varphi  \xdif y \xdif \tau.
\end{multline*}
The next step consists in establishing the strong convergence of the density. Specifically, we aim to identify the nonlinear pressure term  $(a_f \overline {\rho_\delta^{\gamma_f}} + \delta \overline {\rho_\delta^\beta})$ by proving that $\overline {\rho_\delta^{\gamma_f}} = \rho_\delta^{\gamma_f}$ and $\overline {\rho_\delta^\beta} = \rho_\delta^\beta$. To this end, we introduce the effective viscous flux $$a_f \rho_\varepsilon^{\gamma_f} + \delta \rho_\varepsilon^\beta  - \left(\frac43 \mu_f + \nu_f\right) \xdiv (u_\varepsilon).$$ which exhibits improved compactness properties compared to the pressure or velocity fields separately.
\begin{lmm} For any test function $\varphi \in \cD(Q_{f, \delta})$, 
\begin{multline*}
\lim \limits_{\varepsilon \rightarrow 0} \int_0^T \int_{\xR^3} \varphi \left(a_f \rho_\varepsilon^{\gamma_f} + \delta \rho_\varepsilon^\beta - \left(\nu_f + \frac43 \mu_f\right) \mathrm {div} \, (u_\varepsilon)\right) \rho_\varepsilon \xdif y \xdif \tau \\ 
= \int_0^T \int_{\xR^3} \varphi \left(a_f \overline {\rho_\delta^{\gamma_f}} + \delta \overline {\rho_\delta^\beta} - \left(\nu_f + \frac43 \mu_f\right) \mathrm {div} (u_\delta)\right) \rho_\delta \xdif y \xdif \tau 
\end{multline*} 
\label{lemma:vnsh_dspt:cvg_effective_viscous_flux}
\end{lmm} 
\begin{proof}
The proof of this Lemma can be found in  \cite[Theoreme 5.1, Appendix B]{lions1998O}, \cite[Lemma 3.2]{Feireisl2001JMF},  \cite[Lemma 8.2]{Feireisl2003ARM}, \cite[Lemma 7.50]{Novotny2004OUP}.
\end{proof}
Applying Lemma~\ref{lemma:vnsh_dspt:cvg_effective_viscous_flux} and the monotonicity of the mappings $\rho \rightarrow \rho^{\gamma_f}$, $\rho \rightarrow \rho^{\beta}$, we get 
\begin{equation}
\overline {\rho_\delta \xdiv (u_\delta)} \geq \rho_\delta \xdiv (u_\delta) \; \mbox {on} \; Q_{f, \delta}, 
\label{eq:vnsh_dspt:rhodivu_limits_inequality_Qfd}  
\end{equation} 
where $\overline {\rho_\delta \xdiv (u_\delta)}$ is defined by
\begin{equation*}
\rho_\varepsilon \xdiv (u_\varepsilon) \rightarrow \overline {\rho_\delta \xdiv (u_\delta)} \; \mbox {weakly in} \; \xLone((0, T) \times \O). 
\end{equation*} 
Moreover, on any compact $\cK_b \subset Q_{b, \delta}$, there exists $ \varepsilon_0$, possibly depending on $\cK_b$, such that  for $0 < \varepsilon \leq \varepsilon_0$,  
\begin{equation}
\xdiv (u_\delta) = \Lambda_\delta, \; \xdiv (u_\varepsilon) = \Lambda_\varepsilon \;  \mbox{and} \; \Lambda_\delta \rightarrow \Lambda_\varepsilon \; \mbox{weakly in} \; \xLtwo(0, T). 
\label{eq:vnsh_dspt:divue_cvg_kb} 
\end{equation}
From the convergence of $\rho_\varepsilon$ given in \eqref{eq:vnshg_dspt:rhoe_cvg} and of $\xdiv (u_\varepsilon)$ given in  \eqref{eq:vnsh_dspt:divue_cvg_kb}, we deduce 
\begin{equation*}
\overline {\rho_\delta \xdiv (u_\delta)} = \rho_\delta \xdiv (u_\delta) \; \mbox {in} \;  Q_\delta.
\end{equation*}
Finally, we obtain
\begin{equation}
\overline {\rho_\delta \xdiv (u_\delta)} \geq \rho_\delta \xdiv (u_\delta) \; \mbox {on} \;  (0, T) \times \xR^3.
\label{eq:vnsh_dspt:rhodivu_limits_inequality} 
\end{equation} 
The remaining arguments follow standard results from the theory of compressible fluids and can be found in \cite{lions1998O}, \cite{Feireisl2001JMF}[Section 3.5], \cite{Novotny2004OUP}[Lemma 7.51], \cite{Feireisl2003ARM}[Section 8]. For completnesss, we outline the main steps. Noting that $b(s) = s \log(s)$ satisfies the condition \eqref{eq:intro:b_cdt_renormalized_solution} and $\rho_\delta$ satisfies the renormalized continuity equation \eqref{eq:intro:rho_renormalized_continuity_equation}, we deduce, for all $t \in [0, T]$,
\begin{equation}
\int_0^t \int_\O \rho_\delta \xdiv (u_\delta) \xdif y \xdif \tau = \int_\O \rho_{0, \delta} \log (\rho_{0, \delta}) \xdif y - \int_\O \rho_\delta(t) \log(\rho_\delta(t)) \xdif y.
\label{eq:vnsh_dspt:integration_renormalized_continuity_equation}
\end{equation} 
Since $\rho_\varepsilon$ satisfies the regularized continuity equation \eqref{eq:appx_sol:rhoe_continuity_equation} almost everywhere on $(0, T) \times \O$, it follows that,  for any convex $b$, globally Lipschitz on $[0, \infty)$,   
\begin{equation*}
\partial_t b(\rho_\varepsilon) + \xdiv (b(\rho_\varepsilon) u_\varepsilon) + (b^\prime (\rho_\varepsilon) \rho_\varepsilon - b(\rho_\varepsilon)) \xdiv (u_\varepsilon)- \varepsilon \Delta b (\rho_\varepsilon)\leq 0.
\end{equation*} 
In particular taking again $b(s) = s \log(s)$ and integrating on $(0, t) \times \O$, we obtain
\begin{equation}
\int_0^t \int_\O \rho_\varepsilon \xdiv (u_\varepsilon) \xdif y \xdif \tau \leq \int_\O \rho_{0, \varepsilon} \log(\rho_{0, \varepsilon}) \xdif y - \int_\O \rho_\varepsilon(t) \log (\rho_\varepsilon(t)) \xdif y. 
\label{eq:vnsh_dspt:integration_regularized_continuity_equation}
\end{equation} 
The relations \eqref{eq:vnsh_dspt:rhodivu_limits_inequality}, \eqref{eq:vnsh_dspt:integration_renormalized_continuity_equation} together with \eqref{eq:vnsh_dspt:integration_regularized_continuity_equation} lead to 
\begin{equation*}
\limsup_{\varepsilon \rightarrow 0} \int_\O \rho_\delta (t) \log (\rho_\delta(t)) \xdif y \leq \int_\O \rho_\delta(t) \log (\rho_\delta(t)) \xdif y \; \forall \, t \in [0, T]
\end{equation*} 
which implies, for all $1 \leq p < \beta + 1$,  
\begin{equation}
\rho_\varepsilon \rightarrow \rho_\delta \; \mbox {in} \; \xLn{p}((0, T) \times \O).
\label{eq:vnsh_dspt:rhoe_stg_cvg}
\end{equation} 
Consequently, $\overline {\rho_\delta^{\gamma_f}} + \delta \overline {\rho_\delta^\beta} = \rho_\delta^{\gamma_f} + \delta \rho_\delta^\beta$ on $(0, T) \times \O.$  \\ 

Finally, using the following convergences  
\begin{align*}
& \rho_\varepsilon u_\varepsilon \otimes u_\varepsilon \rightarrow \rho_\delta u_\delta \otimes u_\delta \; \mbox{weakly in} \; \xLtwo(0, T; \xLn{\frac{6\beta}{4\beta + 3}}(\O)),  \\ 
& \rho_\varepsilon \rightarrow \rho_\delta \; \mbox{strongly in} \; \xLn{p}((0, T) \times \O) \quad \forall \, 1 \leq p < \beta + 1, 
\end{align*} 
we have 
\begin{equation*}
\begin{aligned}
& \int_\O \rho_\varepsilon |u_\varepsilon|^2 \xdif y  \rightarrow \int_\O \rho_\delta |u_\delta|^2 \xdif y \; \mbox{weakly in} \;  \xLtwo(0, T), \\ 
& \int_\O P_\delta(\rho_\varepsilon, \x1_{\cB_\varepsilon}) \xdif y \rightarrow \int_\O P_\delta(\rho_\delta, \x1_{\cB_\delta}) \xdif y \; \mbox{weakly in} \; \xLn{\frac{\beta + 1}{\beta}}(0, T).
\end{aligned}
\end{equation*}
and 
\begin{equation*}
\displaystyle \int_\O \left(\rho_\varepsilon \frac{|u_\varepsilon|^2}{2} + P_\delta(\rho_\varepsilon, \x1_{\cB_\varepsilon}) \right) \xdif y \xrightarrow[]{\varepsilon \rightarrow 0} \int_\O \left(\rho_\delta \frac{|u_\delta|^2}{2} + P(\rho_\delta, \x1_{\cB_\delta}) \right) \xdif y. 
\end{equation*}
Due to the weak lower semicontinuity of the $\xLtwo$ norms, the convergence of $u_\varepsilon$ given in \eqref{eq:vnsh_dspt:compacity_results_ue} , the strong convergence of $\rho_\varepsilon$ given in \eqref{eq:vnsh_dspt:rhoe_stg_cvg}, the strong convergence of $\chi_\varepsilon$ given in \eqref{eq:vnsh_dspt:1Be_cvg}, we can pass to the limit as $\varepsilon \rightarrow 0$ in the other terms of the energy inequality \eqref{eq:appx_sol:energye_estimate} to establish the estimate \eqref{eq:appx_sol:energyd_estimate}. 

\section{Proof of the main result} \label{sec:proof_main_result}
Proposition~\ref{prop:appx_sol:existence_delta_level} establishes the existence of weak solutions $(\cB_\delta , \rho_\delta, u_\delta)$ to system \eqref{eq:appx_sol:pd_definition}-\eqref{eq:appx_sol:mud_nud_definition}. In this section, we analyse the convergence of these solutions as $\delta \rightarrow 0$ in order to prove Theorem~\ref{theo:intro:main_existence_result}. We consider initial data $\rho_0$, $q_0$ satisfying \eqref{eq:intro:rho0_regularity}-\eqref{eq:intro:u0_compatibility}. Following the same ideas as in \cite[Section 4.0]{Feireisl2001JMF}, \cite[Section 7.10.7]{Novotny2004OUP}, we construct  $\{ \rho_{0, \delta} \} \in \xLn{\beta}(\O)$ and $\{q_{0, \delta}\} \in \xLn{\frac{2\beta}{\beta + 1}}(\O)$ such that  
\begin{equation*}
\rho_{0, \delta} = \rho_{b, 0} \; \mbox{in} \; \cB_0, \quad \rho_{0, \delta} \rightarrow \rho_0 \; \mbox{in} \; \xLn{\gamma_f}(\O), \quad q_{0, \delta} \rightarrow q_0 \; \mbox{in} \; \xLn{\frac{2\gamma_f}{\gamma_f + 1}}(\O), 
\end{equation*} 
and 
\begin{equation*}
E_\delta (\rho_{0, \delta}, q_{0, \delta}, \x1_{\cB_0}) \rightarrow E(\rho_0, q_0, \x1_{\cB_0}). 
\end{equation*}
The energy estimate \eqref{eq:appx_sol:energyd_estimate} yields, up to a subsequence,
\begin{equation}
u_\delta \rightarrow u \; \mbox{weakly in} \; \xLtwo(0, T; \xHone(\O)). \label{eq:vnsh_dspt:compacity_results_ud}
\end{equation}
By Proposition~\ref{prop:smlrt-prop:weak_sequential_continuity}, up a to a subsequence and for all $1 \leq p < \infty$, 
\begin{equation}
\x1_{\cB_\delta} \rightarrow \x1_{\cB} \; \mbox{weakly--* in} \; \xLinfty((0, T) \times \O), \; \mbox{strongly in} \; \xCzero([0, T]; \xLn{p}_{\xloc}(\xR^3)),
\label{eq:vnsh_artificial_pressure:1Bd_cvg} 
\end{equation}
where the pair $(u, \x1_{\cB})$ satisfies 
\begin{equation}
\partial_t \x1_{\cB} + u\cdot \nabla \x1_{\cB} \; \mbox{in} \; \cD^\prime((0, T) \times \xR^3), \quad \cB(0) = \cB_0.
\end{equation}
Moreover, the sequences $\{u_\delta\}$ and $\{\eta_\delta\}$ satisfy the assumptions of Proposition~\ref{coro:smlrt_prop:stability_compatible_velocities}. It then follows that
\begin{equation}
\eta_\delta \rightarrow \eta \; \mbox{weakly in} \; \xHone([0, T]; \cM), \; \mbox{strongly in} \; \xCzero([0, T]; \cM), 
\label{eq:vnsh_artificial_pressure:etad_cvg}
\end{equation}
and, if we define $\cB$ by 
\begin{equation*}
\cB(t) = \eta[t](\cB_0) \quad \forall \, t \in [0, T], 
\end{equation*} 
then the velocity field is compatible with the system $\{\overline {\cB}, \eta\}$. \\
Let $R_b, x_b \in \xHone(0, T)$ such that $\cB = \cB(x_b, R_b)$, we have $$R_b(t) \geq \frac{R_0}{2} \quad \forall \, t \in [0, T].$$  Moreover, from the strong convergence of $\eta_\delta$ in $\xCzero([0, T]; \cM)$, we deduce that for $\delta > 0$ sufficiently small, 
\begin{equation*}
\cB(t) \subset \cB_\delta^{\frac{\sigma}{2}}(t) \quad \forall\,  t \, \in [0, T]. 
\end{equation*} 
Since $\xdist (\cB_\delta (t), \partial \O) \geq 3\sigma/2$ in $[0, T]$, it follows that  
\begin{equation*}
\xdist(\cB(t), \partial \O) \geq \sigma \quad \forall \, t \in [0, T]. 
\end{equation*}

The energy estimate \eqref{eq:appx_sol:energye_estimate} also yields uniform bounds for $\x1_{\cB_\delta} \rho_\delta$ in $\xLinfty(0, T; \xLn{\gamma_f}(\O))$ and $(1-\x1_{\cB_\delta}) \rho_\delta$ in $\xLinfty(0, T; \xLn{\gamma_b}(\O))$. Since $\rho_\delta$ is constant in $\cB_\delta$, the uniform bound in $\xLinfty(0, T; \xLn{\gamma_b}(\O))$ implies a uniform bound in $\xLinfty((0, T) \times \O)$ and in particular in $\xLinfty(0, T; \xLn{\gamma_f}(\O))$. We deduce that, up to a subsequence, 
\begin{align}
\rho_\delta \rightarrow  \rho  \; \mbox{weakly--* in} \; \xLinfty(0, T; \xLn{\gamma_f}(\O)), \label{eq:vnsh_dspt:compacity_results_rhod_f}.  
\end{align}
Following the same arguments as in \cite[Section 4.2]{Feireisl2001JMF}, \cite[Section 7.10.1]{Novotny2004OUP}, we also deduce from \eqref{eq:appx_sol:energyd_estimate}, 
\begin{equation*}
\rho_\delta \rightarrow \rho \; \mbox {in} \; \xCzero([0, T]; L_{\xweak}^{\gamma_f}(\O)), \quad \rho_\delta u_\delta \rightarrow \rho u \; \mbox {weakly--* in} \;  \xLinfty(0, T; \xLn{\frac{2 \gamma_f}{\gamma_f + 1}}(\O)),
\end{equation*}
and $(\rho, u)$ solve the continuity equation in $\cD^\prime((0, T) \times \xR^3)$. \\ 

Similarly to the previous sections, we now focus on the convergence of the momentum equation \eqref{eq:appx_sol:rhodud_momentum_equation}. Let $\varphi \in {\cT} (Q_b)$. As in Section~\ref{subsec:cvg_high_penalization_term_momentum_equation} and Section~\ref{subsec:vnsh_dspt_continuity_equation}, since $\overline {Q_b}$ is a compact subset of $(0, T) \times \O$, there exists $\sigma_0 > 0$ and $\varphi_b \in \xCzero([0, T]; \cS)$ such that, for all $ t \in [0, T]$, 
\begin{equation*}
\varphi(t, \cdot) = \varphi_b(t, \cdot) \; \mbox{in} \; \cB^{\sigma_0}(t)
\end{equation*} 
where for all $\sigma > 0$,  $\cB^\sigma = \cB(x_b, R_b + \sigma)$. In particular, for $0 < \sigma \leq \sigma_0$, we have the decomposition:
\begin{equation}
\varphi = (1 - \x1_{\cB^\sigma}) \varphi + \x1_{\cB^\sigma} \varphi_b.
\label{eq:vnsh_artificial_pressure:varphi_decomposition}
\end{equation}
Using once more the strong convergence of $\eta_\delta$ given in \eqref{eq:vnsh_artificial_pressure:etad_cvg},  we deduce that there exists $\delta_0 > 0$ such that for all $0 < \delta \leq \delta_0$, 
\begin{equation*}
\cB_\delta(t) \subset \cB^{\sigma_0}(t) \quad  \forall \, t \in [0, T],  
\end{equation*}
and
\begin{equation*}
\x1_{\cB_\delta} \x1_{\cB^{\sigma_0}} = \x1_{\cB_\delta}, \quad (1 - \x1_{\cB_\delta}) (1 - \x1_{\cB^{\sigma_0}}) = (1 - \x1_{\cB^{\sigma_0}}), \quad \x1_{\cB_\delta}  (1 - \x1_{\cB^{\sigma_0}}) = 0.
\end{equation*}
Similarly to the previous sections, we first focus on the convergence of the nonlinear convective term
\begin{equation}
\int_0^T \int_\O \rho_\delta u_\delta \otimes u_\delta \xdif y \xdif \tau.
\label{eq:vnsh_artificial_pressure:cvct_cvg}
\end{equation}
When details are omitted, the arguments for the convergence of \eqref{eq:vnsh_artificial_pressure:cvct_cvg} are analogous to those presented in Section~\ref{subsec:cvg_high_penalization_term_momentum_equation}. For $0 < \sigma \leq \sigma_0$, we write
\begin{equation*}
\int_0^T \int_\O \rho_\delta u_\delta \otimes u_\delta \xdif y \xdif \tau - \int_0^T \int_\O \rho u \otimes u \xdif y \xdif \tau = \sum_{i = 1}^4 M_i(\sigma, \delta)
\end{equation*} 
where
\begin{align*}
& M_1(\sigma, \delta) =  \int_0^T \int_\O (1 - \x1_{\cB^\sigma}) \rho_\delta u_\delta \otimes u_\delta : \xD (\varphi) \xdif y \xdif \tau  - \int_0^T \int_\O (1 - \x1_{\cB^\sigma}) \rho u \otimes u : \xD (\varphi) \xdif y \xdif \tau, \\ 
& M_2(\sigma, \delta) = \int_0^T \int_\O (\x1_{\cB^\sigma} - \x1_{\cB}) \rho_\delta u_\delta \otimes u_\delta : \xD (\varphi) \xdif y \xdif \tau + \int_0^T \int_\O (\x1_{\cB^\sigma} - \x1_{\cB}) \rho u \otimes u : \xD (\varphi)  \xdif y \xdif \tau, \\
& M_3(\sigma, \delta) = \int_0^T \int_\O (\x1_{\cB} - \x1_{\cB_\delta}) \rho_\delta u_\delta \otimes u_\delta : \xD (\varphi)  \xdif y \xdif \tau, \\ 
& M_4(\sigma, \delta) = \int_0^T \int_\O \x1_{\cB_\delta} \rho_\delta u_\delta \otimes u_\delta : \xD (\varphi)  \xdif y \xdif \tau - \int_0^T \int_\O \x1_{\cB} \rho u \otimes u : \xD (\varphi) \xdif y \xdif \tau.
\end{align*}
Let $\theta > 0$, there exists $\sigma_1, \delta_1 > 0$ such that for $0 < \sigma \leq \min\{\sigma_0, \sigma_1\}$ and $0 < \delta \leq \min\{\delta_0, \delta_1\}$, 
\begin{equation*}
|M_2(\sigma, \delta)|\leq \theta/4, \quad |M_3(\sigma, \delta)|\leq \theta/4, 
\end{equation*} 
For $M_4(\sigma, \delta)$, we can prove once again an analogue of Lemma~\ref{lem:high_pen:VoL_convergence}, that is for $i \in \{1, 2, 3\}$, 
\begin{align}
& (\rho_\delta u_\delta, \x1_{\cB_\delta} e_i)_\O \xrightarrow[]{\delta \rightarrow 0} (\rho u,  \x1_{\cB} e_i)_\O \; \mbox{in} \;\xCzero([0, T]),\label{eq:vnsh_artificial_pressure:strong_cvg_rhodud_displacement}, \\ 
& (\rho_\delta u_\delta, \x1_{\cB_\delta} e_i \times (x - x_\delta))_\O \xrightarrow[]{\delta \rightarrow 0} (\rho u, \x1_{\cB} e_i \times (x - x_b))_\O \; \mbox{in} \; \xCzero([0, T]),  \label{eq:vnsh_artificial_pressure:strong-cvg-rotation}, \\ 
& (\rho_\delta u_\delta, \x1_{\cB_\delta} (x - x_\delta))_\O \xrightarrow[]{\delta \rightarrow 0} (\rho u, \x1_{\cB}  (x - x_b))_\O \; \mbox{in} \;\xCzero([0, T]) \label{eq:vnsh_artificial_pressure:strong_cvg_rhodud_dilatation}. 
\end{align} 
The proof of \eqref{eq:vnsh_artificial_pressure:strong_cvg_rhodud_displacement}-\eqref{eq:vnsh_artificial_pressure:strong_cvg_rhodud_dilatation} follows the same arguments as the proof of Lemma~\ref{lem:high_pen:tdV_tdo_tdL_convergence} with the right-hand side of the inequality \eqref{eq:high_pen:VoL_estimate} replaced with
\begin{multline*}
\int_{t_1}^{t_2} \|\rho_\delta u_\delta \otimes u_\delta\|_{\xLn{\frac{6\gamma_f}{4\gamma_f + 3}}(\O)} \|\nabla \varphi_{\delta, i}^\nu \|_{\xLn{\frac{6 \gamma_f}{2\gamma_f - 3}}(\O)} + \left(\frac43 \mu_f + \nu_f\right) \|\nabla u_\delta \|_{\xLtwo(\O)} \|\nabla \varphi\|_{\xLtwo(\O)} \\ 
+ \nu_b \|\nabla u_\delta \|_{\xLtwo(\O)} \|\nabla \varphi\|_{\xLtwo(\O)} + \|\rho_\delta\|_{\xLn{\gamma_f + \theta}(\O)}^{\gamma_f} \|\nabla \varphi_{\delta, i}^\nu \|_{\xLn{\frac{\gamma_f + \theta}{\theta}}(\O)} + \delta \|\rho_\delta\|_{\xLn{\beta + \theta}(\O)}^\beta \|\nabla \varphi_{\delta, i}^\nu \|_{\xLn{\frac{\beta + \theta}{\theta}(\O)}} \\ 
+ \|\rho_\delta\|_{0, \beta + \theta} \|g\|_{\xLinfty(\O)} \|\varphi_{\delta, i}^\nu \|_{L^\frac{\beta + \theta}{\theta}(\O)} + \frac{2\kappa_b}{R_{b, 0}} \|\nabla \varphi_{\delta, i}^\nu \|_{\xLone(\O)} \xdif \tau
\end{multline*} 
From the convergences \eqref{eq:vnsh_artificial_pressure:strong_cvg_rhodud_displacement}-\eqref{eq:vnsh_artificial_pressure:strong_cvg_rhodud_dilatation}, we obtain
\begin{equation*}
\lim_{\delta \rightarrow 0} \int_0^T \int_\O \xdiv (\varphi_b) \x1_{\cB_\delta} \rho_\delta |u_\delta|^2 \xdif y \xdif \tau \rightarrow \int_0^T \int_\O \xdiv (\varphi_b) \x1_{\cB} \rho |u|^2 \xdif y \xdif \tau, 
\end{equation*}
and consequently, there exists $\delta_2 > 0$ such that for all $0 < \delta \leq \delta_2$, 
\begin{equation*}
|M_4(\sigma, \delta)| \leq \theta/4. 
\end{equation*}
Finally, setting $\sigma_m = \{\sigma_0, \sigma_1\}$ and following similar arguments to the one presented in \cite[Section 8]{Feireisl2003ARM} and already presented in Section~\ref{subsec:cvg_high_penalization_term_momentum_equation}-Section~\ref{subsec:vnsh_dspt_continuity_equation}, we get
\begin{equation*}
\rho_\delta u_\delta \otimes u_\delta \rightarrow \rho u \otimes u \; \mbox{in} \; \xLtwo(Q_f^{\sigma_m})
\end{equation*}
where 
\begin{equation*}
Q_f^{\sigma_m} = (0, T) \times \O \setminus \overline {Q_b^{\sigma_m}}, \quad \overline {Q_b^{\sigma_m}} = \{(t, x) \in (0, T) \times \O \, | \, x \in \overline {\cB^{\sigma_m}(t)} \}.  
\end{equation*} 

We now aim to identify the pressure term. The local pressure estimates obtained in Lemma~\ref{lemma:vnsh_dspt:pressure_estimate} can be replaced by 
\begin{equation}
\|\rho_\delta\|_{L^{\gamma_f + \theta}(\cK_f)}^{\gamma_f + \theta} + \delta \|\rho_\delta\|_{L^{\beta + \theta}(\cK_f)}^{\beta + \theta} \leq c \left(\cK_f, E_{\delta}(\rho_{0, \delta}, q_{0, \delta}, \x1_{\cB_0}), \|g\|_{\xLinfty((0, T) \times \O)} \right)
\label{eq:vnsh_artificial_pressure:pressure-a-priori-estimates}
\end{equation}
for any compact $\cK_f \subset Q_f$, where $\theta > 0$ is a positive constant independent of $\delta$. The proof of \eqref{eq:vnsh_artificial_pressure:pressure-a-priori-estimates} follows the same arguments as in \cite[Equation (9.5)]{Feireisl2003ARM}. In particular, we can apply \cite[Proposition 2.3]{Feireisl2002MMA}, \cite[Lemma 4.1]{Feireisl2001JMF}, \cite[Lemma 7.52]{Novotny2004OUP} to obtain 
\begin{align*}
a_f \rho_\delta^{\gamma_f}  \rightarrow a_f \overline {\rho^{\gamma_f}} \; \mbox {weakly in} \; L^\frac{\gamma_f + \theta}{\gamma_f} (\cK_f) \; \mbox{for any compact} \; \cK_f \subset Q_f, 
\end{align*}
where the limit functions satisfy the identity 
\begin{equation*}
\int_0^T \int_{\xR^3} (\rho u) \cdot \partial_\tau \varphi + \left(\rho u \otimes u\right): \xD (\varphi) + a_f \overline {\rho^{\gamma_f}} \xdiv (\varphi) \xdif y \xdif \tau = \int_0^T \int_{\xR^3} \xT_f (u) : \xD (\varphi) -  \overline {\rho g} \cdot \varphi \mathrm{d} y \xdif \tau
\end{equation*} 
for all test functions $\varphi \in {\cD}(Q_f)$. Similarly to Section~\ref{subsec:vnsh_dspt_continuity_equation}, the proof will be completed by showing the strong convergence of the density. However, unlike  Section~\ref{subsec:vnsh_dspt_continuity_equation}, we do not know a priori that the renormalized equation \eqref{eq:intro:rho_renormalized_continuity_equation} holds, as we do not have $\rho \in \xLtwo(Q_T)$. This difficulty dissapears if $\gamma \geq 9/5$ as $L^{\gamma + \theta(\gamma)}(Q_T) \subset \xLtwo(Q_T)$. To prove the strong convergence of the density, we proceed as in \cite[Equation (9.9)]{Feireisl2003ARM}, employing \cite[Proposition 4.3]{Feireisl2002ESA}, \cite[Lemma 4.2]{Feireisl2001JMF}, \cite[Lemma 7.55]{Novotny2004OUP} to derive an analogue of Lemma~\ref{lemma:vnsh_dspt:cvg_effective_viscous_flux}, namely, 
\begin{multline}
\lim \limits_{\delta \rightarrow 0} \int_0^T \int_{\xR^3} \varphi \left(a_f \rho_\delta^{\gamma_f} - \left(\frac43 \mu_f + \nu_f\right) \xdiv (u_\delta)\right) T_k(\rho_\delta) \xdif y \xdif \tau = \\ 
\int_0^T \int_{\xR^3} \varphi \left(a_f \overline {\rho^{\gamma_f}} - \left(\frac43 \mu_f + \nu_f \right) \xdiv (u)\right) \overline {T_k(\rho)} \xdif y \xdif \tau
\label{eq:vnsh_artificial_pressure:limit_viscous_effective_flux} 
\end{multline} 
for any $\varphi \in \cD (Q_f)$, where $T_k(\rho) = \min \{\rho, k\}$ are cut-off operators, and $\overline {T_k(\rho)}$ stands for the weak limit of $T_k(\rho_\delta)$.
As in Section~\ref{subsec:vnsh_dspt_continuity_equation}, the relation \eqref{eq:vnsh_artificial_pressure:limit_viscous_effective_flux} implies
\begin{equation}
\overline{T_k(\rho) \xdiv (u)} \geq \overline {T_k(\rho)} \xdiv (u) \; \mbox{on} \; (0, T) \times \xR^3, 
\label{eq:vnsh_artificial_pressure:tkr_dvu_inequality}
\end{equation} 
and proceeding as in \cite[Proposition 6.1]{Feireisl2001CMU}, we can prove that for any compact set $\cK_f \subset Q_f $ such that $|\cK_f| \leq M$, then  
\begin{equation*}
\sup_{k \geq 1} \left(\lim \limits_{\delta \rightarrow 0} \int_{\cK_f} |T_k(\rho_\delta) - T_k(\rho)|^{\gamma_f + 1} \xdif y \xdif \tau \right) \leq c \Big(M, \sup_{\delta > 0} \|\nabla u_\delta\|_{\xLtwo(\cK_f)} \Big).
\end{equation*} 
Let $\cK_b$ a compact subset of $Q_b$, by Lemma~\ref{lem:smlrt-prop:density_bubble}, we get $\rho_\delta \rightarrow \rho$ strongly in $\xLone(\cK_b)$. In particular, 
\begin{equation}
\sup_{k \geq 1} \left(\lim \limits_{\delta \rightarrow 0} \int_{\cK_b} |T_k(\rho_\delta) - T_k(\rho)|^{\gamma_f + 1} \xdif y \xdif \tau \right) = 0.
\label{eq:vnsh_artificial_pressure:oscillations_Kb}
\end{equation} 
Let $\cK \subset (0, T) \times \O$ compact such that  $|\cK| \leq M$. Let $\cU(k)$ a neighborhood of $\partial Q_b$ such that 
$|\cU(k)| \leq k^{-(\gamma_f + 1)}$, we have 
\begin{multline*}
\int_{\cK} |T_k(\rho_\delta) - T_k(\rho)|^{\gamma_f + 1} \xdif y \xdif \tau \\ 
\leq \int_{(\cK \setminus \cU(k)) \cap Q_f} |T_k(\rho_\delta) - T_k(\rho)|^{\gamma_f + 1} + \int_{\cU(k)} (2k)^{\gamma_f + 1}  + \int_{(\cK \setminus \cU(k)) \cap Q_b} |T_k(\rho_\delta) - T_k(\rho)|^{\gamma_f + 1}.
\end{multline*} 
Consequently
\begin{equation*}
\sup_{k \geq 1} \Big(\lim \limits_{\delta \rightarrow 0}  \int_{\cK} |T_k(\rho_\delta) - T_k(\rho)|^{\gamma_f + 1} \xdif y \xdif \tau \Big) \leq  c \Big(M, \sup_{\delta > 0} \|\nabla u_\delta\|_{\xLtwo(\cK_f)} \Big) + 2^{\gamma_f + 1}.
\end{equation*} 
We finally obtain
\begin{equation} 
\sup_{k \geq 1} \left(\lim \limits_{\delta \rightarrow 0} \int_{(0, T) \times \xR^3} |T_k(\rho_\delta) - T_k(\rho)|^{\gamma_f + 1} \xdif y \xdif \tau \right) < \infty.
\label{eq:vnsh_artificial_pressure:oscillations_OTO}
\end{equation} 
By virtue  \cite[Proposition 7.1]{Feireisl2001CMU}, \cite[Lemma 7.57]{Novotny2004OUP}, boundedness of the defection measure \eqref{eq:vnsh_artificial_pressure:oscillations_OTO} implies that $\rho$, $u$ satisfy the renormalized continuity equation \eqref{eq:intro:rho_renormalized_continuity_equation} where $b$ satisfies \eqref{eq:intro:b_cdt_renormalized_solution} with $-1 < \lambda_1 \leq \frac{s(\gamma_f)}{2} - 1$ and $s(\gamma_f) = \frac{5 \gamma_f - 3}{3}$. Based on this result, introducing  
\begin{equation*}
L_k(s) = 
\left\{
\begin{aligned}
& s \log(s),         & s \in [0, k), \\
& s \log(k) + s - k, & s \in [k, \infty), 
\end{aligned}
\right.
\end{equation*} 
we can prove that 
\begin{equation*}
\partial_t \left( \overline {L_k(\rho)} - L_k(\rho) \right) + \xdiv \left( (\overline {L_k(\rho)} - L_k(\rho)) u \right) + \left(\overline {T_k(\rho) \xdiv (u)} - T_k(\rho) \xdiv (u) \right) = 0
\end{equation*} 
which gives together with  \eqref{eq:vnsh_artificial_pressure:tkr_dvu_inequality}
\begin{equation*}
\sup_{k \geq 0} \left( \lim \limits_{\delta \rightarrow 0} \|T_k(\rho_\delta) - T_k(\rho)\|_{L^{\gamma_f + 1}((0, T) \times \xR^3)} \right) \leq 0.
\end{equation*}
This is enough to deduce that $\rho_\delta \rightarrow \rho$ in $\xCzero([0, T]; \xLone(\xR^3))$ for instance. This completes the proof of Theorem~\ref{theo:intro:main_existence_result}, cf \cite[Lemma 4.3, 4.4]{Feireisl2001CMU},\cite[Section 8]{Feireisl2002MMA}, \cite[Section 4.5, 4.6]{Feireisl2001JMF}, \cite[Lemma 7.60]{Novotny2004OUP}. \\ 

\noindent \textbf{Acknowledgment}\\ 
The authors acknowledge the support of ANR project MSM$\Phi$ ANR-23-CE40-0013-01. The authors would like to thank Hélène Mathis and Nicolas Seguin for fruitful discussions when preparing the paper. 

\bibliographystyle{plain}
\bibliography{Existence_Bubbly_Flows_HAL_SUBMISSION}
\end{document}